\documentclass[12pt]{amsart}

\usepackage{amsmath,amstext,amssymb,amsopn,amsthm}
\usepackage{url,verbatim,hyperref}
\usepackage{mathtools}
\usepackage{enumerate}

\usepackage{xcolor,graphicx}

\usepackage[margin=30mm]{geometry}
\usepackage{eucal,mathrsfs,dsfont}%gives nicer set-names. 
\usepackage[section]{placeins}
\usepackage{caption}
\allowdisplaybreaks

\newtheorem{theorem}{Theorem}[section]
\newtheorem{corollary}[theorem]{Corollary}
\newtheorem{lemma}[theorem]{Lemma}
\newtheorem{proposition}[theorem]{Proposition}

\theoremstyle{definition}

\newtheorem{definition}[theorem]{Definition}
\newtheorem{assumption}[theorem]{Assumption}

\newtheorem{remark}[theorem]{Remark}
\newtheorem{example}[theorem]{Example}

\theoremstyle{remark}

\numberwithin{equation}{section}

\newcommand{\eps}{\varepsilon}

\newcommand{\calL}{\mathcal{L}}
\newcommand{\calF}{\mathcal{F}}

\newcommand{\calA}{\mathcal{A}}
\newcommand{\calD}{\mathcal{D}}

\newcommand{\calT}{\mathcal{T}}

\newcommand{\calX}{\mathcal{X}}

\newcommand{\calC}{\mathcal{C}}

\newcommand{\R}{\mathds{R}}

\newcommand{\bV}{\mathbf{V}}

\newcommand{\Z}{{\mathbb Z}}

\newcommand{\prt}{\partial}

\newcommand{\wh}{\widehat}
\newcommand{\wt}{\widetilde}

\DeclareMathOperator{\Leb}{Leb}

\def\bone{{\bf 1}}

\def\bv{{\bf v}}

\newcommand{\red}[1]{{\color{red} #1}}

\title[Transport equations   with freezing]{Coupled transport equations with freezing}
\author{Krzysztof Burdzy and John Sylvester}

\address{Department of Mathematics, Box 354350, University of Washington, Seattle, WA 98195}
\email{burdzy@uw.edu}
\email{sylvest@uw.edu}

\thanks{Research supported in part by Simons Foundation Grants 506732 and 928958. }

\begin{document}

\begin{abstract}
We study a system of two coupled transport equations with freezing.  The solutions freeze in time when they are equal. We prove existence and uniqueness of continuous solutions if the initial conditions are continuous. We discuss several qualitative and quantitative properties of the solutions. The equations arise in a model for collisions of a large number of tightly spaced balls.
\end{abstract}

\maketitle

\section{Introduction}\label{intro}

We will examine a system of nonlinear partial differential
equations \eqref{a18.2}-\eqref{a18.3} below, subject to the  constraint \eqref{a18.4}.
\begin{align}\label{a18.2}
v_t &= - v_x
\bone_{\{v-w >0\}}, \\
w_t &= \label{a18.3}
 w_x
\bone_{\{v-w >0\}},\\
v&-w  \geq 0 ,\label{a18.4}
\end{align}
for $v(x,t)$ and $w(x,t)$, $x\in \R$, $t\geq 0$. 
Our main results include existence of continuous solutions
for all continuous initial  conditions satisfying (\ref{a18.4}),
and uniqueness under natural extra assumptions. We will also present several  qualitative and quantitative results on behavior of the solutions. In particular, we will provide information on shapes of  ``liquid'' and ``frozen'' zones, and behavior of characteristics.

The equations
 \eqref{a18.2}-\eqref{a18.4} are equivalent to equations \eqref{n10.1}-\eqref{n10.3} which arose in analysis of ``pinned balls'' on a line. The general model of ``pinned balls'' was first introduced  in \cite{fold}.
Evidence was presented in \cite{KBJHSS} that
\eqref{n10.1}-\eqref{n10.3}  provide a reasonably
accurate description of  the evolution of parameters of a stochastic model for a large family of pinned balls on a line. The first steps towards solving this system of PDEs were taken in \cite{KBAO}.

Besides the physical motivation, we believe that the
system provides an interesting example of a nonlinear
system of first order PDE's where characteristics
intersect, but solutions remain continuous. Solutions to
\eqref{a18.2}-\eqref{a18.4} with generic initial
conditions are typically nondifferentiable along some
curves but they never develop the shocks that typically occur in
such systems.

\subsection{Pinned balls}\label{sec:pinn}

We will briefly describe the model that motivated the PDEs studied in this paper. 
In a system of pinned billiard balls, introduced in
\cite{fold}, the  balls touch some other balls and have
pseudo-velocities but they do not move. The balls
``collide,'' and their pseudo-velocities change according to the usual laws of totally elastic collisions.

Consider a finite number of balls  arranged on a segment of the real line, with centers  one unit apart and their radii  all equal to $1/2$. 
The spacetime for the model is discrete, i.e.,
the velocities $\bv(x,t)$ are defined for $x=1,2,\dots, n$ and $t=0,1,2, \dots$, where $x$ is the position (i.e., number) of the $x$-th ball.
The evolution, i.e., pseudo-collisions of the balls and transformations of the velocities, is driven by an exogenous random process because the balls do not move and hence they cannot collide in the usual way.

First consider a simplified model in which pairs of adjacent balls are chosen randomly, i.e., in a uniform way, and form an i.i.d. sequence. Every time a pair of adjacent balls is chosen, the velocities become ordered, i.e., if the chosen balls have labels $x$ and $x+1$ and the collision occurs at time $t$ then
\begin{align}\label{a27.1}
\bv(x,t+1) &= \min (\bv(x,t), \bv(x+1, t)), \\
\bv(x+1,t+1) &= \max (\bv(x,t), \bv(x+1, t)).\label{a27.2}
\end{align}
This agrees with the usual transformation rule for velocities of moving balls of equal masses undergoing totally elastic collisions.
The evolution described above  has been studied under the names of
``random sorting networks'' in \cite{AHRV},
``oriented swap process'' in \cite{AHR} and ``TASEP speed process'' in \cite{AAV}. It has been also called ``colored TASEP.''
For a related model featuring confined (but moving) balls, see \cite{Gaspard_2,Gaspard_2008_1,Gaspard_2008}.

In the  model studied in \cite{KBJHSS} the evolution of velocities of pinned balls consists of a sequence of two-step transformations. 
In the first step,  energy is redistributed. In the second step   a pair of velocities is reordered.
It is argued in \cite{KBJHSS} that
the joint distribution of $\{\bv(x,t), 1\leq x\leq n, t\geq 0\}$ converges after appropriate rescaling to
\begin{align*}%\label{a28.1}
\bv(x,t) = \mu(x,t) + \sigma(x,t) W(x,t),
\end{align*}
when the number $n$ of balls goes to infinity. Here
$W(x,t)$ is spacetime white noise
and $\mu(x,t)$ and $\sigma(x,t)$ are deterministic functions satisfying \eqref{n10.1}-\eqref{n10.3} below.
See \cite{KBJHSS}, especially Remark 4.4, for the discussion of how \eqref{n10.1}-\eqref{n10.3} arise in the scaling limit.

The ``pinned balls'' model is a variant of the ``hot rods'' model introduced in \cite{DF,BDS}. In that model, the balls were allowed to move.
For recent results on the hot rods model and a review of the related literature, see \cite{FerEt,FerOl}.
Our partial differential equations  for $\mu$ and $\sigma$ are essentially the same as the equations (2.1) and (2.2)  for the local density  and the local current in \cite{Bert}. 

\subsection{Organization of the paper}

Our main theorems are stated in Section \ref{sec:main}. 
The same section contains also a review of properties of solutions proved later in the paper, an outline of the main ideas of proofs, and a brief review of related results in the PDE theory.   Sections \ref{sec:char}-\ref{sec:ann} review properties of solutions at the heuristic level, to help the reader follow rather technical proofs. The three sections are devoted to characteristics, freezing and thawing boundaries, and annihilation of sublevel  and superlevel sets of solutions. Section \ref{sec:cons} contains the construction of the solutions for ``nice'' initial conditions. This is generalized to arbitrary continuous initial conditions in Section \ref{sec:gen}.
Uniqueness of solutions is proved, under extra assumptions, in Section \ref{sec:uniq}. Section \ref{sec:prop} presents several qualitative and quantitative properties of the solutions. Finally, Section \ref{sec:exam} is devoted to examples.

\section{Main results}\label{sec:main}

The model discussed in Section \ref{sec:pinn} is associated with coupled nonlinear partial differential equations
for $\mu(x,t)$ and $\sigma(x,t)$ 
with $(x,t)\in [a_1,a_2]\times [0,\infty)$ for some $-\infty<a_1<a_2<\infty$. However, from the mathematical point of view, it is equally natural to state and solve the problem
with the space variable taking values in the whole real
line. From now on, we will take $I = \R$ or $I=[a_1,a_2]$
and look for solutions to the following constrained
boundary value problem  in $ I\times [0,\infty)$. 
\begin{align}\label{n10.1}
\sigma_t(x,t) & = - \mu_x(x,t) \bone_{ \sigma(x,t) >0},\\
\mu_t(x,t) &= \label{n10.2}
 - \sigma_x(x,t) \bone_{ \sigma(x,t) >0},\\
\sigma(x,t) & \geq 0 ,\qquad \text{  for  }  x\in I,\  t\geq 0,\label{n10.7}\\
\sigma(a_1,t) &= \sigma(a_2,t) = 0, \qquad \text{  for  } t\geq 0,
\text{  if  } I=[a_1,a_2].\label{n10.3}
\end{align}

\begin{remark}\label{j25.1}
  Suppose that $\sigma(x,0)=0$ for all $x$ but $\mu(x,0)$
  is totally arbitrary (for example, discontinuous or even
  non-measurable). If we let $\sigma(x,t)=0$ and
  $\mu(x,t)=\mu(x,0)$ for all $x$ and $t\geq0$ then thus
  constructed $\sigma(x,t)$ and $\mu(x,t)$ solve
  \eqref{n10.1}-\eqref{n10.3} in $ I\times
  [0,\infty)$. This is a large family of solutions
  that is inconsistent with the pinned balls model.

  For this reason, we elaborate further on the meaning of
  the constraint \eqref{n10.7}, insisting that solutions
  freeze only when they must do so to avoid violating
  the constraint \eqref{n10.7}, and they thaw as soon as they can do so
  without violating \eqref{n10.7}. Here, ``freezing''
  means transitioning from ``liquid'', where \(\sigma\)
  and \(\mu\) satisfy
 
\begin{align}\label{eq:1}
  \sigma_t(x,t) & = - \mu_x(x,t) ,\\
  \label{eq:2}
\mu_t(x,t) &= - \sigma_x(x,t),
\end{align}
 to ``frozen'', where  \(\sigma\)
  and \(\mu\) satisfy
\begin{align*}
\sigma_t(x,t) & = 0,\\
\mu_t(x,t) &= 0,
\end{align*}
\noindent
and ``thawing'' means transitioning from frozen to liquid.

It is easy to check that for the initial data
$\sigma(x,0)=0$ and $\mu(x,0)$ strictly decreasing, as described
above,  \(\sigma\)
  and \(\mu\) may satisfy \eqref{eq:1}-\eqref{eq:2} for \(t>0\)
  without violalting \eqref{n10.7}. Because
  \(-\mu_{x}(x,0)\) is positive, solutions with
  \(\sigma(x,0) =0\)
  instantly thaw, i.e. \(\sigma(x,t)>0\) for all \(t>0\).
  A more precise statement is the following:

  If for some $s\geq 0$ and $x_1< x_2$, there
  exist functions $\mu(x,t)$ and $\sigma(x,t)$ which are
  continuous in the triangle
  $\{(x,t): t\geq s, x_1 +(t-s) < x < x_2-(t-s)\}$,
  satisfy \eqref{n10.1}-\eqref{n10.3} and $\sigma(x,t) >0$
  in the interior of the triangle then we will consider
  only solutions to \eqref{n10.1}-\eqref{n10.3} satisfying
  these conditions.

It will become clear later in the paper that the generic situation covered by the above condition is when 
$x\to \mu(x,s)$ is strictly decreasing and $\sigma(x,s)=0$ for $x\in(x_1,x_2)$.

The above restriction is crucial for the uniqueness of solutions.

The extra condition agrees with the motivating ``pinned balls'' model because the evolution of pseudo-velocities cannot reach the stationary regime until the pseudo-velocities are (approximately)  increasing.
\end{remark}

It is natural to diagonalize equations
\eqref{n10.1}-\eqref{n10.3} by introducing

\begin{align}\label{a22.2}
v(x,t) &= \frac 1 2 (\mu(x,t) + \sigma(x,t)),\\
w(x,t) &= \frac 1 2 (\mu(x,t) - \sigma(x,t)).\label{a22.3}
\end{align}
Then \eqref{n10.1}-\eqref{n10.3} are equivalent to
\begin{align}\label{n10.4}
v_t(x,t) &= - v_x(x,t)
\bone_{v(x,t)-w(x,t) >0}, \\
w_t(x,t) &= \label{n10.5}
 w_x(x,t)
\bone_{v(x,t)-w(x,t) >0},\\
v(x,t)&\geq w(x,t) ,\qquad \text{  for  }  x\in I,\  t\geq 0,\label{n10.8}\\
 v(x,t)&=w(x,t), \qquad \text{  for  } x=a_1,a_2,\  t\geq 0,
\text{  if  } I=[a_1,a_2].\label{n10.6}
\end{align}

\begin{remark}\label{m28.1}
(i)
We will interpret \eqref{n10.4}-\eqref{n10.5} in the following weak sense.
We rewrite the equations in the form of directional derivatives,
\begin{align*}
(\prt_t+\prt_x)v(x,t)=0,
\qquad
(\prt_t-\prt_x)w(x,t)=0.
\end{align*}
We consider the first  condition to be satisfied at $(x,t)$ if $v(x+\delta, t+\delta) = v(x,t)$ for some $\eps>0$ and all $\delta\in(-\eps,\eps)$. A similar remark pertains to the second condition. 
More precisely, if for an open rectangle $F$ in the space-time, $v(x,t)> w(x,t)$ for $(x,t)\in F$ then all characteristics of $v$ have slope 1  and all characteristics of $w$ have slope $-1$ in $F$, i.e., $v(x+\delta, t+\delta) = v(x,t)$  for all $(x,t),(x+\delta, t+\delta) \in F$
and $w(x-\delta, t+\delta) = w(x,t)$ for all $(x,t),(x-\delta, t+\delta)\in F$.
If for an open rectangle $F$ in the space-time, $v(x,t)= w(x,t)$ for $(x,t)\in F$ then all characteristics of $v$ and $w$ have slope 0 in $F$, i.e., $v(x, t+\delta) = v(x,t)$ and $w(x, t+\delta) = w(x,t)$ for all $(x,t),(x, t+\delta)\in F$. 

The derivatives $\prt _t$ and $\prt _x$ need not exist for \eqref{n10.4}-\eqref{n10.5} to be satisfied. From now on, whenever we assert existence of solutions to \eqref{n10.4}-\eqref{n10.5}, it will be in the sense of this remark.

(ii) We translate the condition stated in Remark \ref{j25.1} into the present context.

\begin{align}\label{a5.4}
&\text{
If for some  $s\geq 0$ and $x_1< x_2$ there exist functions $\tilde{v}(x,t)$ and $\tilde{w}(x,t)$ }\\
&\text{
which are 
continuous in the triangle } \notag \\
&\quad U:=\{(x,t): t\geq s, x_1 +(t-s) < x < x_2-(t-s)\}, \notag \\
 &\text{ satisfy \eqref{n10.4}-\eqref{n10.6} with
   $\tilde{v}(x,t) >\tilde{w}(x,t)$ in the interior of
   $U$, and}  \notag \\
  &\quad \tilde{v}(x,s) =  v(x,s)\ \text{  and  }\ \tilde{w}(x,s) = w(x,s)
  \ \text{  for  } x_1 < x < x_2,
    \notag \\
 &\text{ 
then  $v(x,t)=\tilde{v}(x,t)$ and $w(x,t)=\tilde{w}(x,t)$  in the interior of $U$.  } \notag 
\end{align}

The generic situation covered by \eqref{a5.4} is when 
$x\to v(x,s)$ is strictly decreasing and $v(x,s)=w(x,s)$ for $x\in(x_1,x_2)$.

(iii) Condition \eqref{n10.6} can be interpreted as  reflection of a wave at the boundary of a container.
\end{remark}

Our main results assert existence and uniqueness of solutions to \eqref{n10.4}-\eqref{a5.4}. We need different assumptions for the two claims. We start with the existence result.

\begin{theorem}\label{m28.4}
Assume that the initial conditions $v(x,0)$ and $w(x,0)$ are continuous on $I$ and $v(x,0)\geq w(x,0)$ for all $x\in I$. If $I=[a_1,a_2]$ then assume that $v(a_1,0)=w(a_1,0)$ and $v(a_2,0)=w(a_2,0)$.
Then there exist jointly continuous functions $v(x,t)$ and $w(x,t)$ with $x\in I$, $t\geq 0$, such that \eqref{n10.8}-\eqref{a5.4} are satisfied, and for every open set $D\subset I \times(0,\infty)$, if $v(x,t)> w(x,t)$ for all $(x,t) \in D$ or $v(x,t)= w(x,t)$ for all $(x,t) \in D$ 
then \eqref{n10.4}-\eqref{n10.5} are satisfied in $D$.
\end{theorem}

\begin{remark}
(i)
Solutions to \eqref{n10.4}-\eqref{a5.4} need not be differentiable in the classical sense even if the initial conditions are smooth. For generic smooth initial conditions, $v$ and $w$ are smooth in most of the space-time but not everywhere. 
Specifically, the solutions need not be differentiable on freezing and thawing boundaries---see Section \ref{sec:char}.

(ii)
Solutions to \eqref{n10.4}-\eqref{a5.4} can be defined by explicit formulas \eqref{j17.1a}-\eqref{j17.8a}. All we need to assume about the initial conditions $v(x,0)$ and $w(x,0)$ to apply these formulas is measurability. 
While completely arbitrary measurable initial conditions might not be interesting, it may be interesting to investigate solutions to \eqref{n10.4}-\eqref{a5.4} with initial conditions which have isolated jumps. We leave this case as a potential future project. 

(iii) If $v(x,t)> w(x,t)$ for all $(x,t) $ in an open set $D\subset I \times(0,\infty)$
then we will call $D$ a ``liquid zone.''  If $v(x,t)= w(x,t)$ for all $(x,t) \in D$ then we will call $D$ a ``frozen zone.''
\end{remark}

As we announced, our main uniqueness result needs assumptions stronger than the existence theorem. More accurately, we will prove uniqueness of solutions within a restricted class of functions. Before we state the uniqueness result, we will state some definitions.

    For a first order quasilinear intial value problem
    \begin{align}\label{eq:3}
      f_{t} + b(x,t,f)f_{x} &= c(x,t,f),\\
      f(x,0) &= f_{0}(x),\notag
    \end{align}
a characteristic curve \((\chi(t,x_{0},f_{0}(x_{0}),t)\) is defined as the solutions to
    the system of ODE's
    \begin{align}\label{eq:4}
      \frac{d\chi}{dt} &= b(\chi,t,z),
      \\
      \frac{dz}{dt} &= c(\chi,t,z),\label{eq:5}
      \end{align}
      with initial conditions
      \begin{align*}
      \chi(x_{0},0) = x_{0} \ \text{ and }\  z(x_{0},0) = f(x_{0},0).
    \end{align*}

  If every \((x,t)\) lies on a unique characteristic, i.e
  \((x,t) = (\chi(x_{0},t),t)\) for a unique \(x_{0}\), then we
  may define \(f(x,t) = z(x_{0},t)\) and  \(f\) will
  satisfy (\ref{eq:3}). In this way the initial value
  problem for the PDE is reduced to an initial value
  problem for a system of ODE's.

  For our system, (\ref{n10.4})-~(\ref{n10.5}), the \(v\)-characteristics, \((\chi_{v}(x_{0},t),t)\) are piecewise
  linear curves with \(d\chi/dt\) equal to zero in
  the frozen region and one in the liquid region. As
  \(c(x,t,z)=0\), \(v\) is constant along its
  characteristics.  Similarly, \(d\chi/dt\) is
  equal to zero or minus one for the
  \(w\)-characteristics, and \(w\) is constant along its
  characteristics.

  Our formal definition requires that a characteristic
  satisfy (\ref{eq:4}-\ref{eq:5} in the open liquid and frozen regions,
  and be continuous across the transitions.

\begin{definition}\label{m26.3} 

Suppose that $v(x,t)$ and $w(x,t)$ are defined for $x\in I$ and $s\le t$.\\

(i)
If $I = \R$ then
we will call a function
$\{\chi_v(u), 0\leq u \leq t\}$ with values in $I$
a characteristic of $v$ if for some $x\in I$, $\chi_v(0)=x$, 
$v(\chi_v(u),u) =v(x,0)$ for all $u\in[0,t]$, $\chi_v$ is
continuous, and \(\chi_{v}(u)\) satisfies (\ref{eq:3})
with $c\equiv 0$,  \(b=1\) in the interior of the liquid region \(\{(x,t):v(x,t)>w(x,t)\}\), and
 \(b=0\) in the interior of the frozen
region  \(\{(x,t):v(x,t)=w(x,t)\}\).\\

If $I = [a_1,a_2]$ then
we will call a function
$\{\chi_v(u), s\leq u \leq t\}$ with values in $I$
a  characteristic of $v$ if for some $x\in I$, $\chi_v(s)=x$, 
$v(\chi_v(u),u) =v(x,s)$ for all $u\in[s,t]$,  $\chi_v$ is continuous,
 $s=0$ or $\chi_v(s)=a_1$,
  and \(\chi_{v}(u)\) satisfies (\ref{eq:3})
with $c\equiv 0$,  \(b=1\) in the interior of the liquid region, and
 \(b=0\) in the interior of the frozen
region.\\

Similarly, if
 $I = \R$ then
we will call a function
$\{\chi_w(u), 0\leq u \leq t\}$ with values in $I$
a characteristic of $w$ if for some $x\in I$, $\chi_w(0)=x$, 
$v(\chi_w(u),u) =w(x,0)$ for all $u\in[0,t]$, $\chi_w$ is
continuous, and \(\chi_{w}(u)\) satisfies (\ref{eq:3})
with $c\equiv 0$,  \(b=1\) in the interior of the liquid region, and
 \(b=0\) in the interior of the frozen
region.\\

If $I = [a_1,a_2]$ then
we will call a function
$\{\chi_w(u), s\leq u \leq t\}$ with values in $I$
a  characteristic of $w$ if for some $x\in I$, $\chi_w(s)=x$, 
$v(\chi_w(u),u) =w(x,s)$ for all $u\in[s,t]$,  $\chi_w$ is continuous,
 $s=0$ or $\chi_w(s)=a_2$,
  and \(\chi_{w}(u)\) satisfies (\ref{eq:3})
with $c\equiv 0$,  \(b=-1\) in the interior of the liquid region, and
 \(b=0\) in the interior of the frozen
region.\\

Suppose that $\{\chi_v(u), s\leq u \leq t\}$ is a characteristic and $x_1=\chi_v(u_1)$ for some $u_1\in[s,t]$. Then we will call  $\{\chi_v(u), s\leq u \leq u_1\}$ a backward characteristic (emanating from $(x_1,u_1)$) and $\{\chi_v(u), u_1\leq u \leq t\}$ a forward characteristic (emanating from $(x_1,u_1)$). Similar terms will be applied to $w$-characteristics.

By abuse of
terminology, we will use the term  characteristic in
reference to $\chi_v$ and $\chi_w$ or the graphs of
$u\to (\chi_v(u),u)$ and $u\to (\chi_w(u),u)$.\\

(ii) We will call a (part of) characteristic $\{\chi_v(u), s\leq u \leq t\}$  subsonic if 
\begin{align}\label{m28.2}
0\leq \chi_v(u_2) - \chi_v(u_1) \leq u_2-u_1
\qquad \text{ for all } s\leq u_1 \leq u_2 \leq t.
\end{align}

We will call a characteristic $\{\chi_w(u), s\leq u \leq t\}$  subsonic if 
\begin{align}\label{m28.3}
u_1-u_2\leq \chi_w(u_2) - \chi_w(u_1) \leq 0
\qquad \text{ for all } s\leq u_1 \leq u_2 \leq t.
\end{align}

Recall that $x$ is called a local maximum of $f$ if for some $\delta>0$, we have $f(x) = \sup_{x-\delta\leq y \leq x+\delta} f(y)$. Similarly, $x$ is called a local minimum of $f$ if for some $\delta>0$, $f(x) = \inf_{x-\delta\leq y \leq x+\delta} f(y)$.
% If local extrema of a continuous function are isolated
% then the function is strictly monotone on the
% complementary intervals.
A continuous function is monotone on all closed intervals
that contain no local extrema in the interior. If all
local extrema are isolated, the union of these closed
interverals is all of $I$.

\end{definition}
 
\begin{theorem}\label{m28.5}
  Assume that the initial conditions $v(x,0)$ and $w(x,0)$
  are continuous on $I$ and $v(x,0)\geq w(x,0)$ for all
  $x\in I$. If $I=[a_1,a_2]$ then assume that
  $v(a_1,0)=w(a_1,0)$ and $v(a_2,0)=w(a_2,0)$.  Suppose
  that the total number of local extrema of $v(x,0)$ and
  $w(x,0)$ is finite on every finite interval.
  Then there exist unique jointly continuous functions
  $v(x,t)$ and $w(x,t)$ with $x\in I$, $t\geq 0$, such
  that \eqref{n10.8}-\eqref{a5.4} are satisfied, subsonic
  {backward} characteristics of $v$ and $w$ exist for all
  $(x,t) \in I\times [0,\infty)$, and for every open set
  $D\subset I \times(0,\infty)$, if $v(x,t)> w(x,t)$ for
  all $(x,t) \in D$ or $v(x,t)= w(x,t)$ for all
  $(x,t) \in D$ then \eqref{n10.4}-\eqref{n10.5} are
  satisfied in $D$.
\end{theorem}

\subsection{Properties of  solutions}\label{j26.1}

In Section \ref{sec:prop}
we will state and prove the following properties of  solutions $v$ and $w$ to \eqref{n10.4}-\eqref{a5.4} under Assumption \ref{m29.2}.

(i) If $I=[a_1,a_2]$ then the system will freeze forever after $2(a_2-a_1)$ units of time.

(ii) If $I=\R$, the total variations of $v(\,\cdot\,,t)$
and $w(\,\cdot\,,t)$ cannot increase as time increases.

(iii) If $I=[a_1,a_2]$ then momentum and energy, expressed in terms of $\mu$ and $\sigma$, are conserved.

(iv) The difference of occupation measures for $v$ and $w$ is a conserved quantity.

(v) Assuming smooth initial conditions and generic
behavior at extremal points, the slopes  of freezing and thawing curves at their endpoints are universal.

(vi) The solutions depend in a monotone way on the initial conditions.

(vii) One can solve the equations with (some) terminal rather than initial conditions.

(viii) The solutions are not time-reversible for generic initial conditions.

(ix) The solutions are not differentiable along some curves even if the initial conditions are smooth.

(x) Freezing boundaries have slopes between $-1$ and 1. Thawing boundaries have slopes either less than $-1$ or greater than 1. (This is proved in Section \ref{sec:cons}.)

Several properties have obvious counterparts for $\sigma$ and $\mu$; these follow easily from those for $v$ and $w$.

\subsection{The main ideas of proofs}

We expect (and assume) that the speeds of characteristics
will not exceed the propagation speed for the linear
transport equation. Hence, if $y_1<y_2$, the solutions $v$
and $w$ in a triangle of the form
$\{(x,t): t\geq 0, x\in( y_1 + t , y_2 -t)\}$ can depend
only on the data at its base (at time $t=0$).  We first
find the solutions for initial conditions $v(x,0)$ and
$w(x,0)$ satisfying extra conditions, the chief of which
is that there are only finitely many local extrema on
every finite interval.  This allows
us to divide the space into subintervals where the initial
conditions have at most two intervals of strict
monotonicity. Then we can define $v$ and $w$ in the
corresponding space-time triangles in a fairly explicit
way.

To generalize our construction to arbitrary continuous
initial conditions, we express the solutions using totally
explicit formulas in terms of sublevel and superlevel sets
of $v(x,0)$ and $w(x,0)$. This allows us to prove that the
mapping from the initial conditions to solutions is
Lipschitz in appropriate norms.  {The initial data
  satisfying the extra conditions are dense among all
  continuous initial data that satisfy the constraint, so
  we define solutions for more general initial conditions
  as limits of the special ones.}

\subsection{Comparison to Conservation Laws and Shocks}

{ Our system of nonlinear transport equations is
  similar in many ways to nonlinear systems of conservation
  laws \cite{Lax1984,Serre,Smoller}. A key difference is
  that the appearance of shocks---curves across which
  solutions jump---which is unavoidable for genuinely
  nonlinear
  conservation laws, does not occur for our system.\\

A nonlinear system of  conservation laws has the form
\begin{align*}
  u_{t}=  (f(u))_{x} =  A(u)u_{x}
\end{align*}
with \(A(u)= Df(u)\), the Jacobian of \(f\).
\begin{enumerate}
\item For general initial data, solutions are not
  continuous for all time, no matter how smooth the
  initial data. Typically, characteristics which start at
  points where \(u\) has different values intersect,
  forcing the function they would define to be
  discontinuous or multi-valued.
\item Weak (discontinuous) solutions exist for all initial
  data (e.g. \(L^{\infty}\)), but they are
  not unique.
\item Addition of a ``shock criterion'', restores
  uniqueness. For a scalar conservation law, one version of
  the shock criterion is that ``forward characteristics
  impinge on the shock curve.'' An alternative equivalent statement,
  which has a more direct analogy in  our system, is
  that ``every point \((x,t)\) has a backward characteristic''.
\end{enumerate}
\subsubsection{Transport equations with freezing}

Equations \eqref{n10.4}-\eqref{n10.5} can be written in the form

  \begin{align*}
    \begin{pmatrix}
      v
      \\
      w
    \end{pmatrix}_{t} =
    \begin{pmatrix}
      \bone_{(v>w)}& 0
      \\
      0&-\bone_{(v>w)}
    \end{pmatrix}\;
    \begin{pmatrix}
      v
      \\
      w
    \end{pmatrix}_{x}
  \end{align*}
  
  \begin{enumerate}
\item The system in not a conservation law, but
  similar. It is a $2\times2$ nonlinear hyperbolic system.
\item It is a conservation law in each of the liquid and
  frozen regions, where \(A(v,w)\) is diagonal with
  constant eigenvalues  \(\pm 1\) in the liquid region and
 both equal to zero in the frozen region.
\item The initial data, and the solutions must both
  satisfy the constraint \(v(x,0)\ge w(x,0)\) for all
  \(x\).
\item For initial data with isolated extrema, the
  liquid regions, where the eigenvalues are \(\pm 1\), are
  separated from the frozen regions, where the
  eigenvalues are both zero, by continuous freezing and thawing 
  curves.
\item Solutions are continuous. There are no shocks, even
  though forward characteristics starting at different
  points may intersect. The freezing process prevents
  $v$-characteristics (or $w$-characteristics) with the different
  values from intersecting. Two $v$-characteristics with
  the same value may intersect, and the constraint forces
  those characteristic to end, i.e. the point where the
  two meet has no forward characteristic.
\item The solutions need not be unique until we specify
  the analog of shock conditions. One version is: ``Every
  point has at least one backward $v$-characteristic and
  one backward $w$-characteristic.''  This is automatic in
  the frozen and liquid regions, and on the freezing curves,
  but not on the thawing curves.
  
\end{enumerate}
}

\section{Phenomena: characteristics}\label{sec:char}

The proofs of our main results are rather technical. This and the next two sections provide an informal but accurate description of the phenomena associated with the solutions to the  system \eqref{n10.4}-\eqref{a5.4}  of PDEs. 
We hope that this heuristic discussion will help the reader follow the rigorous proofs.
This section will focus on the shape of characteristics.

In all figures, if time is one of the coordinates, the horizontal axis represents space and the vertical axis represents time.

According to Remark \ref{m28.1} (i), if $v>w$ in an open set (``liquid'' region) then within this region $v$-characteristics have to be straight line segments with slope 1 and 
$w$-characteristics have to be straight line segments with slope $-1$.
If $v=w$ in an open set (``frozen'' region) then within this region characteristics for both $v$ and $w$ have slope 0.
See Fig. \ref{fig20}.

\begin{center}
\begin{figure}%[h]
\includegraphics[width=0.5\linewidth]{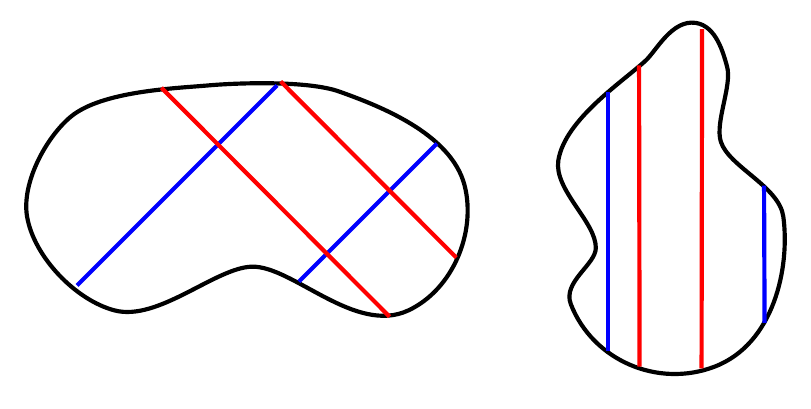}
\caption{A family of characteristics of $v$ (blue) and $w$ (red). 
The region within the curve on the left is ``liquid,'' i.e., $v>w$ inside the region. Slopes of $v$ characteristics (blue) are 1, and $-1$ for $w$ (red) in the liquid region. The region within the curve on the right is ``frozen,'' i.e., $v=w$ inside the region.
Characteristics of $v$ (blue) and $w$ (red) are vertical in the frozen region. Black lines are used to mark subsets of liquid and frozen regions. Their shapes are not typical for freezing and thawing boundaries.}
\label{fig20}
\end{figure}
\end{center}

The main, and perhaps only, interesting aspect of characteristics of $v$ and $w$ is their behavior at the boundary between liquid and frozen zones.
In the rest of this section, all characteristics represent the same single value for $v$ and $w$.
The two most common phenomena at the boundary are (i) transition from sloped to vertical characteristics at a meeting location in space-time (freezing point) of a characteristic of $v$ and a characteristic of $w$, and (ii) transition from overlapping vertical characteristics of $v$ and $w$ at a thawing point, where another sloped characteristic of either $v$ or $w$ meets the first two. In case (ii), if the sloped characteristic that hits the vertical ones represents $v$, a new sloped characteristic representing $w$ is emanating from the meeting location, and vice versa. See the seven lowest transition points  in Fig. \ref{fig15}.

\begin{center}
\begin{figure}%[h]
\includegraphics[width=0.7\linewidth]{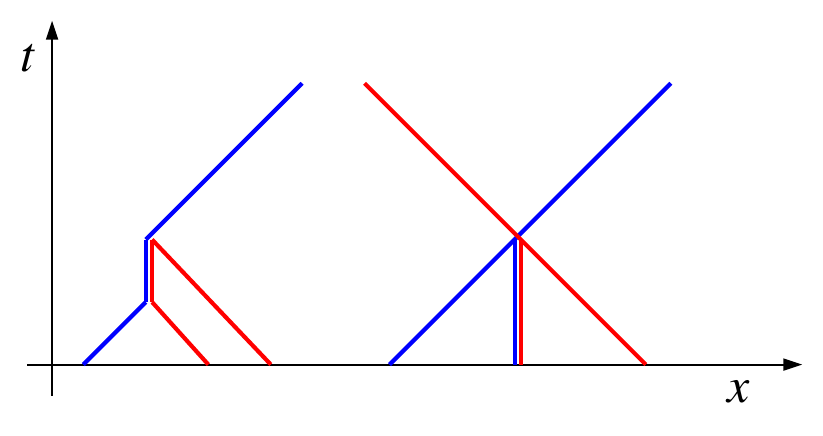}
\caption{A family of characteristics of $v$ (blue) and $w$ (red) with the same value. Characteristics are vertical in the frozen region. Their slopes are 1 for $v$ or $-1$ for $w$ in the liquid region.}
\label{fig15}
\end{figure}
\end{center}

Another non-trivial transition is when four characteristics meet at the same point, two representing $v$ and two representing $w$. For each of the functions $v$ and $w$, one of the characteristics is sloped and one is vertical. Two sloped characteristics, for $v$ and $w$, emanate from the meeting point. See the transition point on the right in  Fig. \ref{fig15}.

It is also possible for the sloped characteristics of $v$
and $w$ to pass through each other
without any interaction. This occurs when the $v$-characteristic carries a local minimum for $v$ and the $w$-characteristic carries a local maximum for $w$. For example, if $v(x,0) = x^2$ and $w(x,0) = -(x-1)^2$ then the characteristics corresponding to the value $0$ will pass through each other
without any interaction.

\section{Phenomena: Freezing and thawing}\label{sec:freeze}

 This section will discuss the shape of frozen and liquid regions, and freezing and thawing boundaries when the initial conditions satisfy Assumption \ref{m29.2} (iv)-(v), saying that the initial conditions have only finitely many local extrema on every finite interval and they are strictly monotone between local extrema.

The boundary between liquid and frozen regions consists of freezing curves and thawing curves. A freezing curve occurs if the value of $v$ increases along a family of  characteristics, the value of $w$ increases along a family of characteristics, the values for $v$ and $w$ span the same intervals, and the two families of characteristics meet. More precisely, wherever a characteristic of $v$ meets a characteristic of $w$ with the same value, the meeting point belongs to the boundary between liquid and frozen zones. The liquid zone is below and the frozen zone is above the boundary. See Fig. \ref{fig10}.

\begin{center}
\begin{figure}%[h]
\includegraphics[width=0.6\linewidth]{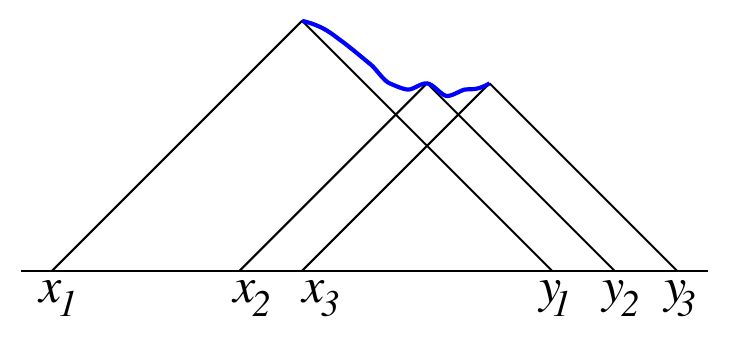}
\caption{Blue curve represents a part of the freezing boundary. $v$ is increasing between $x_1$ and $x_3$. $w$ is increasing between $y_1$ and $y_3$. $v(x_k,0)=w(y_k,0)$ for $k=1,2,3$. Black slanted lines represent characteristics of $v$ (with slope 1) and $w$ (with slope $-1$). The liquid zone is below the blue line and the frozen zone is above the blue line. }
\label{fig10}
\end{figure}
\end{center}

A thawing piece of the boundary occurs if the value of $v$ increases along a family of  characteristics, the value of $w$ decreases along a family of characteristics, the values for $v$ and $w$ span the same intervals, and the two families of characteristics meet. Wherever a characteristic of $v$ meets a characteristic of $w$ with the same value, the meeting point belongs to the boundary between liquid and frozen zones. The liquid zone is to the right and the frozen zone is to the left of the boundary. See Fig. \ref{fig11}.
The roles of $v$ and $w$ can be interchanged. See, for example, Fig. \ref{fig14}.

\begin{center}
\begin{figure}%[h]
\includegraphics[width=0.7\linewidth]{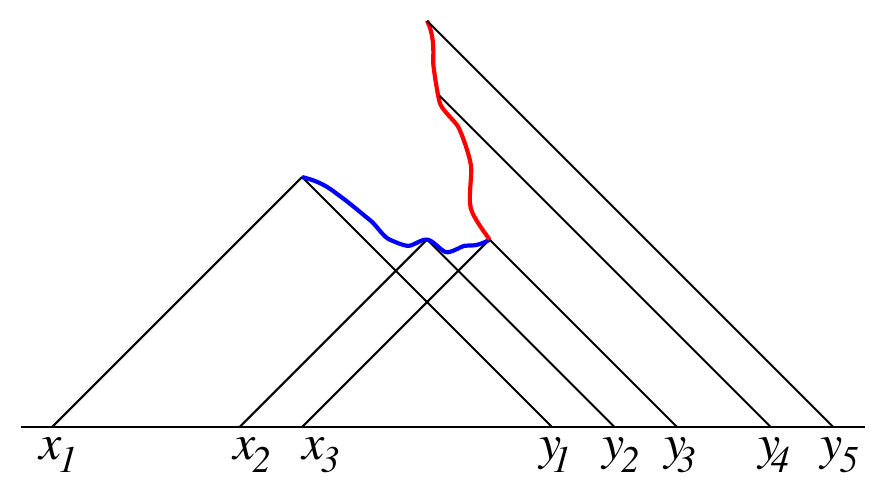}
\caption{Blue curve represents a part of freezing boundary.
Red curve represents a part of thawing boundary.
 $v$ is increasing between $x_1$ and $x_3$. $w$ is increasing between $y_1$ and $y_3$, and decreasing between $y_3$ and $y_5$. $v(x_k,0)=w(y_k,0)$ for $k=1,2,3$. Black slanted lines represent characteristics of $v$ (with slope 1) and $w$ (with slope $-1$). The liquid zone is to the right of the red curve and and below the blue curve. The frozen zone is to the left of the red curve and above the blue curve. }
\label{fig11}
\end{figure}
\end{center}

The simplest bounded frozen region is a curvilinear triangle. See Fig. \ref{fig14}. Its lower side is a freezing part of the boundary. Its left and right sides are thawing parts of the boundary. Frozen regions may have boundaries consisting of many freezing and thawing parts, as in Fig. \ref{fig17}. 

The freezing boundary has slopes  between $-1$ and 1. A thawing piece of the boundary has slopes  between 1 and $\infty$, or it has slopes  between  $-\infty$ and $-1$. 

\begin{center}
\begin{figure}%[h]
\includegraphics[width=0.7\linewidth]{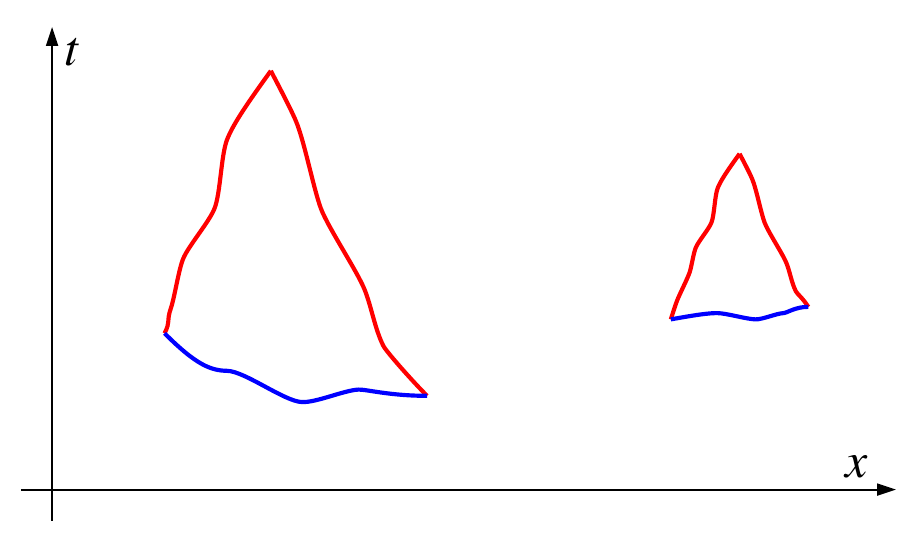}
\caption{The regions within curvilinear ``triangles'' are frozen. The region outside the triangles is liquid. Blue lines represent freezing boundaries. Their slopes are between $-1$ and 1. Red lines represent thawing boundaries. The slopes of the left sides of the triangles are between 1 and $\infty$. The slopes of the right sides of the triangles are between  $-\infty$ and $-1$. A characteristic emanating from  a vertex of a frozen triangle cannot intersect a curvilinear side of another triangle between its vertices.
  }
\label{fig14}
\end{figure}
\end{center}

\begin{center}
\begin{figure}%[h]
\includegraphics[width=0.5\linewidth]{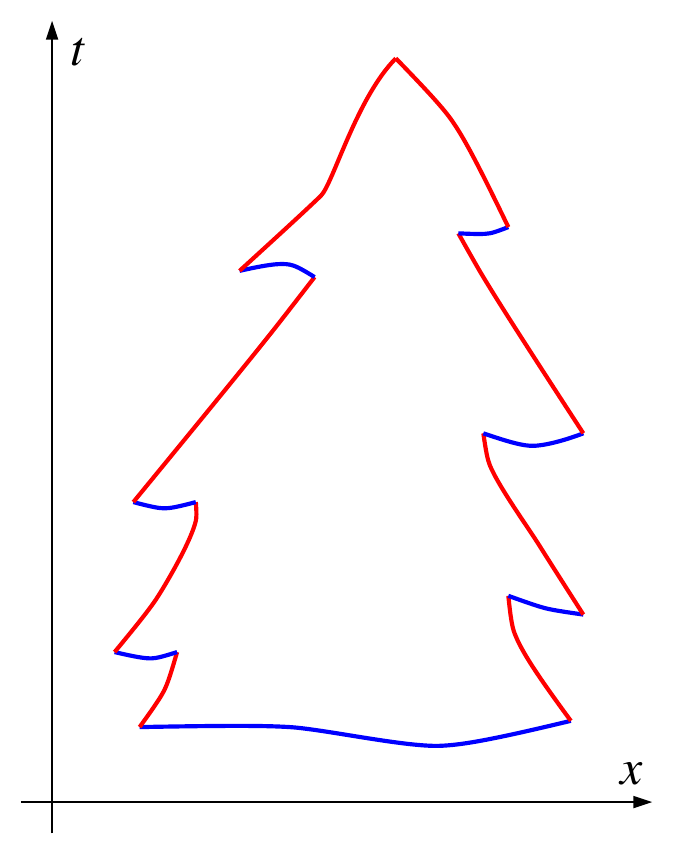}
\caption{The region within curvilinear ``Christmas tree'' is frozen.  Blue curves represent freezing parts of the boundary. Their slopes are between $-1$ and 1. Red curves represent thawing parts of the boundary. Their slopes are either greater than 1 or smaller than $-1$.
  }
\label{fig17}
\end{figure}
\end{center}

Fig. \ref{fig18} illustrates interaction between a frozen region and characteristics of $v$  and $w$. Characteristics are vertical in the frozen region and sloped in the liquid region.

\begin{center}
\begin{figure}%[h]
\includegraphics[width=0.6\linewidth]{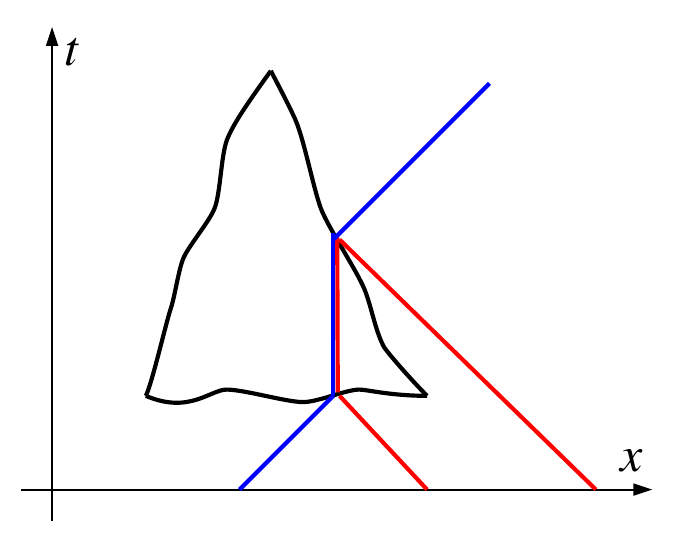}
\caption{
Interaction between a frozen region (interior of the  curvilinear black triangle) and characteristics of $v$ (blue) and $w$ (red). Characteristics are vertical in the frozen region and sloped in the liquid region.
  }
\label{fig18}
\end{figure}
\end{center}

See also Example \ref{a2.1} (ii) with the accompanying Fig. \ref{fig:31}.

\subsection{Freezing Prevents Shocks}

 It is  typical for solutions to nonlinear transport equations to
 develop shocks (discontinuities). The sequence of graphs
in Figs. \ref{fig1a}-\ref{fig4b} illustrate why
\(v\)-characteristics with different values cannot meet to
cause a shock. Characteristics of $v$ are green and those of $w$ are red.

\FloatBarrier
\begin{center}
\begin{figure}[h]
\includegraphics[width=0.5\linewidth]{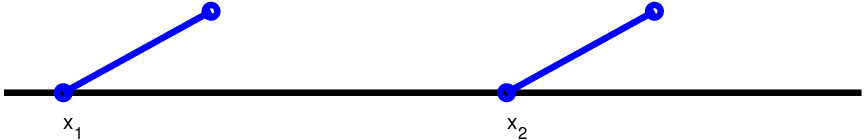}
\caption{Suppose that
  $v(x_{1})<v(x_{2})$. Parallel
  characteristic can't meet to cause a shock.}
\label{fig1a}
\end{figure}
\end{center}

\begin{center}
\begin{figure}[h]
\includegraphics[width=0.5\linewidth]{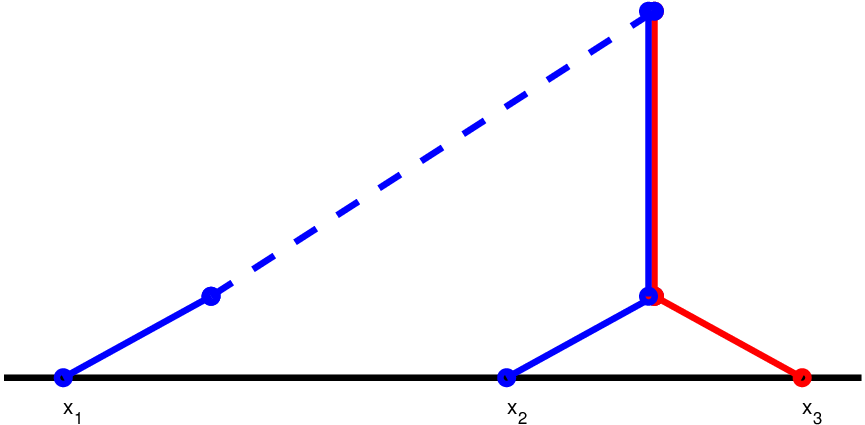}
\caption{But if the  $v$-characteristic from $(x_2,0)$ meets a
$w$-characteristic and freezes, the 
$v$-characteristic from $(x_1,0)$ might catch up.}
\label{fige2a}
\end{figure}
\end{center}
    
\begin{center}
\begin{figure}[h]
\includegraphics[width=0.5\linewidth]{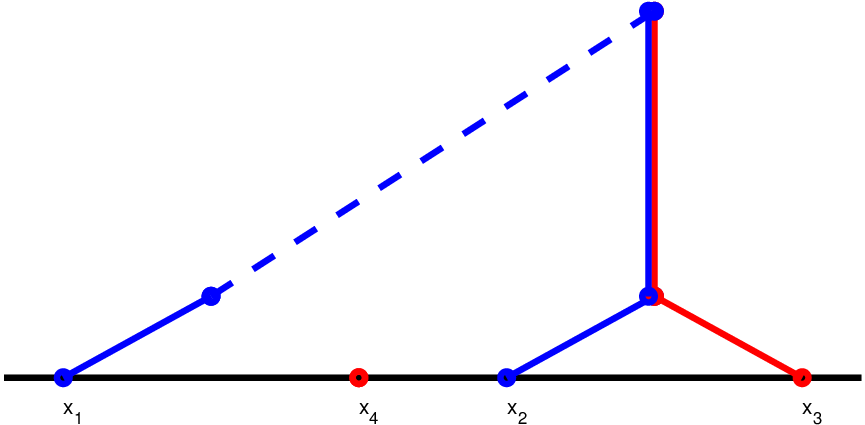}
\caption{However, as $w(x_{1},0)<v(x_{1},0)$, and
    $w(x_{3},0)=v(x_{2},0)>v(x_1,0)$, there exists
    $x_{1}<x_{4}<x_{3}$ with $w(x_{4},0)=v(x_1,0)$.}
\label{fig3a}
\end{figure}
\end{center}
    
\begin{center}
\begin{figure}[h]
\includegraphics[width=0.5\linewidth]{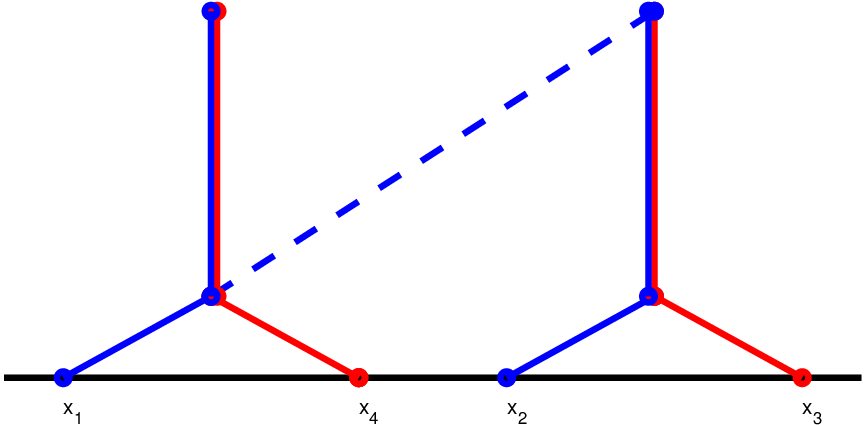}
\caption{And the $w$-characteristic starting from $(x_{4},0)$
    freezes the $v$-characteristic starting from $(x_{1},0)$
    before the shock can occur.}
\label{fig4a}
\end{figure}
\end{center}
\FloatBarrier
\begin{center}
\begin{figure}%[h]
\includegraphics[width=0.5\linewidth]{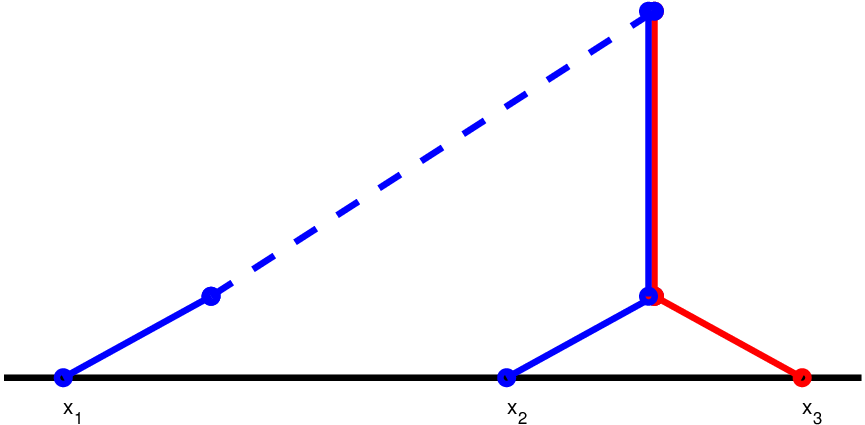}
\caption{If $v(x_{1},0)>v(x_{2},0)$, we need 
    slightly different argument.  Parallel
characteristic can't meet to cause a shock, so assume that
the $v$-characteristic emanating from $(x_{2},0)$ freezes.}
\label{fige12}
\end{figure}
\end{center}

\begin{center}
\begin{figure}%[h]
\includegraphics[width=0.5\linewidth]{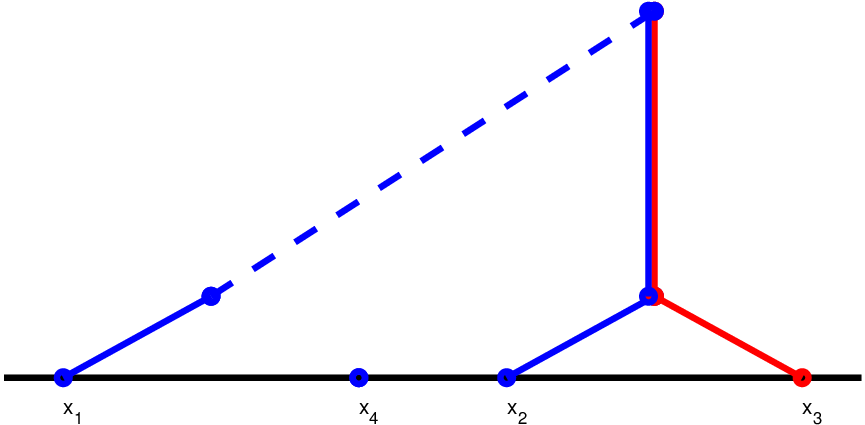}
\caption{As $v(x,0)$ is increasing
to the left of $x_2$ and $v(x_{1},0)>v(x_{2},0)$,
there must be an  \(x_{4}\) between \(x_{1}\)
and \(x_{2}\) satisfying $v(x_{4},0)=v(x_2,0)$.}
\label{fig3b}
\end{figure}
\end{center}

\begin{center}
\begin{figure}%[h]
\includegraphics[width=0.5\linewidth]{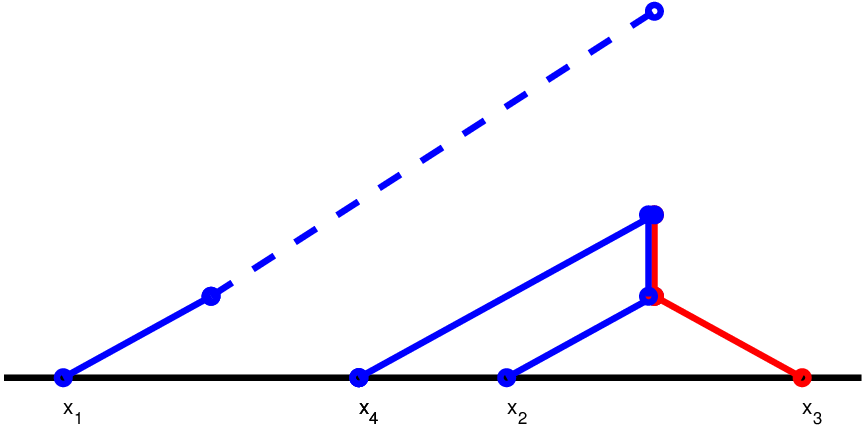}
\caption{The decreasing $v$-characteristic emanating from $(x_{4},0)$
doesn't freeze, and is parallel to the $v$-characteristic
emanating from $(x_{1},0)$, so that characteristic cannot
intersect the $v$-characteristic emanating from $(x_{2},0)$.
The  $v$-characteristics from $(x_{2},0)$ and $(x_{4},0)$
may (but need not) meet and
thaw, as they do  in the figure.}
\label{fig4b}
\end{figure}
\end{center}

\section{Phenomena: annihilation of sublevel and superlevel sets}\label{sec:ann}

A function can be recovered from knowledge of its sublevel
or superlevel sets. As we will see in Section
\ref{sec:gen}, when the initial conditions satisfy
Assumption \ref{m29.2}, the evolution of the sublevel sets
of \(v\) and the superlevel sets of \(w\) is surprisingly
simple. Moreover, the description of the evolution of
these sets persists even without Assumption
\ref{m29.2}. We illustrate this  below, leaving the
details for Section \ref{sec:gen}.

The sets $A(v,b,t) = \{x:v(x,t)\leq b\} $ represent sublevel
sets of of \(v\) at a fixed \(t\)  and
the sets $A(w,b,t) = \{x:w(x,t)\geq b\}$ represent superlevels
of $w$  at \(t\). See Figs. \ref{fig21}-\ref{fig25}.

\FloatBarrier
\begin{center}
\begin{figure}%[h]
\includegraphics[width=0.8\linewidth]{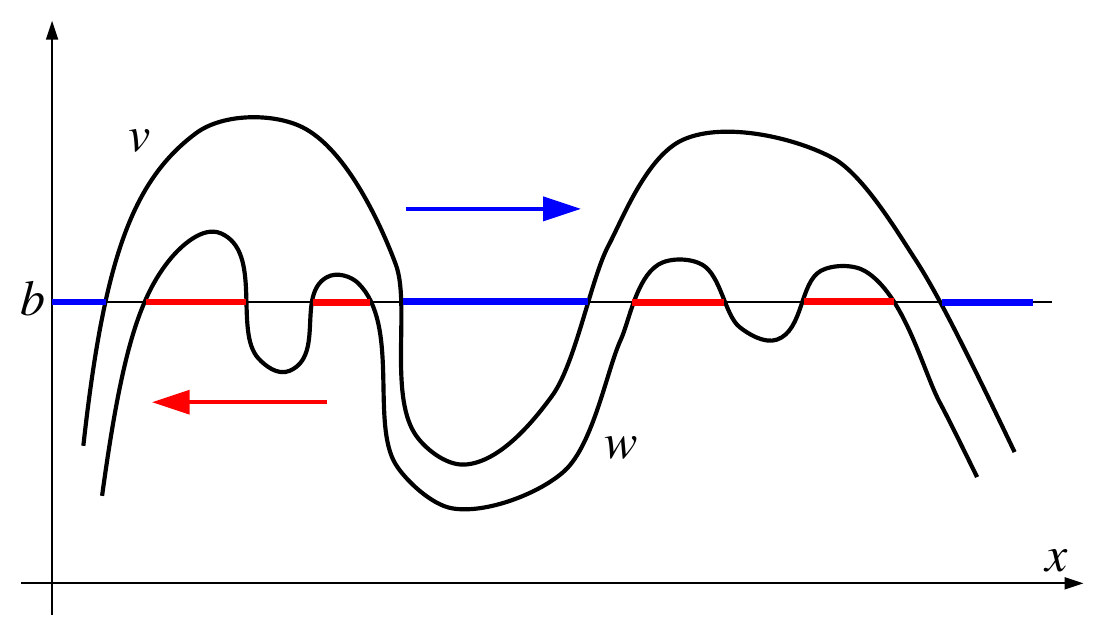}
\caption{ Blue intervals represent the sublevel set of
  $v(x,0)$ and red intervals represent the superlevel set
  of $w(x,0)$, both corresponding to value $b$.  With
  time, the blue set is moving to the right and the red
  set is moving to the left. When they meet, the sets annihilate each
  other at the same rate. The vertical axis represents the
  values of $v$ and $w$, not time.  }
\label{fig21}
\end{figure}
\end{center}

\begin{center}
\begin{figure}%[h]
\includegraphics[width=0.8\linewidth]{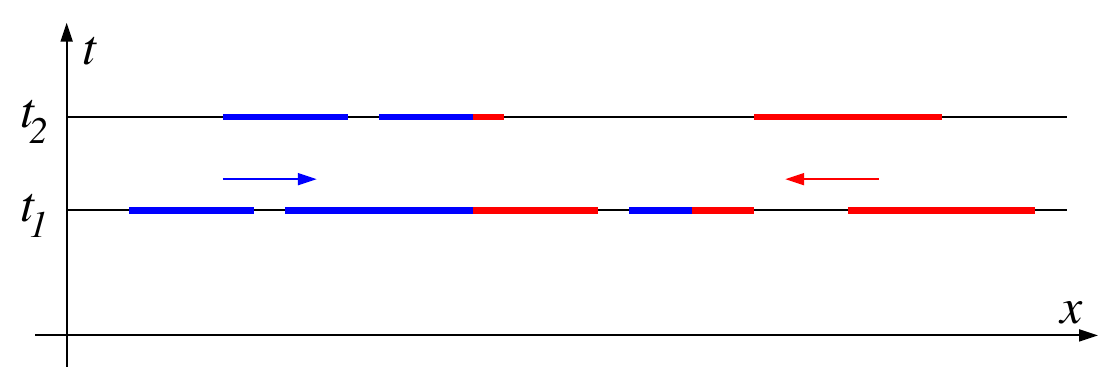}
\caption{ Blue intervals represent the sublevel set of
  $v(x,t)$ and red intervals represent the superlevel set
  of $w(x,t)$, both corresponding to value $b$, at times
  $t=t_1,t_2$, with $t_1< t_2$. Until they meet, blue intervals move to the
  right with speed 1 and red intervals move left with the
  same speed. When a blue interval meets a red interval, the
  endpoints that meet stop moving, while the
  other endpoints continue to move. Thus
  both intervals shorten at the same rate until the
  shorter interval becomes empty. Then the interval that
  remains resumes its translation.}
\label{fig22}
\end{figure}
\end{center}

\begin{center}
\begin{figure}%[h]
\includegraphics[width=0.8\linewidth]{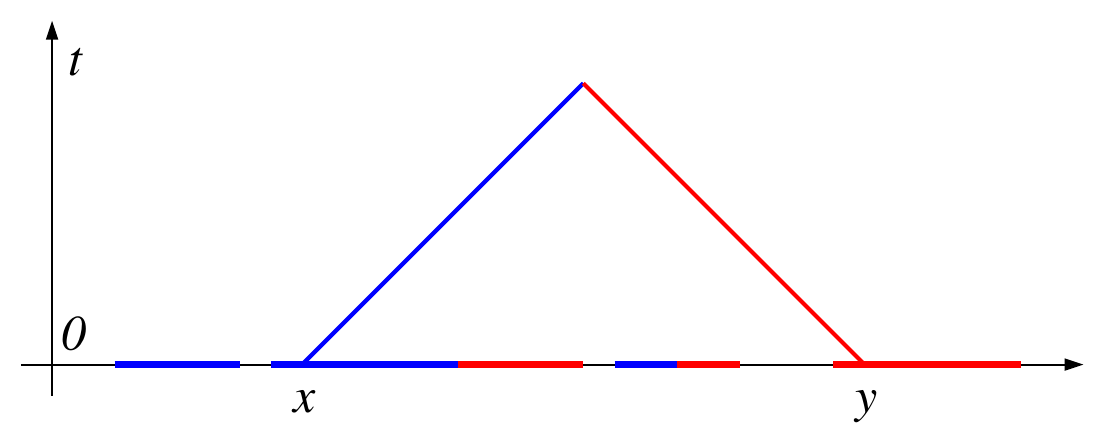}
\caption{
Blue intervals represent the sublevel set of $v(x,0)$ and red intervals represent the superlevel set of $w(x,0)$, both corresponding to  value $b$.
Slanted lines represent rays emanating from $x$ and $y$ and containing  points in space-time which belong to the sublevel set of $v$ and superlevel set of $w$ corresponding to value $b$. 
The total length of the blue intervals between $x$ and $y$ is equal to the total length of red intervals between $x$ and $y$.
The slanted lines are \emph{not} characteristics of $v$ and $w$; in particular, their slopes are always 1 for $v$ and $-1$ for $w$. 
  }
\label{fig23}
\end{figure}
\end{center}

\begin{center}
\begin{figure}%[h]
\includegraphics[width=0.7\linewidth]{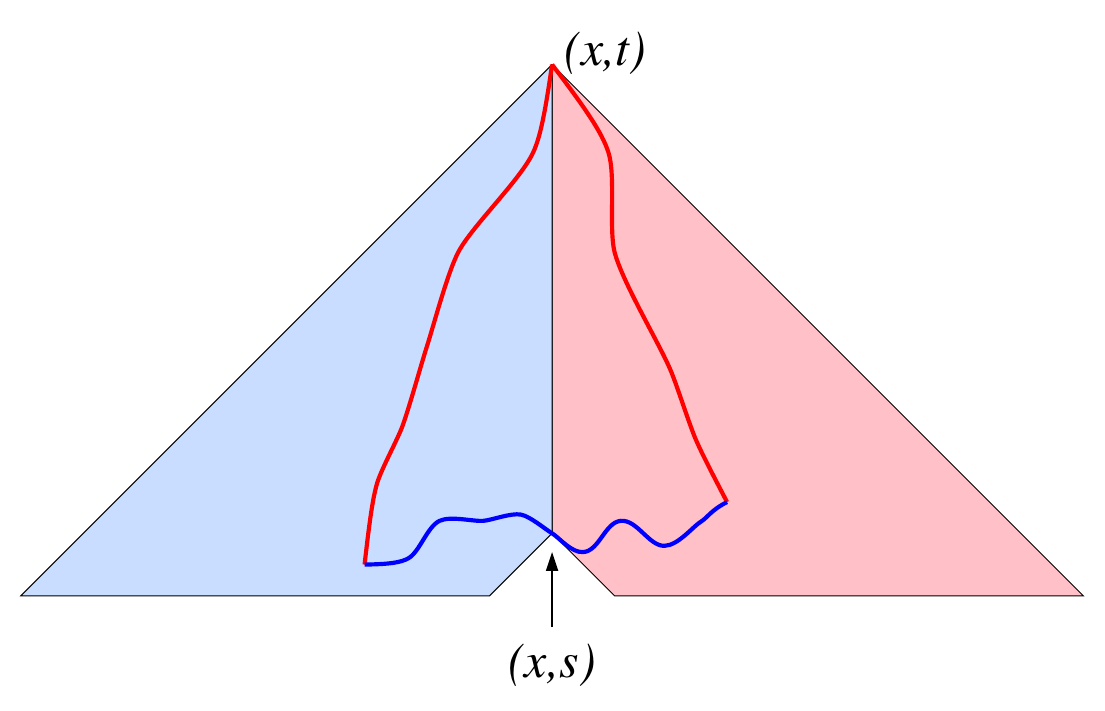}
\caption{
Interaction between a frozen region within curvilinear triangle and sublevels of $v$ (light blue region on the left) and superlevels of $w$ (light red region on the right), both corresponding to the same value of $b$. The value of $b$ is chosen so that the sublevels and superlevels reduce to a single point $(x,t)$ at the top of the frozen region. The $v$-sublevels and $w$-superlevels meet for the first time at $(x,s)$, a point in space time with the same first coordinate as $(x,t)$ that belongs to the freezing part of the boundary of the triangle.  Blue curve represents the freezing boundary. Red curves represent thawing boundaries. 
  }
\label{fig25}
\end{figure}
\end{center}

\section{Construction and existence of solutions}\label{sec:cons}

In this section we will prove one of our main results under the following rather strong   assumption.

\begin{assumption}\label{m29.2}

(i) $v(x,0) \geq w(x,0)$ for $x\in I$.

(ii) If $I=[a_1,a_2]$ then $v(a_1,0)=w(a_1,0)$ and $v(a_2,0)=w(a_2,0)$.

(iii) The functions $v(x,0)$ and $w(x,0)$ are Lipschitz with a constant $\lambda <\infty$, i.e., for all $x$ and $y$,
\begin{align}\label{m29.3}
|v(x,0)-v(y,0)| \leq \lambda |x-y|,
\qquad
|w(x,0)-w(y,0)| \leq \lambda |x-y|.
\end{align}

(iv) The total number of local extrema of  $v(x,0)$ and $w(x,0)$ is finite on every finite interval.

% (v) Neither $v(x,0)$ nor $w(x,0)$ is constant on any interval of positive length.

\end{assumption}

We will define  ``freezing curves'' $\calF$ and a ``thawing curves'' $\calT$ in the statement of the next lemma.

The idea of the freezing curve $\calF$ is that $v>w$ below the curve (so this is the liquid zone) and $v=w$ on $\calF$, hence the name the ``freezing curve.''
For the thawing curve $\calT$, $v>w$ above the curve (so this is the liquid zone) and $v=w$ on $\calT$.

\begin{lemma}\label{m22.1}
(i) Suppose that $y_1<y_2$,
$[y_1,y_2]\in I$ and let
\begin{align}\label{m25.1}
\wh U&=\{(x,t): t\geq 0, x\in( y_1 + t , y_2 -t)\},\\
\calF&=\{(x,t)\in \wh U: v(x-t,0)= w(x+t,0)\}.\label{m25.2}
\end{align}
If the functions $ v(x,0)$ and $ w(x,0)$ are continuous
and strictly increasing on $[y_1, y_2]$, then the set
$\calF$ is connected, possibly empty. If it is non-empty
then for some $z_1,z_2 \in [y_1, y_2]$, $\calF$ is the
graph of a function $h_\calF: [z_1,z_2] \to [0,\infty)$
(from space to time) such that
\begin{align}\label{m22.5}
|h_\calF(x_1) - h_\calF(x_2)|
< |x_1-x_2|,
\end{align}
for all $x_1,x_2\in [z_1,z_2]$.
If $\calF$ is non-empty then $\wh U\setminus \calF$ has two connected components, one below $\calF$ (say, $\wh U^-$), and the other above $\calF$ (say, $\wh U^+$).
One of the sets $\wh U^-$ and $\wh U^+$ may be empty.

(ii) Suppose that $\eps>0$,
$[z-\eps,z+\eps]\in I$, and let  
\begin{align}\label{m29.1}
U&=U_{z,\eps}=\{(x,t): t\geq 0, x\in( z-\eps + t , z+\eps -t)\},\\
U_r&=\{(x,t)\in U: x\geq z-t\},\notag\\
\calT&=\{(x,t)\in U_r: x\leq z, x+t\geq z, v(x,z-x)= w(x+t,0)\}.\label{a28.1}
\end{align}
Suppose that the function $x\to w(x,0)$ is continuous and
strictly decreasing on $[z, z+\eps]$, the functions
$x\to v(x,z-x)$ and $x\to w(x,z-x)$ are (defined and) continuous and strictly
increasing on $[z-\eps/2,z]$,  
$ v(x,z-x)= w(x,z-x)$  for $x\in[z-\eps/2,z]$ (i.e. this set is frozen),
and
$v(z,0)=w(z,0)$. Then the set $\calT$ is non-empty,
connected and for some $z_3 \in [z-\eps/2, z)$, $\calT$ is
the graph of a function $h_\calT: [z_3,z] \to [0,\infty)$
(from space to time) such that
\begin{align}\label{m22.6}
h_\calT(x_1) - h_\calT(x_2)
> x_2-x_1,
\end{align}
 for all $z_3\leq x_1<x_2\leq z$.
The set $U_r\setminus \calT$ has two connected components, one below $\calT$ (say, $U_r^-$), and the other above $\calT$ (say, $U_r^+$).

\end{lemma}

The strange looking assumptions in part (ii) of Lemma \ref{m22.1} that  the function
$x\to v(x,z-x)$ is defined and 
increasing, and the set is frozen are needed in
case (iv) in the proof of Lemma \ref{j6.5}.
In that proof we show that the assumptions in part (i) of Lemma \ref{m22.1}
imply those in part (ii). See Fig. \ref{fig19a}.

\begin{proof}
(i)
Suppose that $\calF$ is non-empty, $(x_1,t_1), (x_2,t_2) \in \calF$ and $(x_1,t_1)\ne (x_2,t_2) $. 

We will argue that $x_1-t_1\ne x_2-t_2$. Suppose otherwise. 
If $t_1=t_2$ then $x_1 = x_2$, a contradiction.
Assume without loss of generality that $t_1< t_2$. Then
\begin{align*}
x_1 + t_1 &= x_1-t_1+ 2 t_1= x_2-t_2 + 2 t_1 
< x_2-t_2 + 2 t_2 = x_2 + t_2,\\
w(x_1+t_1,0) &= v(x_1-t_1,0) = v(x_2 - t_2,0) = w(x_2+t_2,0).
\end{align*}
But we cannot have $w(x_1+t_1,0) = w(x_2+t_2,0)$ because
$x_1+t_1 < x_2+t_2$ and $ w(x,0)$ is strictly increasing.
Hence $x_1-t_1\ne x_2-t_2$
and for a similar reason $x_1+t_1\ne x_2+t_2$.

We will
assume without loss of generality that $x_1-t_1< x_2-t_2$.
The definition of $\calF$ and the fact that  $ v(x,0)$ and $ w(x,0)$ are strictly increasing imply that $x_1+t_1< x_2+t_2$.
Adding the two inequalities together, we obtain $x_1 < x_2$. Hence, $\calF$ is the graph of a function $h_\calF$.

The functions  $ v(x,0)$ and $ w(x,0)$ are continuous and strictly increasing,
$v(x_1-t_1,0) = w(x_1+t_1,0) $ and $ v(x_2 - t_2,0) = w(x_2+t_2,0)$,
so for every $z_1 \in ( x_1-t_1, x_2-t_2)$ there exists $f(z_1) :=z_2 \in ( x_1+t_1, x_2+t_2)$ such that $v(z_1,0)=w(z_2,0)$.
Let $x_3 = (z_1+z_2)/2$ and $t_3 = (z_2-z_1)/2$. Then 
$(x_3,t_3)\in\calF$.
The function $f$ is continuous and maps $( x_1-t_1, x_2-t_2)$ onto $( x_1+t_1, x_2+t_2)$. This implies that  $z_1\to w(f(z_1),0)$ is continuous and, therefore, $z_1 \to (x_3,t_3)$ is continuous. Thus $\calF$ is a connected set. 

Suppose that $\calF$ does not satisfy \eqref{m22.5}. Then there exist $(x_1,t_1), (x_2,t_2) \in \calF$ with $x_1-t_1< x_2-t_2$ and $|t_1-t_2| > x_2-x_1$.
If $t_1 > t_2 +x_2-x_1$ then $x_1 + t_1 > x_2 +t_2$ 
and
\begin{align*}
w(x_1+t_1,0) = v(x_1-t_1,0) < v(x_2-t_2,0) = w(x_2-t_2,0),
\end{align*}
a contradiction with the assumption that $ w(x,0)$ is strictly increasing.
If $t_2 > t_1 +x_2-x_1$ then $x_1 - t_1 > x_2 -t_2$,
a contradiction with the assumption that $x_1-t_1< x_2-t_2$.
This proves \eqref{m22.5}.

If $\calF$ is non-empty then there exist $\eps_1,\eps_2>0$ such that 
the range of $v(\,\cdot\,,0)$ over $[y_1,y_1+\eps_1]$ is the same as the 
range of $w(\,\cdot\,,0)$ over $[y_2-\eps_2,y_2]$. Hence, for every $x\in [y_1,y_1+\eps_1]$ there is $y\in[y_2-\eps_2,y_2]$ such that $v(x,0) = w(y,0)$, and for every $y\in [y_2-\eps_2,y_2]$ there is $x\in[y_1,y_1+\eps_1]$ such that $v(x,0) = w(y,0)$. Therefore, $\calF$ must extend from the left hand side boundary of $\wh U$ to the right hand side boundary of $\wh U$. As a consequence,
 $\wh U\setminus \calF$ has two connected components, $\wh U^-$  below $\calF$, and $\wh U^+$ above $\calF$.

This completes the proof of (i).

\bigskip
(ii) Recall that
$v(z,0)=w(z,0)$,
the function $x\to w(x,0)$ is continuous and strictly decreasing on $[z, z+\eps]$, and the function $x\to v(x,z-x)$ is  continuous and strictly increasing on $[z-\eps/2,z]$. Hence
there exist $x_1\in [z-\eps/2,z]$ and $y_1\in [z, z+\eps]$ such that $v(x_1, z-x_1) = w(y_1,0)$. Then  $x_1+(y_1-x_1) \geq z$ and 
\begin{align*}
v(x_1 , z- x_1) = w(y_1,0) = w(x_1 + (y_1-x_1),0),
\end{align*}
thus showing that
$(x_1, y_1-x_1)  \in \calT$ and, therefore, $\calT$ is non-empty.

Suppose that $(x_1,t_1), (x_2,t_2) \in \calT$ and $(x_1,t_1)\ne (x_2,t_2) $. 

We will argue that $x_1\ne x_2$. Suppose otherwise. 
Assume without loss of generality that $t_1< t_2$. Then
\begin{align}\notag
x_1 + t_1 &<  x_2 + t_2,\\
w(x_1+t_1,0) &= v(x_1,z-x_1) = v(x_2,z - x_2) = w(x_2+t_2,0).\label{m23.1}
\end{align}
But we cannot have $w(x_1+t_1,0) = w(x_2+t_2,0)$ because
$x_1+t_1 < x_2+t_2$ and $ w(x,0)$ is strictly decreasing.
Hence $x_1\ne x_2$ and, therefore, $\calT$ is the graph of a function $h_\calT$.

Assume without loss of generality that $x_1< x_2$.
Since
$v(z,0)=w(z,0)$,
the function $x\to w(x,0)$ is continuous and strictly decreasing on $[z, z+\eps]$, and the function $x\to v(x,z-x)$ is  continuous and strictly increasing on $[z-\eps/2,z]$,
we have $x_1+t_1> x_2+t_2$.

By continuity of $x\to v(x,z-x)$ and $x\to w(x,0)$, for every $z_1 \in ( x_1, x_2)$ there exists $g(z_1) :=z_2 \in ( x_1+t_1, x_2+t_2)$ such that $v(z_1,z-z_1)=w(z_2,0)$.
Let $x_3 = z_1$ and $t_3 = z_2-z_1$. Then 
$(x_3,t_3)\in\calT$.
The function $g$ is continuous and maps $( x_1, x_2)$ onto $( x_2+t_2, x_1+t_1)$. This implies that  $z_1\to w(g(z_1),0)$ is continuous and, therefore, $z_1 \to (x_3,t_3)$ is continuous. We conclude that $\calT$ is a connected set. 

Suppose that $\calT$ does not satisfy \eqref{m22.6}. Then there exist $(x_1,t_1), (x_2,t_2) \in \calF$ with $x_1< x_2\leq z$ and $ h_\calT(x_1)-h_\calT(x_2) \leq x_2-x_1$. Since $x\to w(x,0)$ is  decreasing,
\begin{align*}
t_1-t_2 &= h_\calT(x_1)-h_\calT(x_2) \leq x_2-x_1,\\
x_1+t_1 &\leq x_2+t_2,\\
v(x_1,z-x_1)&=w(x_1+t_1,0) \geq w(x_2+t_2,0) = v(x_2, z-x_2).
\end{align*}
This is a contradiction with the assumption that $x\to v(x,z-x)$ is   strictly increasing. We conclude that \eqref{m22.6} holds true.

Since $\calT$ is connected,  $h_\calT$ is defined on an interval $[z_3,z]$ for some $z_3\in[z-\eps/2,z)$. We will argue that $h_\calT(z_3) = \eps - (z-z_3)$, i.e., $\calT$ touches the boundary of $U_r$ at the point $(z_3, h_\calT(z_3))= (z_3, \eps - (z-z_3))$. Suppose that $h_\calT(z_3) < \eps - (z-z_3)$. Then $z_3 >z - \eps/2$. If $y_1 = z_3+ h_\calT(z_3)$ then
\begin{align*}
y_1 = z_3+ h_\calT(z_3)= z_3+ \eps - (z-z_3) < z+\eps .
\end{align*}
Let $\delta>0$ be the distance from $(z_3, h_\calT(z_3))=(z_3, y_1-z_3)$ to the boundary of $U_r$.
By continuity and strict monotonicity of $x\to w(x,0)$ and $x\to v(x,z-x)$ we can find 
$z_4\in(z-\eps/2, z_3)$ and $y_2 \in (y_1, z+\eps)$ such that $v(z_4,z-z_4) = w(y_2,0)
= w (z_4 +(y_2-z_4),0)$ and the distance from $(z_4, y_2-z_4)$ to the boundary of $U_r$ is greater than $\delta/2$. Therefore, $(z_4, y_2-z_4)\in \calT$, a contradiction with the assumption that $z_3$ is the left endpoint of  the domain of $h_\calT$.
We see that $\calT$  extends from the left hand side of the upper boundary of $U_r$ to  $(z,0)$. As a consequence,
 $U_r\setminus \calT$ has two connected components, $U_r^-$  below $\calT$, and $U_r^+$ above $\calT$.
\end{proof}

\begin{lemma}\label{j6.5}
Suppose that Assumption \ref{m29.2} holds. Then

(a)
For every $z\in I$ there exists $\eps>0$ such that either $[z-\eps,z+\eps]\in I$ or $I=[a_1,a_2]$ and $z=a_1$ or $a_2$, and 

(1) both $v(x,0)$ and $w(x,0)$ are strictly monotone on $[z-\eps,z]\cap I$ and they are strictly monotone on $[z,z+\eps]\cap I$, and

(2) either $\inf_{x\in[z-\eps, z+\eps]\cap I}v(x,0)\geq \sup_{x\in[z-\eps, z+\eps]\cap I}w(x,0)$ or $v(z,0)=w(z,0)$.

\medskip

(b) Suppose $z$ and $\eps$ satisfy (1) and (2) and let
\begin{align*}
U=\{(x,t): t\geq 0, x\in( z-\eps + t , z+\eps -t)\cap I\}.
\end{align*}
Then there exists a continuous  solution to 
\eqref{n10.4}-\eqref{a5.4} in $U$ in the sense of Theorem \ref{m28.4} satisfying \eqref{a5.4}.

\end{lemma}

\begin{proof}

Part (a) follows easily from our assumptions that $v(x,0)$ and $w(x,0)$ are continuous, have a finite number of local extrema on any finite interval, are strictly monotone between the local extrema, and $v(x,0)\ge w(x,0)$ for all $x$.

(b) 
We will consider different cases of  monotone behavior (increasing or decreasing) of functions $v$ and $w$ on $[z-\eps,z]$ and $[z, z+\eps]$. 
In every case we will define $v(x,t)$ and $w(x,t)$ explicitly. It is easy to check that our definitions agree with the conditions specified in Theorem \ref{m28.4} and \eqref{a5.4}; we leave the verification to the reader.

In cases (i)-(vi) below, we assume that $[z-\eps,z+\eps]\in I$.

Recall Lemma \ref{m22.1} and its notation.

\medskip

(i) Suppose that $\inf_{x\in[z-\eps, z+\eps]}v(x,0)\geq \sup_{x\in[z-\eps, z+\eps]}w(x,0)$.
If we take $[y_1,y_2]=[z-\eps, z+\eps]$ in \eqref{m25.1} then  $\calF $ defined in \eqref{m25.2} is empty or contains one point. For $(x,t)\in U$, we let
$v(x,t) = v(x-t,0)$  and $w(x,t) = w(x+t,0)$.

The whole set $U$ is a liquid zone, i.e. $v(x,t)>w(x,t)$ for all $(x,t)\in U$. The function $v$ propagates along characteristics with the slope 1 and $w$'s characteristics have slope $-1$. 

\medskip

In view of assumption (2) of the lemma, we can assume that $v(z,0) = w(z,0)$ in cases (ii)-(vi) below.

\medskip

(ii) Suppose that
 $v(z,0) = w(z,0)$ and  one of the following conditions (a)-(c) is satisfied.
\smallskip

(ii)(a) $x\to v(x,0)$ is strictly decreasing on $[z-\eps, z+\eps]$.

\smallskip

(ii)(b) $y\to w(y,0)$ is strictly decreasing on $[z-\eps, z+\eps]$.

\smallskip

(ii)(c) $x\to v(x,0)$ is strictly decreasing on $[z-\eps, z]$
and $x\to v(x,0)$ is strictly increasing on $[z, z+\eps]$
and $y\to w(y,0)$ is strictly increasing on $[z-\eps, z]$ and
$y\to w(y,0)$ is strictly decreasing on $[z, z+\eps]$.

\smallskip
Let $[y_1,y_2]=[z-\eps, z+\eps]$ in \eqref{m25.1}.
It is easy to check that in all cases (ii) (a)-(c), the freezing curve $\calF$ is empty or contains one point.

As in case (i), we let
$v(x,t) = v(x-t,0)$  and $w(x,t) = w(x+t,0)$ for $(x,t)\in U$.
Once again,
the whole set $U$ is a liquid zone, i.e. $v(x,t)>w(x,t)$ for all $(x,t)\in U$. The function $v$ propagates along characteristics with the slope 1 and $w$'s characteristics have slope $-1$. 

See \eqref{a5.4} in relation to the construction presented in case (ii).

\medskip

(iii) 
Let $[y_1,y_2]=[z-\eps, z+\eps]$ in \eqref{m25.1}. Suppose that
 $v(z,0) = w(z,0)$ and the functions $x\to v(x,0)$ and $x\to w(x,0)$ are strictly increasing on $[z-\eps, z+\eps]$.  These assumptions imply that
 $\calF$ is non-empty.  
  
We let
$v(x,t) = v(x-t,0)$  and $w(x,t) = w(x+t,0)$ for $(x,t)\in \wh U^-$.
The set $\wh U^-$ is a liquid zone, i.e. $v(x,t)>w(x,t)$ for all $(x,t)\in \wh U^-$. The function $v$ propagates in $\wh U^-$ along characteristics with the slope 1 and $w$'s characteristics have slope $-1$.

For $(x,t)\in \wh U^+$, we find $s\leq t$ such that $(x,s) \in \calF$ and we let $v(x,t)=w(x,t) = v(x,s) =w(x,s)$. 
The set $\wh U^+$ is a frozen zone, i.e. $v(x,t)=w(x,t)$ for all $(x,t)\in \wh U^+$. The functions $v$ and $w$ propagate in $\wh U^+$ along characteristics with the slope 0.

\medskip
(iv) Suppose that 
 $v(z,0) = w(z,0)$, the functions $x\to v(x,0)$ and $x\to w(x,0)$ are strictly increasing on $[z-\eps, z]$ and $x\to w(x,0)$ is strictly decreasing on $[z, z+\eps]$.
 
Let 
\begin{align*}%\label{?}
\wh U&=\{(x,t): t\geq 0, x\in( z-\eps + t , z -t)\}.
\end{align*}
We will apply Lemma \ref{m22.1} (i) to the set $\wh U$ with $[z-\eps,z]$ playing the role of $[y_1,y_2]$.

Case (iv) is illustrated in Fig. \ref{fig19a}.  The triangle with the base $[z-\eps,z+\eps]$ is split by a line with the slope $-1$ passing through $(z,0)$ into a smaller triangle $\wh U$ with the base $[z-\eps,z]$ and the complementary part $U_r$ of $U$. The right endpoints of the freezing curve $\calF$ and thawing curve $\calT$ are at $(z,0)$ because $v(z,0) = w(z,0)$.

\begin{center}
\begin{figure}%[h]
\includegraphics[width=0.9\linewidth]{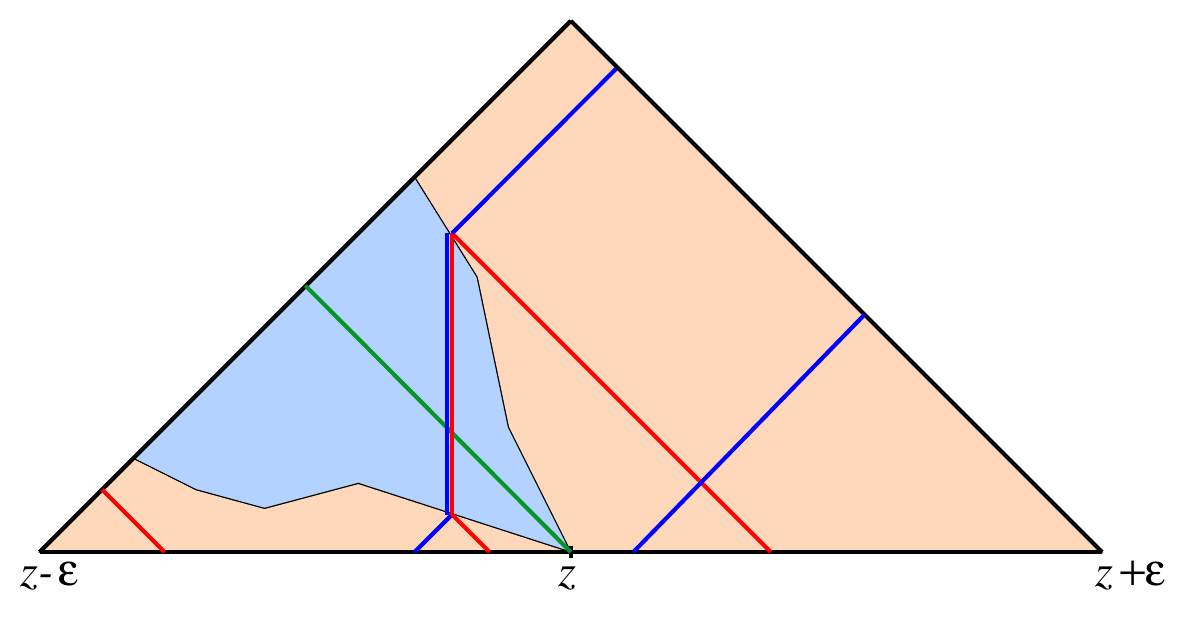}
\caption{
Case (iv) of the proof of Lemma \ref{j6.5}. The green line 
cuts $U$ into a smaller triangle $\wh U$ to the left and $U_r$ to the right. 
The liquid zone is orange and the frozen zone is blue. The right endpoints of the freezing curve $\calF$ and thawing curve $\calT$ are at $(z,0)$ because $v(z,0) = w(z,0)$. Characteristics of $v$ are blue while characteristics of $w$ are red.
  }
\label{fig19a}
\end{figure}
\end{center}

The functions $x\to v(x,0)$ and $x\to w(x,0)$ are strictly increasing on $[z-\eps, z]$.  These assumptions imply that
 $\calF$ is non-empty.  
  
We let
$v(x,t) = v(x-t,0)$  and $w(x,t) = w(x+t,0)$ for $(x,t)\in \wh U^-$.
The set $\wh U^-$ is a liquid zone, i.e. $v(x,t)>w(x,t)$ for all $(x,t)\in \wh U^-$. The function $v$ propagates in $\wh U^-$ along characteristics with the slope 1 and $w$'s characteristics have slope $-1$.

For $(x,t)\in \wh U^+$, we find $s\leq t$ such that $(x,s) \in \calF$ and we let $v(x,t)=w(x,t) = v(x,s) =w(x,s)$. 
The set $\wh U^+$ is a frozen zone, i.e. $v(x,t)=w(x,t)$ for all $(x,t)\in \wh U^+$. The functions $v$ and $w$ propagate in $\wh U^+$ along characteristics with the slope 0.

Since $v(z,0) = w(z,0)$, the right endpoint of the  freezing curve is $(x,0)$. Hence, the set $\{(x,z-x): z-\eps/2\leq x\leq z\}$, the right part of the upper boundary of $\wh U$, belongs to the frozen zone $\wh U^+$. Note that $x\to v(x, z-x)$ is strictly increasing on $[z-\eps/2,z]$.

Recall notation and the claims made in Lemma \ref{m22.1} (ii). 

For $(x,t)\in U_r^-$ we let $v(x,t) = w(x,t)= v(x,z-x)= w(x,z-x)$.

For $(x,t)\in U_r^+$ we let $w(x,t) = w(x+t,0)$.

It follows from \eqref{m22.6} that if $(x,t)\in U_r^+$ then there is a unique point $(x_1,t_1) \in \calT \cup [z,z+\eps]\times \{0\}$ such that $(x,t) = (x_1 + s, t_1+s)$ for some $s\geq 0$. We let $v(x,t) = v(x_1,t_1)$. 

The set $\wh U^- \cup U_r^+$ consisting of two connected components is the liquid zone, 
 i.e. $v(x,t)>w(x,t)$ for all $(x,t)\in \wh U^-\cup U_r^+$.
The function $v$ propagates in $\wh U^- \cup U_r^+$ along characteristics with the slope 1 and $w$'s characteristics have slope $-1$.

The connected set $\wh U^+ \cup U_r^-$ is a frozen zone, i.e. $v(x,t)=w(x,t)$ for all $(x,t)\in \wh U^+ \cup U_r^-$. The functions $v$ and $w$ propagate in $\wh U^+$ along characteristics with the slope 0.

\medskip
(v)
Suppose that $v(z,0) = w(z,0)$, the functions $x\to v(x,0)$ and $x\to w(x,0)$ are strictly increasing on $[z, z+\eps]$ and $x\to v(x,0)$ is strictly decreasing on $[z-\eps, z]$
 This case is the ``mirror image'' of (iv). The solutions $v$ an $w$ can be defined in a way analogous to that in (iv).

\bigskip
The following cases (vi) (a)-(b) cannot occur because they would violate the assumption that $v(x,0)\geq w(x,0)$ for all $x$.

\smallskip
(vi)(a) $v(z,0) = w(z,0)$, the function $x\to v(x,0)$ is strictly increasing on $[z-\eps, z]$ and $y\to w(y,0)$ is strictly decreasing on $[z-\eps, z]$;

\smallskip
(vi)(b) $v(z,0) = w(z,0)$, the function $x\to v(x,0)$ is strictly decreasing on $[z, z+\eps]$ and $y\to w(y,0)$ is strictly increasing on $[z, z+\eps]$.

\bigskip
In the remaining cases we assume that $I=[a_1,a_2]$ and $z=a_1$ or $a_2$.

\bigskip

(vii)  Suppose that $z= a_1 $ and $v( a_1 , 0) = w( a_1 ,0)$. Assume that  the function  $x\to w(x,0)$ is strictly decreasing on $[ a_1 , a_1 + \eps]$.

For all $(x,t) \in U$, we let $w(x,t) = w(x+t,0)$.

For $(x,t) \in U$ such that $x\geq a_1 +t$ we let $v(x,t) = v(x-t,0)$.

For $(x,t) \in U$ such that $x< a_1 +t$ we let $v(x,t) =  v( a_1 ,t-(x-a_1))=w(a_1+ t-(x-a_1),0)$.

The set $ U$ is a liquid zone, i.e. $v(x,t)>w(x,t)$ for all $(x,t)\in \wh U$, $x>a_1$. The function $v$ propagates in $U$ along characteristics with the slope 1 and $w$'s characteristics have slope $-1$.

\bigskip

(viii) Suppose that $z= a_1 $, $v( a_1 ,0) = w( a_1 ,0)$ and the functions $x\to v(x,0)$ and $x\to w(x,0)$ are strictly increasing on $[ a_1 , a_1 + \eps]$.

If the definition of $\wh U$ in \eqref{m25.1} is replaced with 
\begin{align*}%\label{m25.1}
\wh U&=\{(x,t): t\geq 0, x\in( a_1 , a_1+\eps -t)\},
\end{align*}
then the rest of the statement of Lemma \ref{m22.1} (i) and its proof remain valid. Hence, we have a well defined freezing curve $\calF$ and it follows from our current assumptions that $\calF$ is non-empty, it cuts $U$ into two connected sets $\wh U^-$ and $\wh U^+$, and the left endpoint of $\calF$ is $(a_1,0)$.

We let
$v(x,t) = v(x-t,0)$  and $w(x,t) = w(x+t,0)$ for $(x,t)\in \wh U^-$.
The set $\wh U^-$ is a liquid zone, i.e. $v(x,t)>w(x,t)$ for all $(x,t)\in \wh U^-$. The function $v$ propagates in $\wh U^-$ along characteristics with the slope 1 and $w$'s characteristics have slope $-1$.

For $(x,t)\in \wh U^+$, we find $s\leq t$ such that $(x,s) \in \calF$ and we let $v(x,t)=w(x,t) = v(x,s) =w(x,s)$. 
The set $\wh U^+$ is a frozen zone, i.e. $v(x,t)=w(x,t)$ for all $(x,t)\in \wh U^+$. The functions $v$ and $w$ propagate in $\wh U^+$ along characteristics with the slope 0.

\bigskip

The following cases are symmetric to (vii)-(viii).

\bigskip

(ix) 
 Suppose that $z= a_2 $  and $v( a_2 , 0) = w( a_2 ,0)$. Assume that  the function  $x\to v(x,0)$ is strictly increasing on $[ a_2-\eps , a_2]$.

\bigskip

(x)  
Suppose that $z= a_2 $, $v( a_2 ,0) = w( a_2 ,0)$ and the functions $x\to v(x,0)$ and $x\to w(x,0)$ are strictly decreasing on $[ a_2-\eps , a_2]$.

The functions $v$ and $w$ can be defined in these cases in a way analogous to that in (vii)-(viii).

\bigskip
The above exhausts all possible cases. The proof is complete. 
\end{proof}

Recall the concept of a characteristic from Definition \ref{m26.3} and the notation from Lemma 
\ref{j6.5}.

\begin{lemma}\label{m26.1}
Suppose that $v(x,0)$ and $w(x,0)$ satisfy Assumption \ref{m29.2}
and
$v(x,t)$ and $w(x,t)$ are  solutions to \eqref{n10.4}-\eqref{a5.4} in $U$
constructed in Lemma \ref{j6.5}. Then

(i) For every $(x,t)\in U$ there is one or there are two
characteristics of $v$ connecting $(x,t)$ with a point
$(y,0)$ or points $(y_1,0)$ and $(y_2,0)$ for some
$y,y_1,y_2\in [z-\eps,z+\eps]$. The same holds for
$w$. Every characteristic of $v$ is a piecewise linear
function consisting of at most three intervals of
linearity, with the slopes 1 or 0. Every characteristic of
$w$ is a piecewise linear function consisting of at most
three intervals of linearity, with the slopes $-1$ or 0.
Different characteristics of $v$ do not intersect except
that two characteristics may share the same endpoint. The
same remark applies to characteristics of $w$.

(ii)
For every $t\geq 0$, the functions $x\to v(x,t)$ and $x\to w(x,t)$ are Lipschitz with the constant $\lambda $ on the interval where they are defined, i.e., for all $t,x$ and $y$ such that $(x,t),(y,t)\in U$,
\begin{align}\label{m26.2}
|v(x,t)-v(y,t)| \leq \lambda |x-y|,
\qquad
|w(x,t)-w(y,t)| \leq \lambda |x-y|.
\end{align}

(iii) For every $x\in[z-\eps,z+\eps]$, the functions $t\to v(x,t)$ and $t\to w(x,t)$ are Lipschitz with constant $\lambda$ on the interval where they are defined, i.e., for $s,t$ and $x$ such that $(x,s), (x,t)\in U$,
\begin{align}\label{m26.2}
|v(x,t)-v(x,s)| \leq \lambda |s-t|,
\qquad
|w(x,t)-w(x,s)| \leq \lambda |s-t|.
\end{align}

(iv) By assumption, $ v(x,0)$ has a local maximum at $x=z$ or no local maxima in $[z-\eps,z+\eps]$. If it has a local maximum at $x=z$ then for every $t\in[0,\eps/2]$, $x\to v(x,t)$ has a unique local maximum at $z+t$, and for every $t\in[\eps/2,\eps]$, $x\to v(x,t)$ does not have a local maximum in $(z-\eps +t, z+\eps-t)$ (in the interior of the interval where it is defined). 
If $ v(x,0)$ does not have a local maximum then for every $t\in[0,\eps]$, $x\to v(x,t)$ does not have a local maximum in $(z-\eps +t, z+\eps-t)$.
An analogous claim holds for $w$, except that a local minimum should be substituted for a local maximum. 

(v) For any $0\leq t <\eps$, $v(x,t)$ is not constant on any interval of positive length and neither is $w(x,t)$.

\end{lemma}

\begin{proof}
(i) This follows directly and easily from the construction of $v$ and $w$ presented in the proof of Lemma \ref{j6.5}. See Fig. \ref{fig19a} for the illustration of the most complicated case.

(ii) 
We will analyze only case (iv) in the proof of Lemma \ref{j6.5}. This is the most complicated case. The other cases can be treated in the same way.  Fig. \ref{fig19a} can help follow our argument.

Suppose that $(x,t),(y,t)\in U$, assume without loss of generality that $x<y$, and consider characteristics $\chi_v^1$  passing through $(x,t)$ and $\chi_v^2$  passing through $(y,t)$.

If $\chi_v^2(0)\geq z$ then the characteristic $\chi_v^2$ is a linear function with slope 1. The characteristic $\chi_v^1$ is a piecewise linear function with slopes 0 and 1 (the number of intervals of linearity can be 1, 2 or 3). 
Therefore, $|x-y| \geq |\chi_v^1(0) -\chi_v^2(0)|$ and, using \eqref{m29.3},
\begin{align*}
|v(x,t)-v(y,t)| = |v(\chi_v^1(0),0) -v(\chi_v^2(0),0)|
\leq \lambda |\chi_v^1(0) -\chi_v^2(0)| \leq \lambda |x-y|.
\end{align*}

Suppose that $\chi_v^1(0)< z$ and $\chi_v^2(0)< z$. 
Since $x<y$, we must have $\chi_v^1(0)< \chi_v^2(0)$.
There are $0\leq t_1\leq t_2 \leq t$ and $0\leq s_1 \leq s_2\leq t$ such that the slope of $\chi_v^1$ is 1 on $(0,t_1)$, 0 on $(t_1,t_2)$ and 1 on $(t_2,t)$. Similarly, the slope of $\chi_v^2$ is 1 on $(0,s_1)$, 0 on $(s_1,s_2)$ and 1 on $(s_2,t)$. Since the thawing curve $\calT$ satisfies \eqref{m22.6}, we must have $t_2 \geq s_2$.

If $t_1 \leq s_1$ then $\chi_v^2(s) - \chi_v^1(s) \geq \chi_v^2(0) - \chi_v^1(0)$ for all $s\in[0,t]$ and, using \eqref{m29.3},
\begin{align}\notag
|v(x,t)-v(y,t)| &= |v(\chi_v^1(0),0) -v(\chi_v^2(0),0)|
\leq \lambda |\chi_v^1(0) -\chi_v^2(0)| \leq \lambda |\chi_v^1(t) -\chi_v^2(t)|\\
& = \lambda |x-y|.\label{a20.1}
\end{align}

Suppose that $t_1 > s_1$.
Our conventions concerning naming of the times $t, t_1 $ and $ s_1$ imply that $t\geq t_1 > s_1$.
 Let $\wt\chi_w^1$ and $\wt\chi_w^2$ be the characteristics of $w$ going from $\chi_v^1(t_1)$ and 
$\chi_v^2(s_1)$ to $[z-\eps,z+\eps]\times \{0\}$. 
These characteristics are linear functions.

Since $t_1 > s_1$, we have
\begin{align*}
|\chi_v^1(s) -\chi_v^2(s)| &\geq |\chi_v^1(t_1) -\chi_v^2(t_1)|
\text{  for  } s\in[t_1,t],\\
|\chi_v^1(s) -\chi_v^2(s)| &\geq |\wt\chi_w^1(s) -\wt\chi_w^2(s)|
\text{  for  } s\in[s_1,t_1],\\
|\wt\chi_w^1(s) -\wt\chi_w^2(s)| &= |\wt\chi_w^1(0) -\wt\chi_w^2(0)|
\text{  for  } s\in[0,s_1].
\end{align*}
Therefore, 
\begin{align}\label{m27.2}
|x-y|=|\chi_v^1(t) -\chi_v^2(t)| \geq |\wt\chi_w^1(0) -\wt\chi_w^2(0)|.
\end{align}

The point $(\chi_v^1(t_1),t_1)$ is on the freezing curve so
\begin{align}\label{m27.1}
w(\wt\chi_w^1(t_1),t_1) = v(\wt\chi_w^1(t_1),t_1)=v(\chi_v^1(t_1),t_1).
\end{align}

Since $\chi_v^1(0)< \chi_v^2(0)< z$ and $r\to v(r,0)$ is increasing, $v(x,t)= v(\chi_v^1(0),0)< v(\chi_v^2(0),0)=v(y,t)$. 
Recall that, by definition, $\wt\chi^2_w(s_1) = \chi^2_v(s_1)$.
These remarks, \eqref{m27.2} and \eqref{m27.1} imply that
\begin{align*}
|v(x,t)-v(y,t)| &= v(y,t) - v(x,t) =  v(\chi_v^2(s_1),s_1) -v(\chi_v^1(t_1),t_1)\\
&= w(\wt\chi_w^2(s_1),s_1) -w(\wt\chi_w^1(t_1),t_1)
= w(\wt\chi_w^2(0),0) -w(\wt\chi_w^1(0),0)\\
&\leq |w(\wt\chi_w^2(0),0) -w(\wt\chi_w^1(0),0)|
\leq \lambda |\wt\chi_w^2(0) -\wt\chi_w^1(0)|  \leq \lambda |x-y|.
\end{align*}
This and \eqref{a20.1} prove the first of the two inequalities in \eqref{m26.2}.

Next consider the function $x\to w(x,t)$ for a fixed $t$. Let $x_0$ be such that $(x_0,t) \in \calT$. It will suffice to prove \eqref{m26.2} for $w$ separately on each of the two intervals: $[z-\eps+t, x_0]$ and $[x_0, z+\eps-t]$. 

Let  $\chi_w^1$ and $\chi_w^2$ be characteristics of $w$ connecting $(x,t)$ and $(y,t)$ with $[z-\eps,z+\eps]\times \{0\}$.

If $x_0\leq x < y$ then the functions  
$\chi_w^1$ and $\chi_w^2$ are linear with the slope $-1$. 
Therefore, $|x-y| = |\chi_w^1(0) -\chi_w^2(0)|$ and
\begin{align*}
|w(x,t)-w(y,t)| = |w(\chi_w^1(0),0) -w(\chi_w^2(0),0)|
\leq \lambda |\chi_w^1(0) -\chi_w^2(0)| = \lambda |x-y|.
\end{align*}

If $ x < y\leq x_0$ then
there are $0\leq t_1\leq t$ and $0\leq s_1 \leq s$ such that the slope of $\chi_w^1$ is $-1$ on $(0,t_1)$ and 0 on $(t_1,t)$. Similarly, the slope of $\chi_w^2$ is $-1$ on $(0,s_1)$ and 0 on $(s_1,s)$. 
This case can be further split into subcases when $t_1\geq s_1$ and $t_1< s_1$. One can now prove \eqref{m26.2} for $w$ just like it was done in the proof of \eqref{m26.2} for $v$.

(iii)
Suppose that $(x,s), (x,t)\in U$ and assume without loss of generality that $s< t$. Let 
$\chi_v$ be the characteristic of $v$ connecting $(x,t)$ with $[z-\eps,z+\eps]\times \{0\}$. Since the slope of $\chi_v$ is either 0 or 1, $|\chi_v(s) - x| \leq t-s$. We use \eqref{m26.2} to conclude that 
\begin{align*}%\label{m26.2}
|v(x,t)-v(x,s)| = |v(\chi_v(s),s)-v(x,s)| \leq \lambda |\chi_v(s) - x| \leq \lambda |s-t|.
\end{align*}
The proof for $w$ is analogous.

(iv) 
Suppose that  $x\to v(x,0)$ has no local maxima in $[z-\eps,z+\eps]$. The  characteristics of $v$ from different points in $U$ either do not intersect or one of them contains the other one. This implies that 
for every $t\in[0,\eps]$, $x\to v(x,t)$ does not have a local maximum in $(z-\eps +t, z+\eps-t)$.

Next suppose that $x\to v(x,0)$ has a local maximum at $z$. Then we have case (iv) of the proof of Lemma \ref{j6.5} with the extra assumption that $x\to v(x,0)$ is strictly decreasing on $[z,z+\eps]$. For $t\in [0,\eps/2]$ and $x\geq z+t$, $v(x,t) =v(x-t,0)$. Hence, $x\to v(x,t)$ is strictly decreasing for $x\geq z+t$. For $x\leq z+t$, the characteristic from $(x,t)$ has an endpoint at $(y,0)$ with $y\leq z$. Characteristics starting at different points $(x_1,t)$ and $(x_2,t)$ do not intersect. Since $x\to v(x,0)$ is strictly increasing on $[z-\eps,z]$, it follows that  $x\to v(x,t)$ is strictly increasing for $x\leq z+t$. We see that $v(x,t)$ has a unique local maximum at $z+t$.

The proof of (iv) for $w$ is similar.

(v) 
Once again, consider case (iv) of the proof of Lemma \ref{j6.5}.
Recall that $x\to v(x,0)$ and $x\to w(x,0)$ are strictly monotone
on $[z-\eps,z]$ and $[z,z+\eps]$. Note that different characteristics of $v$ are
either disjoint or one is contained in the other. Characteristics of $w$ are either disjoint, or one is contained in the other, or they have a common endpoint in $\calT$. All of this implies that for  any $0\leq t <\eps$, $v(x,t)$ is not constant on any interval of positive length and neither is $w(x,t)$.
\end{proof}

\begin{lemma}\label{j6.1}
If Assumption \ref{m29.2} is satisfied
then for every finite interval $[a_3,a_4]\subset I$\red{,} there exists $t_1>0$ such that there is a  solution to 
\eqref{n10.4}-\eqref{a5.4} in $[a_3,a_4]\times [0,t_1]$.

\end{lemma}

\begin{proof}
By compactness, there is a finite number of points $z_k\in[a_3-1,a_4+1]\cap I$  and $\eps_k>0$ satisfying Lemma \ref{j6.5} (a), such that 
$[a_3-1,a_4+1]\cap I\subset \bigcup_{k=1}^n (z_k-\eps_k/2, z_k+\eps_k/2)$.
Let $U_k$ be the triangle with parameters $z_k$ and $\eps_k$, defined as in Lemma \ref{j6.5} (b).
Since $U(z_k)$'s are overlapping triangles, there exists $t_1>0$ such that 
$[a_3,a_4]\times [0,t_1] \subset \bigcup_{k=1}^n  U_k$.

It is tedious but totally elementary to check that solutions constructed in 
the proof of Lemma \ref{j6.5} (b) agree on the intersections of $U_k$'s.
\end{proof}

\begin{theorem}\label{m30.1}
If Assumption \ref{m29.2} is satisfied then there exist jointly continuous functions $v(x,t)$ and $w(x,t)$ with $x\in I$, $t\geq 0$, such that \eqref{n10.8}-\eqref{a5.4} are satisfied, and for every open set $D\subset I \times(0,\infty)$, if $v(x,t)> w(x,t)$ for all $(x,t) \in D$ or $v(x,t)= w(x,t)$ for all $(x,t) \in D$ 
then \eqref{n10.4}-\eqref{n10.5} are satisfied in $D$.
\end{theorem}

\begin{proof}

First we will assume that $I = \R$ because this case requires ``localization.'' 
Given initial conditions $v(x,0)$ and $w(x,0)$, we define modified initial conditions
$v^b(x,0)$ and $w^b(x,0)$ for every $b>0$ by
\begin{align}\label{a13.1}
v^b(x,0)&=
\begin{cases}
v(x,0) & \text{  for  } -b \leq x \leq b,\\
v(-b,0) +\lambda|x-b| & \text{  for  }  x \leq -b,\\
v(b,0) +\lambda|x-b| & \text{  for  }  x \geq b,
\end{cases}
\\
w^b(x,0)&=
\begin{cases}
w(x,0) & \text{  for  } -b \leq x \leq b,\\
w(-b,0) -\lambda|x-b| & \text{  for  }  x \leq -b,\\
w(b,0) -\lambda|x-b| & \text{  for  }  x \geq b.
\end{cases}\label{a13.2}
\end{align}
It is easy to see that for each $b>0$, the initial conditions $v^b(x,0)$ and $w^b(x,0)$ satisfy Assumption \ref{m29.2}. By Lemma \ref{j6.1}, for some $t_b\in(0,b)$, there exist  solutions $v^b(x,t)$ and $w^b(x,t)$ to 
\eqref{n10.4}-\eqref{a5.4} in $[-4b,4b]\times [0,t_b]$.

By Assumption \ref{m29.2}, $v(x,0)$ and $w(x,0)$ are Lipschitz with constant $\lambda$ so
\begin{align*}
\sup_{-b\leq x\leq b} v^b(x,0) - \inf_{-b\leq x\leq b} v^b(x,0)\leq 2 \lambda b,\\
\sup_{-b\leq x\leq b} w^b(x,0) - \inf_{-b\leq x\leq b} w^b(x,0)\leq 2 \lambda b.
\end{align*}
This and \eqref{a13.1} imply that
\begin{align*}
\inf_{ x\leq-3b} v^b(x,0)&= 3 \lambda b + v^b(-b,0) 
\geq 2 \lambda b + \inf_{-b\leq x\leq b} v^b(x,0) 
\geq 2 \lambda b + \sup_{-b\leq x\leq b} v^b(x,0) -2\lambda b\\
&\geq 2 \lambda b + \sup_{-b\leq x\leq b} w^b(x,0) -2\lambda b
\geq 2 \lambda b + w^b(-b,0)-2\lambda b \\
&= 2 \lambda b + 2\lambda b+ \sup_{ x\leq -3b} w^b(x,0) -2\lambda b 
= \sup_{ x\leq -3b} w^b(x,0) +2\lambda b.
\end{align*}
Hence $v^b(x,t) = v^b(x-t,0)>w^b(x+t,0) = w^b(x,t)$ for $x\leq -3b$ and $t\in[0,b]$. For a similar reason, 
$v^b(x,t) = v^b(x-t,0)$ and $w^b(x,t) = w^b(x+t,0)$ for $x\geq 3b$ and $t\in[0,b]$.
Since $t_b < b$, there exist  solutions $v^b(x,t)$ and $w^b(x,t)$ to 
\eqref{n10.4}-\eqref{a5.4} in $\R\times [0,t_b]$.

It follows from our constructions that the solutions $v^b(x,t)$ and $w^b(x,t)$ satisfy conditions in Lemma \ref{m26.1} (i)-(iii) in $\R\times [0,t_b]$.

Let $N^b_t$ be the number of local maxima of $x\to v^b(x,t)$ and local minima of $x\to w^b(x,t)$. In view of our construction of $v^b(x,0)$ and $w^b(x,0)$, $N^b_0$ is finite. Lemma \ref{m26.1} (iv) implies that $t\to N^b_t$ is non-increasing on $[0,t_b]$.

Let $T$ be the supremum of $u$  
such that there exist solutions to 
\eqref{n10.4}-\eqref{a5.4} in $\R\times [0,u]$
satisfying conditions in Lemma \ref{m26.1} (i)-(iii) and such that $t\to N^b_t$ is non-increasing on $[0,T]$.
Note that $T\geq t_b >0$.

Suppose that $T<\infty$.
For every $t<T$, the function $x\to v^b(x,t)$ is Lipschitz continuous with constant $\lambda$ by Lemma \ref{m26.1} (ii).
By Lemma \ref{m26.1} (ii)-(iii), the function $(x,t)\to v^b(x,t)$ is Lipschitz continuous with constant $2\lambda$ on $\R\times[0,T)$. This  implies that $v^b(x,T) := \lim_{t\uparrow T} v^b(x,t)$ exists and is a Lipschitz function with constant $\lambda$.
Similarly, $w^b(x,T) := \lim_{t\uparrow T} w^b(x,t)$ exists and is a Lipschitz function with constant $\lambda$.

We have shown that $N^b_t$ is bounded on $[0,T)$. The limits $v^b(x,T) := \lim_{t\uparrow T} v^b(x,t)$ and $w^b(x,T) := \lim_{t\uparrow T} w^b(x,t)$ are uniform so $N^b_T \leq \lim_{t\uparrow T} N^b_t < \infty$.
Since $N^b_t$ is bounded on $[0,T]$, the total number of local maxima of $x\to v^b(x,T)$ and local minima of $x\to w^b(x,T)$ is finite. Hence, the total number of local extrema of $x\to v^b(x,T)$ and  $x\to w^b(x,T)$ is finite.

By assumption, $v^b(x,0)$ and $w^b(x,0)$ can take a given value only at a finite number of $x$.
Neither of the functions $x\to v^b(x,T)$ and  $x\to w^b(x,T)$ is constant on any interval of positive length because otherwise we would have an arbitrarily large number of characteristics with ends at the same point of $\R\times[0,\infty)$.

We see that $x\to v^b(x,T)$ and  $x\to w^b(x,T)$ satisfy Assumption \ref{m29.2}. Hence, for some $\delta >0$, we can find solutions $v^b(x,t)$ and $w^b(x,t)$ for $x\in \R$ and $t\in[T,T+\delta]$. This contradicts the definition of $T$ so we conclude that $T=\infty$.

We have proved the theorem for the set of initial conditions $v^b(x,0)$ and $w^b(x,0)$ for any $b>0$. If $0<b<c$ then
the solutions $v^b(x,t)$ and $w^b(x,t)$  agree with $v^c(x,t)$ and $w^c(x,t)$
on the triangle $\{(x,t): -b+t<x<b-t\}$. Hence, we can define $v(x,t)=\lim_{b\to\infty} v^b(x,t)$ and $w(x,t)=\lim_{b\to\infty} w^b(x,t)$ for all $x\in\R$ and $t\geq 0$. It is easy to see that thus defined $v(x,t)$ and $w(x,t)$ satisfy the theorem.

The case when $I=[a_1,a_2]$ can be dealt with in a similar but simpler way because we do not need ``localization.''
\end{proof}

\begin{remark}
Recall remarks from Section \ref{sec:char} about points where four characteristics meet, illustrated by the top transition point in Fig. \ref{fig15}. This type of transition does not appear in Lemmas \ref{m22.1} and \ref{j6.5}. One can explain its absence in the following way. 
A point of this type can appear at time $T<\infty$ in the proof of Theorem \ref{m30.1}. But the proof shows that $x\to v^b(x,T)$ and  $x\to w^b(x,T)$ satisfy Assumption \ref{m29.2} so there is no need to provide a separate discussion of points where four characteristics meet.
\end{remark}

\section{General initial conditions}\label{sec:gen}

The goal of this section is to construct solutions to \eqref{n10.4}-\eqref{a5.4} for any continuous initial conditions. We will start with some definitions that apply even to arbitrary measurable initial conditions $v(x,0)$ and $w(x,0)$.

We will use the convention that $\inf \emptyset = -\infty$ and $\sup\emptyset=\infty$.

Suppose that $b\in \R$, $v(x,0)$ and $w(x,0)$ are measurable functions and $v(x,0)\geq w(x,0)$ for all $x$. 

If  $I=\R$ then  for $x\in \R$, $t\geq 0$ and  $b\in \R$, let
\begin{align} \label{j17.1a} 
\alpha(v,b,x) &= \inf\left\{y\in I:
y\geq x, \int_x^y \left(\bone_{\{v(z,0) \leq b\}} - \bone_{\{w(z,0) \geq b\}}\right)dz <0
\right\},\\
\alpha(w,b,x) &= \sup\left\{y\in I:
y\leq x, \int_y^x \left(\bone_{\{v(z,0) \leq b\}} - \bone_{\{w(z,0) \geq b\}}\right)dz >0
\right\}, \label{j17.2a} \\
 \label{j17.5a} 
A(v,b,t)&= \{z\in I: z = x+t, v(x,0)\leq b, t\leq (\alpha(v,b,x)-x)/2\} ,\\
A(w,b,t)&= \{z\in I: z = x-t, w(x,0)\geq b, t\leq (x-\alpha(w,b,x))/2\}. \label{j17.6a} 
\end{align}

If $I=[a_1,a_2]$ then for $x\in \R$, $t\geq 0$ and  $b\in \R$, let
\begin{align} \label{j17.3} 
&\alpha(v,b,x)\\
 &= \inf\left\{
y\geq x: \int_x^y \left(\bone_{\{z\in I,v(z,0) \leq b\}} + \bone_{\{z\leq a_1\}}
- \bone_{\{z\in I,w(z,0) \geq b\}}- \bone_{\{z\geq a_2\}}\right)dz <0
\right\},\notag\\
&\alpha(w,b,x) \label{j17.4}  \\
 &= \sup\left\{
y\leq x: \int_y^x \left(\bone_{\{z\in I,v(z,0) \leq b\}} +\bone_{\{z\leq a_1\}}
- \bone_{\{z\in I,w(z,0) \geq b\}}
-\bone_{\{z\geq a_2\}}\right)dz >0
\right\},  \notag
\end{align}
\begin{align}
\label{j18.5} 
A(v,b,t)&= \{z\in I:z=x+t, t\leq (\alpha(v,b,x)-x)/2 \text{  and  }( v(x,0)\leq b  \text{  or  }
x \leq a_1)\} , \\
A(w,b,t)&= \{z\in I:z=x-t, t\leq (x-\alpha(w,b,x))/2\label{j18.6} 
 \text{  and  }( w(x,0)\geq b  \text{  or  } x \geq a_2)\} . 
\end{align}

Note that if $x,y\in I$ then  $y=\alpha(v,b,x)$ if and only if $x=\alpha(w,b,y)$.

Let
\begin{align}
v_*(x,t) & = \inf\{b\in\R: x\in A(v,b,t)\}, \qquad x\in I, t\geq 0, \label{j17.7a} \\
w_*(x,t) & = \sup\{b\in\R: x\in A(w,b,t)\}, \qquad x\in I, t\geq 0. \label{j17.8a} 
\end{align}

\begin{remark}\label{a24.1}
(i)
Formulas \eqref{j17.1a}-\eqref{j17.8a} uniquely define measurable functions $v_*(x,t)$ and $w_*(x,t)$ for any measurable $v(x,0)$ and $w(x,0)$.

(ii)
It is easy to check that $b\to \alpha(v,b,x)$ is non-decreasing and $b\to \alpha(w,b,x)$ is non-increasing. This implies that $b\to A(v,b,t)$ is non-decreasing and $b\to A(w,b,t)$ is non-increasing.

(iii)
In the case $I=[a_1,a_2]$ one can extend $v(x,0)$ and $w(x,0)$ to the whole real line by declaring that $v(x,0)=w(x,0) = -\infty$ for $x< a_1$ and  $v(x,0)=w(x,0) = \infty$ for $x> a_2$. Then  \eqref{j17.1a}-\eqref{j17.6a} applied to thus extended initial conditions yield the same formulas as in \eqref{j17.3}-\eqref{j18.6}.

\end{remark}

\begin{lemma}\label{a22.6}
If $v$ and $w$ are the solutions to \eqref{n10.4}-\eqref{a5.4} in $I\times [0,\infty)$ constructed in Theorem \ref{m30.1} under Assumption \ref{m29.2} then
 $v_*(x,t)=v(x,t)$ and $w_*(x,t)=w(x,t)$ for all $x\in I$ and $t\geq 0$.
\end{lemma}

\begin{proof}

\emph{Step 1}.
Let $v$ and $w$ be the solutions to \eqref{n10.4}-\eqref{a5.4} in $I\times [0,\infty)$ constructed in Theorem \ref{m30.1} under Assumption \ref{m29.2}.
We define sublevels of $v$ and superlevels of $w$  by
\begin{align}\label{j21.1}
\calA(v,b,t)&= \{x\in I: v(x,t) \leq b\}, \\
\calA(w,b,t)&= \{x\in I: w(x,t)\geq b\}. \label{j21.2}
\end{align}
Then
\begin{align}
v(x,t) & = \inf\{b\in\R: x\in \calA(v,b,t)\}, \qquad x\in I, t\geq 0, \label{j17.7b} \\
w(x,t) & = \sup\{b\in\R: x\in \calA(w,b,t)\}, \qquad x\in I, t\geq 0. \label{j17.8b} 
\end{align}

By comparing \eqref{j17.5a}-\eqref{j17.6a} and \eqref{j18.5}-\eqref{j17.8a} to \eqref{j21.1}-\eqref{j17.8b}, we see that it will suffice to show that $\calA(v,b,t)= A(v,b,t)$ and $\calA(w,b,t)= A(w,b,t)$ for all $b\in \R$ and $t\geq 0$.
The proof will be based on the following heuristic description of the evolution of $t\to \calA(v,b,t)$ and $t\to \calA(w,b,t)$. The set $\calA(v,b,t)$ moves to the right at the speed 1, $\calA(w,b,t)$ moves to the left at the same speed, and the two sets annihilate each other wherever they meet. We will make this precise in (I)-(VIII) below.

Assumption \ref{m29.2} and Lemma \ref{m26.1} (iv) imply that  for every $t$, $v(x,t)$ and $w(x,t)$ have only a finite number of extrema on every finite interval. Thus the intersection of  $ \calA(v,b,t)$  with any finite interval consists of a finite number of disjoint closed intervals and the same can be said about $\calA(w,b,t)$.

Since $v(x,t)\geq w(x,t)$ for all $x$ and $t$, if
$x\in \calA(v,b,t)\cap \calA(w,b,t)$
then $v(x,t) = w(x,t)$. We claim that 
$ \calA(v,b,t)\cap \calA(w,b,t)$
can contain only some endpoints of the intervals comprising 
$ \calA(v,b,t)$ and $ \calA(w,b,t)$. Otherwise, it would contain an interval $J$ of strictly positive length. We would have $v(x,t) = w(x,t)=b$ for all $x\in J$. This is impossible because for a fixed $t$, $v(x,t)$ and $w(x,t)$ are not constant over an interval of positive length, by Lemma \ref{m26.1} (v).

Suppose that  $b\in \R$ and $z$ is in the interior of $I$.
The above observations imply that there exists $\eps>0$  such that 
conditions (1)-(2) of Lemma \ref{j6.5} are satisfied on the interval $[z-\eps , z+\eps ]$ and, moreover,
one of the conditions (I)-(V) listed below holds.
If $I=[a_1,a_2]$ and $z= a_1$ then one of the conditions (VI)-(VIII) holds.

We will analyze the time evolution of $ \calA(v,b,t)$ and $ \calA(w,b,t)$ on a time interval $[0, \eps]$. 
For an interval $[z-\eps , z+\eps ]$ and $t\geq  0$ we will write
$J_{t}= [z-\eps +t, z+\eps -t]$.

\medskip
\emph{Step 2}.
We will  describe  the time evolution of $ \calA(v,b,t)$ and $ \calA(w,b,t)$ depending on the form of intersections of these sets with $[z-\eps , z+\eps ]$.
We will  prove that our representation is correct in the next step.

\smallskip
(I) 
Suppose that  one of the following conditions holds, 
\begin{align}
&\calA(v,b,0)\cap [z-\eps ,z+\eps ]=[z,z+\eps], 
\  \calA(w,b,0)\cap [z-\eps ,z+\eps ]=[z-\eps,z] , \label{a11.5}\\
&\calA(v,b,0)\cap [z-\eps ,z+\eps ]=[z,z+\eps], 
\  \calA(w,b,0)\cap [z-\eps ,z+\eps ]=\{z\} , \label{a11.10} \\
&\calA(w,b,0)\cap [z-\eps ,z+\eps ]=[z-\eps,z], 
\  \calA(v,b,0)\cap [z-\eps ,z+\eps ]=\{z\} .\label{a11.11} 
\end{align}
 Then for $t\in (0,  \eps ]$, 
\begin{align}\label{j22.2}
\calA(v,b, t) \cap J_t& =(\calA(v,b, 0)+ t ) \cap J_t ,\\
\calA(w,b, t) \cap J_t &= (\calA(w,b, 0)- t ) \cap  J_t .\label{j22.5}
\end{align}

\smallskip
(II) 
Suppose that  one of the following conditions holds, 
\begin{align}
&\calA(v,b,0)\cap [z-\eps ,z+\eps ]=\emptyset, \label{a11.1} \\
&\calA(w,b,0)\cap [z-\eps ,z+\eps ]=[z,z+\eps], 
\  \calA(v,b,0)\cap [z-\eps ,z+\eps ]=\{z\} , \label{a11.4} \\
&\calA(w,b,0)\cap [z-\eps ,z+\eps ]=[z-\eps,z+\eps], 
\  \calA(v,b,0)\cap [z-\eps ,z+\eps ]=\{z\} . \label{a11.7} 
\end{align}
 Then for $t\in (0,  \eps ]$, 
\begin{align}\label{j22.2a}
\calA(v,b, t) \cap J_t& =\emptyset ,\\
\calA(w,b, t) \cap J_t &= (\calA(w,b, 0)- t ) \cap  J_t .\label{j22.5a}
\end{align}

\smallskip
(III) 
Suppose that  one of the following conditions holds, 
\begin{align}
&\calA(w,b,0)\cap [z-\eps ,z+\eps ]=\emptyset, \label{a11.2} \\
&\calA(v,b,0)\cap [z-\eps ,z+\eps ]=[z-\eps,z], 
\  \calA(w,b,0)\cap [z-\eps ,z+\eps ]=\{z\} , \label{a11.3} \\
&\calA(v,b,0)\cap [z-\eps ,z+\eps ]=[z-\eps,z+\eps], 
\  \calA(w,b,0)\cap [z-\eps ,z+\eps ]=\{z\} .\label{a11.6} 
\end{align}
 Then for $t\in (0,  \eps ]$, 
\begin{align}\label{j22.2b}
\calA(v,b, t) \cap J_t& =(\calA(v,b, 0)+ t ) \cap J_t ,\\
\calA(w,b, t) \cap J_t &= \emptyset .\label{j22.5b}
\end{align}

\smallskip
(IV)
Suppose that  $ \calA(v,b,0)\cap [z-\eps ,z+\eps ] = [z-\eps ,z]$ and $ \calA(w,b,0)\cap [z-\eps ,z+\eps ] = [z,z+\eps ]$.
 Then for $t\in (0,  \eps ]$, 
\begin{align}\label{j20.2}
\calA(v,b, t) \cap  J_t 
&= [z-\eps , z]\cap J_t,\\
\calA(w,b, t) \cap  J_t 
&= [z, z+\eps ]\cap J_t.\label{j20.3}
\end{align}

\smallskip
(V)
Suppose that  $ \calA(v,b,0)\cap [z-\eps ,z+\eps ] = \calA(w,b,0)\cap [z-\eps ,z+\eps ] = \{z\}$.
 Then for $t\in (0,  \eps ]$, 
\begin{align}\label{j22.10}
\calA(v,b, t) \cap  J_t 
&= \{z+ t \}\cap J_t,\\
\calA(w,b, t) \cap  J_t 
&= \{z- t \}\cap J_t.\label{j22.11}
\end{align}

\smallskip

Recall that cases (I)-(V) apply when either $I=\R$ or $I=[a_1,a_2]$.
The  discussion in (VI)-(VIII) below is concerned with $I=[a_1,a_2]$ and the case when 
$[z-\eps ,z+\eps ]\cap I=[z,z+\eps ] = [a_1,a_1+\eps]$.
We have
$a_1\in \calA(v,b,0)\cap [a_1,a_1+\eps ]$ or $a_1\in \calA(w,b,0)\cap [a_1,a_1+\eps ] $
because $v(a_1,0)=w(a_1,0)$. 
Let $K_t= [a_1, a_1+\eps - t ]$.
For a given $b\in\R$  there exists $\eps >0$ such that one  of the following cases (IV)-(VI) holds.

\smallskip
(VI) Suppose that 
$ \calA(v,b,0)\cap [a_1,a_1+\eps ]= \calA(w,b,0)\cap [a_1,a_1+\eps ]=\{a_1\} $.
 Then for $t\in (0,  \eps ]$, 
\begin{align}\label{a10.1}
\calA(v,b, t) \cap K_t& =[a_1,a_1+ t ] \cap K_t ,\\
\calA(w,b, t) \cap K_t &= \emptyset .\label{a10.2}
\end{align}

\smallskip
(VII) Suppose that 
$ \calA(v,b,0)\cap [a_1,a_1+\eps ] =[a_1,a_1+\eps ]$ and $ \calA(w,b,0)\cap (a_1,a_1+\eps ]=\emptyset $.
 Then for $t\in (0,  \eps ]$, 
\begin{align}\label{a10.3}
\calA(v,b, t) \cap K_t& = K_t ,\\
\calA(w,b, t) \cap K_t &= \emptyset .\label{a10.4}
\end{align}

\smallskip
(VIII) Suppose that 
$ \calA(v,b,0)\cap (a_1,a_1+\eps ]=\emptyset$ and $ \calA(w,b,0)\cap [a_1,a_1+\eps ]= [a_1,a_1+\eps ] $.
 Then for $t\in (0,  \eps ]$, 
\begin{align}\label{a10.5}
\calA(v,b, t) \cap K_t& =\emptyset ,\\
\calA(w,b, t) \cap K_t &= K_t .\label{a10.6}
\end{align}

The case when $[z-\eps ,z+\eps ]\cap I=[a_2-\eps , a_2]$ is symmetric to $[z-\eps ,z+\eps ]\cap I=[a_1,a_1+\eps ]$ so it will be omitted.

\medskip

\emph{Step 3}.
Now we will prove \eqref{j22.2}-\eqref{a10.6}. First we will verify that they hold in every triangle $U$ of the type considered in Lemma \ref{j6.5}.
We will refer to cases listed in part (b) of the proof of Lemma \ref{j6.5}.
Note that the fact that we choose $\eps>0$ so small that one of conditions (I)-(VI) is satisfied adds extra assumptions to the cases in  Lemma \ref{j6.5}, greatly simplifying them.

In case (i), if the inequality between the infimum and supremum is strict then either \eqref{a11.1} or \eqref{a11.2} holds. Then it is clear that the definition of $v$ and $w$ in case (i) agrees with \eqref{j22.2a}-\eqref{j22.5a} or \eqref{j22.2b}-\eqref{j22.5b}.

If the infimum is equal to the supremum in case (i) then this corresponds to  case (V). Once again, it is easy to check that the definition of $v$ and $w$ in case (i) agrees with \eqref{j22.10}-\eqref{j22.11}.

In cases  (ii) (a) and (b), one of the conditions \eqref{a11.5}-\eqref{a11.11} is satisfied. The definition of $v$ and $w$ in cases  (ii) (a) and (b) agrees with \eqref{j22.2}-\eqref{j22.5}.

Case (ii) (c)  corresponds to  case (V). The definition of $v$ and $w$ in case (ii) (c) agrees with \eqref{j22.10}-\eqref{j22.11}.

Case (iii) corresponds to (IV). The characteristics of $v$ and $w$ emanating from $(z,0)$ are identical and have slope 0. Therefore, \eqref{j20.2}-\eqref{j20.3} are satisfied.

In case (iv), one of the conditions \eqref{a11.1}, \eqref{a11.2},\eqref{a11.3} or \eqref{a11.6} is satisfied. 

If \eqref{a11.1} holds then  all characteristics of $w$ where $w$ takes the value $b$ have to be straight lines with slope $-1$ so \eqref{j22.2a}-\eqref{j22.5a} is true.

If one of the conditions  \eqref{a11.2},\eqref{a11.3} or \eqref{a11.6} is satisfied then the construction in part (iv) of the proof of Lemma \ref{j6.5} shows that all characteristics of $v$ where $v$ takes the value $b$ have to be straight lines with slope $1$ so \eqref{j22.2b}-\eqref{j22.5b} is true.

Case (v) is a symmetric version of (iv).

Case (vii)  can occur simultaneously with any of the cases (VI)-(VIII). 

Consider case (vii) combined with (VI). Characteristics of $v$ with endpoints 
$(x,0)$ with $x>a_1$ will carry values greater than $b$. Those with endpoints
$(a_1,t)$ with $t>0$ will carry values less than $b$ because  $(a_1,t)$
is the endpoint of a characteristic of $w$ starting from a point $(y,0)$ with $y>a_1$.
This agrees with \eqref{a10.1}-\eqref{a10.2}.

Consider case (vii) combined with (VII). In this case there will be no characteristics of $w$ with values greater than or equal to $b$ so all characteristics of $v$ will carry values less than or equal to $b$. Hence,  \eqref{a10.3}-\eqref{a10.4} will hold.

Consider case (vii) combined with (VIII). Characteristics of $v$ with endpoints 
$(x,0)$ with $x>a_1$ will carry values greater than $b$. Those with an endpoint
$(a_1,t)$ with $t>0$ will carry values greater than $b$ because  $(a_1,t)$
is the endpoint of a characteristic of $w$ starting from a point $(y,0)$ with $y>a_1$.
This agrees with \eqref{a10.5}-\eqref{a10.6}.

Case (viii)  can occur simultaneously with  cases (VII)-(VIII). 

Consider case (viii) combined with (VII). In this case there will be no characteristics of $w$ with values greater than or equal to $b$ so all characteristics of $v$ will carry values less than or equal to $b$. Hence,  \eqref{a10.3}-\eqref{a10.4} will hold.

Consider case (viii) combined with (VIII).  Characteristics of $v$ with endpoints 
$(x,0)$ with $x>a_1$ will carry values greater than $b$. Those with an endpoint
$(a_1,t)$ with $t>0$ will carry values greater than $b$ because  $(a_1,t)$
is the endpoint of a characteristic of $w$ starting from a point $(y,0)$ with $y>a_1$.
This agrees with \eqref{a10.5}-\eqref{a10.6}.

Cases (ix) and (x) are symmetric with (vii) and (viii).

\medskip

\emph{Step 4}.
Next we will prove that (I)-(VIII) hold when 
$ \calA(v,b,t)$ and $ \calA(w,b,t)$ are replaced with $ A(v,b,t)$ and $ A(w,b,t)$.

It follows  from \eqref{j17.5a}-\eqref{j17.6a}, \eqref{j18.5}-\eqref{j18.6} and \eqref{j21.1}-\eqref{j21.2} that $ \calA(v,b,0)= A(v,b,0)$ and $ \calA(w,b,0)= A(w,b,0)$.

For a set $C\subset \R$, let $\Leb(C)$ denote its Lebesgue measure.  
Formulas \eqref{j17.1a}-\eqref{j17.6a} can be interpreted as follows. For every point $x\in A(v,b,0)$ we try to find a matching point $y= \alpha(v,b,x)\in  A(w,b,0)$ such that $ \Leb( A(v,b,0) \cap [x,y] )= \Leb(A(w,b,0)\cap [x,y])$. If such a point $y$ exists then $x= \alpha (w,b,y)$, except for a set of $x$ and $y$ of measure zero consisting of some endpoints of intervals comprising
$  A(v,b,0) $ and $ A(w,b,0)$.
Suppose that  $x,y\in I$.
We let $x$ and $y$ travel with speed 1 towards each other until they meet and annihilate each other. At time $t$, we include the transported  $x$ (i.e., $x+t$) into $  A(v,b,t) $ and we include the transported $y$ (i.e., $y-t$) into $ A(w,b,t)$---this is the meaning of \eqref{j17.5a}-\eqref{j17.6a}.
If $I=\R$ and there is no $y$ corresponding to $x$ then $x$ keeps moving forever and $x+t$ is included in  $  A(v,b,t) $ for all $t\geq 0$. It is easy to check that this interpretation agrees the description of the evolution of $ \calA(v,b,t)$ and $ \calA(w,b,t)$ given in (I)-(V).

Now suppose that $I=[a_1,a_2]$. 

In case (VI) definition 
\eqref{j18.6} implies that 
$A(w,b, t) \cap K_t =\emptyset $ for $t>0$ because 
$ A(w,b,0)\cap (a_1,a_1+\eps ]=\emptyset $.

Definition 
\eqref{j17.3} and the assumption that 
$ A(w,b,0)\cap (a_1,a_1+\eps ]=\{a_1\} $ imply that $\alpha(v,b,x)\geq a_1+\eps$  
for $x\in (a_1,a_1+\eps ]$. This, the assumption that $ A(v,b,0)\cap (a_1,a_1+\eps ]=\{a_1\} $ and definition
\eqref{j18.5} imply that 
$(x,t)\notin A(v,b, t) $ for $t\in( 0, x-a_1]$ and $x\in K_t$.
If $t> x-a_1$ then $x=y+t < a_1+t$ for some $y< a_1$. 
According to  \eqref{j18.5}, 
$(x,t)\in A(v,b, t) $ for $t> x-a_1$ and $x\in K_t$.
Hence \eqref{a10.1}-\eqref{a10.2} are satisfied by 
$ A(v,b,t)$ and $ A(w,b,t)$ in place of $ \calA(v,b,t)$ and $ \calA(w,b,t)$.

In case (VII), 
definition \eqref{j18.6} implies that 
$A(w,b, t) \cap K_t = \emptyset $ for $t>0$ because 
$ A(w,b,0)\cap (a_1,a_1+\eps ]=\emptyset $.
Definition 
\eqref{j17.3} and the assumption that 
$ A(w,b,0)\cap (a_1,a_1+\eps ]=\emptyset $ imply that $\alpha(v,b,x)\geq a_1+\eps$  
for $x\in (a_1,a_1+\eps ]$. This,
definition \eqref{j18.5} and the assumption that 
$ A(v,b,0)\cap [a_1,a_1+\eps ]=[a_1,a_1+\eps ] $ imply that for every $x\in K_t$, $(x,t) \in A(v,b,t)$, no matter whether $x-a_1 < t$ or not.
Hence \eqref{a10.3}-\eqref{a10.4} are satisfied by 
$ A(v,b,t)$ and $ A(w,b,t)$ in place of $ \calA(v,b,t)$ and $ \calA(w,b,t)$.

In case (VIII), definition 
\eqref{j17.4} and the assumption that 
$ A(w,b,0)\cap (a_1,a_1+\eps ]=(a_1,a_1+\eps ] $ imply that $\alpha(w,b,x)= a_1-(x-a_1)$ 
for $x\in (a_1,a_1+\eps ]$. This, the assumption that $ A(w,b,0)\cap (a_1,a_1+\eps ]=(a_1,a_1+\eps ] $ and definition
\eqref{j18.6} imply that 
$A(w,b, t) \cap K_t = K_t $ for $t>0$.
Definition 
\eqref{j17.3} and the assumption that 
$ A(v,b,0)\cap (a_1,a_1+\eps ]=\emptyset $ imply that $\alpha(v,b,x)= x$ 
for $x\in (a_1,a_1+\eps ]$. This and definition
\eqref{j18.5} imply that 
$A(v,b, t) \cap K_t = \emptyset $ for $t>0$.
Hence \eqref{a10.5}-\eqref{a10.6} are satisfied by 
$ A(v,b,t)$ and $ A(w,b,t)$ in place of $ \calA(v,b,t)$ and $ \calA(w,b,t)$.

\medskip

\emph{Step 5}.
Suppose that $I=\R$.
For any $x_1< y_1$, the function
$b\to \Leb\{z\in [x_1,y_1]: v(z,0)\leq b\}$ is non-decreasing and $b\to \Leb\{z\in [x_1,y_1]: w(z,0)\geq b\}$
is non-increasing. 
Hence, definition \eqref{j17.1a} shows that $b\to \alpha(v,b,x)$ is non-decreasing. 

We will argue that if $y\in(x, \alpha(v,b,x))$ then
\begin{align}\label{a18.1}
&(y, \alpha(v,b,y))\subset(x, \alpha(v,b,x)),\\
\label{a17.1}
&(\alpha(v,b,x)-x)/2 
\geq (\alpha(v,b,y)-y)/2 .
\end{align}
By \eqref{j17.1a}, there are $\delta_n \downarrow 0$ such that
\begin{align*} %\label{j17.1a} 
\int_x^{\alpha(v,b,x)+\delta_n} \left(\bone_{\{v(z,0) \leq b\}} - \bone_{\{w(z,0) \geq b\}}\right)dz &<0,\\
 \int_x^y \left(\bone_{\{v(z,0) \leq b\}} - \bone_{\{w(z,0) \geq b\}}\right)dz &\geq 0,
\end{align*}
so
\begin{align*} %\label{j17.1a} 
\int_y^{\alpha(v,b,x)+\delta_n} \left(\bone_{\{v(z,0) \leq b\}} - \bone_{\{w(z,0) \geq b\}}\right)dz &< 0.
\end{align*}
This implies that $\alpha(v,b,y)\leq \alpha(v,b,x)$, proving \eqref{a18.1} and, therefore, \eqref{a17.1}.

\medskip

Fix some $x\in I$  and for $b\in \R$, $y\in I$, $y\geq x$, let
\begin{align*}
g_b(y)=\int_x^y \left(\bone_{\{v(z,0) \leq b\}} - \bone_{\{w(z,0) \geq b\}}\right)dz.
\end{align*}
The graph $\Gamma_b$ of $g_b$ is a continuous polygonal line consisting a finite line segments on each finite interval because $v(x,0)$ and $w(x,0)$ have a finite number of local extrema on every finite interval. The slopes of line segments in $\Gamma_b$ can be $-1,1$ or $0$.

Let $y_1 = \alpha(v,b,x)$. It follows from \eqref{j17.1a} that $g(y) \geq 0$ for $y\in[x,y_1]$ and the slope of $\Gamma_b$ is $-1$ on an interval $[y_1,y_2]$, with $y_2>y_1$.

When $b$ increases, intervals where $g_b$ is increasing expand at both ends (except at $x$) and intervals where $g_b$ is decreasing are shrinking at both ends, with both types of evolution being continuous due to strict monotonicity of $v(x,0)$ and $w(x,0)$ between the local extrema.

We will argue that if $b_1>b$  then one of the following holds.

(a) For some $\delta_1>0$, all $\delta\in(0,\delta_1)$ and all $y\in(x+\delta, y_1)$,  $g_{b_1}(y) \geq g_{b_1}(x+\delta) $.

(b) $g_{b_1}(y)$ is non-decreasing for $y\in[x,y_1]$. 

Suppose that $g_b(y)>0$ for some $y \in[x,y_1]$. Then there are $x\leq x_1\leq x_2$ such that $g'_b(y)=0$  
for $y\in(x,x_1)$ and $g'_b(y)=1$ for $y\in(x_1,x_2)$. Suppose that $x_2$ is the maximum $x_2$ with this property. Then there are $x\leq x_3 \leq x_1$ and $x_4 > x_2$ such that 
 $g'_{b_1}(y)=0$  for $y\in(x,x_3)$ and $g'_{b_1}(y)=1$ for $y\in(x_3,x_4)$.
Our earlier remarks imply that $g_{b_1}(y) \geq g_b(y) + (x_4-x_2)\geq (x_4-x_2)$ for $y\in  (x_4, y_1]$. Since $g'_{b_1}(y)\leq 1$ for $y\in(x,x_4)$,
we have for  $\eta_2 = \min(x_4-x, \eta_1)$, all $\delta\in(0,\eta_2)$ and $y\in  (x+\delta, y_1]$,
\begin{align*}
g_{b_1}(x+ \delta) \leq g_{b_1}(x) +\delta =\delta  \leq 
x_4-x_2 \leq g_{b_1}(y).
\end{align*}
This proves (a).

If $g_b(y)=0$ for all $y \in[x,y_1]$ then $g'_{b_1}(y) \geq 0$ for $y\in[x,y_1]$ in view of our earlier remarks on the evolution of $b\to g_b$. This proves (b).

If (a) or (b) holds then for some $\delta_1>0$ and all $\delta\in(0,\delta_1)$,
\begin{align*}
\alpha(v,b_1,x+\delta ) = \inf\{y\geq x+\delta: g_{b_1}(y) < g_{b_1}(x+\delta)\}
\geq y_1 = \alpha(v,b,x) ,
\end{align*}
so
\begin{align}\label{a25.2}
&(\alpha(v,b_1,x+\delta)-(x+\delta))/2 
\geq (\alpha(v,b,x)-x)/2 -\delta.
\end{align}

Remark \ref{a24.1} (iii) implies that the claims proved in this step for $I=\R$ hold also when $I=[a_1,a_2]$.

\medskip

\emph{Step 6}.
We will argue that
definitions \eqref{j17.1a}-\eqref{j18.6} are consistent under time shifts in the following sense. Suppose that $t> s > 0$. If we take $\wt v(x,0) = v(x,s)$ and $\wt w(x,0)= w(x,s)$ for all $x\in I$, and define $\wt A(v,b,t-s)$ and $\wt A(w,b,t-s)$ as in \eqref{j17.1a}-\eqref{j18.6} relative to  $\wt v(x,0) $ and $\wt w(x,0)$ then $\wt A(v,b,t-s)=A(v,b,t)$ and $\wt A(w,b,t-s)=A(w,b,t)$.

Consider the case when $I=\R$.
We recall the following interpretation from Step 4.
In \eqref{j17.1a}-\eqref{j17.2a}, for every point $x\in A(v,b,0)$ we try to find a matching point $y= \alpha(v,b,x)\in  A(w,b,0)$ such that $ \Leb( A(v,b,0) \cap [x,y] )= \Leb(A(w,b,0)\cap [x,y])$. If such a $y$ exists then $x= \alpha (w,b,y)$, except for a set of $x$ and $y$ of measure zero---some endpoints of intervals comprising
$  A(v,b,0) $ and $ A(w,b,0)$.
We let $x$ and $y$ travel with speed 1 towards each other until they meet and annihilate each other. At a time $t$ before the annihilation time we include the transported  $x$ (i.e., $x+t$) into $  A(v,b,t) $ and we include the transported $y$ (i.e., $y-t$) into $ A(w,b,t)$---this is the meaning of \eqref{j17.5a}-\eqref{j17.6a}. 
If $I=\R$ and there is no $y$ corresponding to $x$ then $x$ keeps moving forever and $x+t$ is included in  $  A(v,b,t) $ for all $t\geq 0$. The above representation of \eqref{j17.1a}-\eqref{j17.6a} makes it clear why these definitions are invariant under time shifts, as stated at the beginning of this step.

Remark \ref{a24.1} (iii) implies that the assertion made at the beginning of  this step holds also when  $I=[a_1,a_2]$.

\medskip

\emph{Step 7}.
Lemma \ref{m26.1} (ii)-(iii) implies that
the functions $v(x,t)$ and $w(x,t)$ are jointly continuous. 

We will prove that the functions $v_*(x,t)$ and $w_*(x,t)$ are left continuous in $t$, i.e., if $x\in I$ and $t>0$  then for any $\delta>0$  there is $\zeta>0$ such that  $|v_*(x,t_1) -v_*(x,t)|\leq \delta$ if $t_1 \in(t-\zeta,t)$, and an analogous statement holds for $w_*$.

Suppose that $I=\R$.

Recall from Step 5 that $b\to A(v,b,t)$ is non-decreasing. 

Suppose that $t>0$ and $v_*(r,t)=b$. By \eqref{j17.7a}, $r\in A(v,b_1,t)$ for $b_1>b$ and $r\notin A(v,b_1,t)$ for $b_1<b$.
 In view of \eqref{j17.5a}, $v(r-t,0) \leq b_1$ and $t \leq(\alpha(v,b_1,r-t)-r+t)/2 $ for $b_1>b$. Thus $v(r-t,0) \leq b$.
 
Consider an arbitrarily small $\delta >0$.
Since $ v(x,0)$ is continuous,
we can find $\eta>0$  so small that $|v(r-t_1,0) - v(r-t,0)|<\delta$ if  $t_1 \in(t-\eta,t)$.
It follows from $v(r-t,0) \leq b$
that $v(r-t_1,0) \leq b +\delta$ for  $t_1 \in(t-\eta,t)$. 

Recall that $t \leq(\alpha(v,b_1,r-t)-r+t)/2 $ for $b_1>b$. Fix some $b_1\in(b, b+\delta)$ and, using \eqref{a25.2},  find $b_2\in(b_1, b+\delta)$ such that for some $\eta_1\in(0,\eta)$ and $t_1\in(t-\eta_1, t)$, 
\begin{align*}
t& \leq(\alpha(v,b_1,r-t)-r+t)/2 
\leq (\alpha(v,b_2,r-t_1)-r+t_1)/2 +(t-t_1),\\
t_1
&\leq (\alpha(v,b_2,r-t_1)-r+t_1)/2 .
\end{align*}
This, the fact that $v(r-t_1,0) \leq b +\delta$ and \eqref{j17.5a} imply that 
$r\in A(v,b_2,t_1)$ for some $b_2 \in(b, b+\delta)$ and $t_1\in(t-\eta_1, t)$. 
Definition \eqref{j17.7a} implies that $v_*(r,t_1) \leq b +\delta$ for $t_1 \in(t-\eta_1,t)$.

If $v(r-t,0) = b$ and $t_1 \in(t-\eta,t)$ then $v(r-t_1,0) \geq b -\delta$ and, in view of 
\eqref{j17.5a}, $r\notin A(v,b_1,t_1)$ for $b_1 \leq b-\delta$. Definition \eqref{j17.7a} implies that $v_*(r,t_1) \geq b -\delta$ for $t_1 \in(t-\eta,t)$.

Suppose that $v(r-t,0) = b_2 < b$.
Definition \eqref{j17.7a} implies that $r\in A(v,b_1,t)$ for $b_1>b$ and $r\notin A(v,b_1,t)$ for $b_1<b$.  This, definition \eqref{j17.5a} and the assumption that $v(r-t,0) = b_2 < b$ imply that $t >(\alpha(v,b_1,r-t)-r+t)/2 $ when $b_1 \in(b_2, b)$. Fix any $b_1 \in(b_2, b)$ and find $\eta_2>0$ such that if $t_1 \in(t- \eta_2,t)$ then  $t_1 >(\alpha(v,b_1,r-t)-r+t)/2 $.
Since $r-t_1 > r-t$, \eqref{a17.1} shows that
\begin{align*}
t_1 >(\alpha(v,b_1,r-t)-r+t)/2 
\geq (\alpha(v,b_1,r-t_1)-r+t_1)/2 .
\end{align*}
In view of 
\eqref{j17.5a}, $r\notin A(v,b_1,t_1)$ for $t_1 \in(t-\eta_2,t)$. Definition \eqref{j17.7a} implies that $v_*(r,t_1) \geq b -\delta$ for  $t_1 \in(t-\eta_2,t)$.

The above estimates show  that for any $\delta>0$ there is $\eta_3>0$ such that  $|v_*(r,t_1) -v_*(r,t)|=|v_*(r,t_1) -b|\leq \delta$ if $t_1 \in(t-\eta_3,t)$.

The proof for $w_*$ is analogous. 

Suppose that $I=[a_1,a_2]$.
If $r\in(a_1,a_2)$ and $t>0$ then we can apply Step 6 and consider a triangle of the form \eqref{m25.1}, such that $(r,t)$ is in the interior of the triangle. Then we can apply the argument presented earlier in this step, noting that $A(v,b,s)$ and $A(w,b,s)$ are defined in the interior of the triangle using the data at its base.

Suppose that $t>0$ and $v_*(a_1,t)=b$. Then
 we should take $x=a_1-t$ in \eqref{j18.5}. Since $x\leq a_1$,  the only condition in \eqref{j18.5} that matters is $t\leq (\alpha(v, b, x) - x)/2$. Recall that
 $a_1\in A(v,b_1,t)$ for $b_1>b$ and $a_1\notin A(v,b_1,t)$ for $b_1<b$ so 
$t\leq (\alpha(v, b, a_1-t) - a_1+t)/2$.
Since $c\to\alpha(v,c,z)$ is  increasing, $t\leq (\alpha(v, b_1, a_1-t) - a_1+t)/2$ for $b_1>b$.

By \eqref{a25.2}, we can find $b_1\in(b, b+\delta)$ and $b_2\in(b_1,b+\delta)$ such that for some $\eta_4\in(0,\eta)$ and all $t_1\in(t-\eta_4, t)$, 
\begin{align*}
t& \leq(\alpha(v,b_1,a_1-t)-a_1+t)/2 
\leq (\alpha(v,b_2,a_1-t_1)-a_1+t_1)/2 +(t-t_1),\\
t_1
&\leq (\alpha(v,b_2,a_1-t_1)-a_1+t_1)/2 .
\end{align*}
This  and \eqref{j17.5a} imply that 
$a_1\in A(v,b_2,t_1)$ for some $b_2 \in(b, b+\delta)$. 
Definition \eqref{j17.7a} implies that $v_*(a_1,t_1) \leq b +\delta$ for $t_1 \in(t-\eta_4,t)$.
 
For $b_1<b$, $t> (\alpha(v, b_1, a_1-t) - a_1+t)/2$.
Let $\eta_5>0$ be such that $t> (\alpha(v, b_1, a_1-t) - a_1+t)/2 + \eta_5$. 
Then, using \eqref{a17.1}, for $\eta_6\in(0,\eta_5)$,
\begin{align*}
t-\eta_6&> (\alpha(v, b_1, a_1-t) - a_1+t)/2+\eta_5-\eta_6
> (\alpha(v, b_1, a_1-t) - a_1+t)/2\\
&\geq (\alpha(v, b_1, a_1-t+\eta_6) - a_1+t-\eta_6)/2 .
\end{align*}
This and \eqref{j18.5} show that  $a_1 \notin A(v, b_1, t-\eta_6)$ for $\eta_6\in(0,\eta_5)$. Definition \eqref{j17.7a} implies that for every $b_1<b$ there is $\eta_5>0$ such that for all $t_1 \in(t-\eta_5,t)$ we have $v_*(a_1,t_1) \geq b_1 $.
We conclude that for any $\delta>0$ there is $\eta_7>0$ such that  $|v_*(a_1,t_1) -v_*(a_1,t)|=|v_*(a_1,t_1) -b|\leq \delta$ if $t_1 \in(t-\eta_7,t)$.

The proof in the case when $r=a_2$, and for $w_*$ is analogous.

\medskip
\emph{Step 8}.
Consider the case when $I=\R$.
Recall the localization argument based on \eqref{a13.1}-\eqref{a13.2}.
It is easy to check that $ A(v^b,b,t) \setminus [-4b,4b]=\calA(v^b,b,t) \setminus [-4b,4b]$ and $ A(w^b,b,t) \setminus [-4b,4b]=\calA(w^b,b,t) \setminus [-4b,4b]$ for $t\in[0,b]$. The argument given at the end of the proof of Theorem \ref{m30.1} shows that it is enough to prove the theorem for an arbitrary fixed $b>0$ and then we can let $b\to \infty$.

We fix an arbitrary $b>0$ and drop the superscript $b$ on $v$ and $w$. 

We have shown that for every $z\in \R$ there is $\eps>0$ such that for $t\in[0,\eps/2)$,
\begin{align*}
 A(v,b,t) \cap [z-\eps/2,z+\eps/2]&=\calA(v,b,t)  \cap [z-\eps/2,z+\eps/2],\\
  A(w,b,t) \cap [z-\eps/2,z+\eps/2]&=\calA(w,b,t)  \cap [z-\eps/2,z+\eps/2].
\end{align*}
By compactness, there exists
 $t_1>0$ such that  $ A(v,b,t) =\calA(v,b,t)  $ and $ A(w,b,t) =\calA(w,b,t) $ for $t\in[0,t_1]$. Let $t_2$ be the supremum of $s$ such that 
  $ A(v,b,t) =\calA(v,b,t)  $ and $ A(w,b,t) =\calA(w,b,t) $ for $t\in[0,s]$. 
  
Since
 $ A(v,b,t) =\calA(v,b,t)  $ and $ A(w,b,t) =\calA(w,b,t) $ for $t\in [0, t_2)$, it follows that  $v_*(x,t)=v(x,t)$ and $w_*(x,t)=w(x,t)$ for $x\in \R$ and $t\in[0,t_2)$.
If $t_2<\infty$ then,  by Step 7,
  $v_*(x,t)=v(x,t)$ and $w_*(x,t)=w(x,t)$ for $x\in \R$ and $t\in[0,t_2]$.
Let $\wt v(x,0) = v(x,t_2)$ and $\wt w(x,0) = w(x,t_2)$.
By Step 6, there exists
 $t_3>0$ such that  $ A(\wt v,b,t) =\calA(\wt v,b,t)  $ and $ A(\wt w,b,t) =\calA(\wt w,b,t) $ for $t\in[0,t_3]$. This contradicts the definition of $t_2$ and the assumption that $t_2<\infty$. We conclude that $ A(v,b,t) =\calA(v,b,t)  $ and $ A(w,b,t) =\calA(w,b,t) $ for all $t\geq 0$.

The case when $I=[a_1,a_2]$ can be dealt with in a similar manner, without a need for localization.
\end{proof}

If $B\subset \R$, $K\subset \R\times[0,\infty)$, $f_1,f_2:B\to \R$ and $g_1,g_2:K\to \R$ then we let
\begin{align*}
\|(f_1,f_2)\|_B &= \sup_{x\in B} \max(|f_1(x)|,|f_2(x)|), \\
\|(g_1,g_2)\|_K &= \sup_{(x,t)\in K} \max (|g_1(x,t)|,|g_2(x,t)|).
\end{align*}
Let $\calC_1$ be the space of pairs of functions $(v(\,\cdot\,,0),w(\,\cdot\,,0))$ mapping $I^2$ to $ \R^2$,  equipped with the norm $\|\,\cdot\,\|_I$ and such that  $v(x,0)$ and $w(x,0)$ are continuous on $I$ and $v(x,0)\geq w(x,0)$ for all $x\in I$. If $I=[a_1,a_2]$ then we also assume that $v(a_1,0)=w(a_1,0)$ and $v(a_2,0)=w(a_2,0)$.

Let $\calX\subset \calC_1$ be the space of pairs of functions in addition satisfying  Assumption \ref{m29.2}.

Let $\calC_2$ be the space 
of pairs of jointly continuous functions $(v(\,\cdot\,,\,\cdot\,),w(\,\cdot\,,\,\cdot\,))$ mapping $(I\times[0,\infty))^2$ to $ \R^2$,  equipped with the norm $\|\,\cdot\,\|_{I\times[0,\infty)}$. 

The norms $\|\,\cdot\,\|_I$ and  $\|\,\cdot\,\|_{I\times[0,\infty)}$ define metrics $d_1$ and $d_2$ in $\calC_1$ and $\calC_2$. A functional $\calT: \calC_1\to \calC_2$ will be called 1-Lipschitz if
for all $(f_1,f_2), (f'_1,f'_2)\in \calC_1$,
\begin{align*}
d_2(\calT(f_1,f_2) , \calT(f'_1,f'_2))
\leq  d_1((f_1,f_2) , (f'_1,f'_2)).
\end{align*}

Note that $\calC_2$ is complete in the metric $d_2$.

\begin{theorem}\label{j22.1}
(i) If $I=[a_1,a_2]$ then $\calX $ is dense in $\calC_1$. The mapping $\calT:\calX\to\calC_2$ taking the initial conditions $(v(\,\cdot\,,0),w(\,\cdot\,,0))\in\calX$ to the solution of  \eqref{n10.4}-\eqref{a5.4} defined  in Theorem \ref{m30.1} can be extended to a 1-Lipschitz functional $\calT:\calC_1\to\calC_2$. 

(ii)
If $I=\R$ then the mapping $\calT:\calX\to\calC_2$ taking the initial conditions  to the solutions to  \eqref{n10.4}-\eqref{a5.4} defined  in Theorem \ref{m30.1} can be extended to a  functional $\calT:\calC_1\to\calC_2$. For every $a>0$, the functional $\calT$  
 is 1-Lipschitz relative to the norms  $\|\,\cdot\,\|_{[-2a,2a]}$ and $\|\,\cdot\,\|_{[-a,a]\times[0,a)}$.
 
(iii) In both cases, elements of $\calT(\calC_1)$ are solutions to \eqref{n10.4}-\eqref{a5.4}.
 
\end{theorem}

\begin{proof}

Comparing \eqref{j21.1}-\eqref{j17.8b} and \eqref{j17.1a}-\eqref{j17.8a}, we see that 
\eqref{j17.1a}-\eqref{j17.8a} define a functional 
$\calT:\calX\to\calC_2$ mapping  $(v(\,\cdot\,,0),w(\,\cdot\,,0))$ to  solutions $(v(\,\cdot\,,\,\cdot\,),w(\,\cdot\,,\,\cdot\,))\in\calC_2$ to \eqref{n10.4}-\eqref{a5.4}.

(i) Suppose that $I=[a_1,a_2]$. Consider $(v(\,\cdot\,,0),w(\,\cdot\,,0)),(v_1(\,\cdot\,,0),w_1(\,\cdot\,,0))\in\calC_1$ and $\eps>0$. Assume that  $|v_1(x,0) - v(x,0)|\leq \eps$ and $|w_1(x,0) - w(x,0)|\leq \eps$ for  all $x\in I$. Consider any $x\in I$, $t\geq 0$ and $b\in \R$.
Suppose that $x\in A(v,b,t)$. Then, using \eqref{j18.5},
\begin{align}\label{j18.2}
v_1(x-t,0) \leq v(x-t,0)+\eps\leq b+\eps
\qquad \text{  or  }\qquad x-t \leq a_1.
\end{align}

We have
\begin{align*}
\bone_{\{z\in I,v_1(z,0) \leq b+\eps\}} \geq
\bone_{\{z\in I,v(z,0) \leq b\}},
\qquad
\bone_{\{z\in I,w_1(z,0) \geq b+\eps\}}\leq
\bone_{\{z\in I,w(z,0) \geq b\}},
\end{align*}
so
\begin{align*}
&\int_{x-t}^y \left(\bone_{\{z\in I,v_1(z,0) \leq b+\eps\}} 
+ \bone_{\{z \leq a_1\}}
- \bone_{\{z\in I,w_1(z,0) \geq b+\eps\}}- \bone_{\{z \geq a_2\}}\right)dz\\
&\geq
\int_{x-t}^y \left(\bone_{\{z\in I,v(z,0) \leq b\}} + \bone_{\{z \leq a_1\}}
- \bone_{\{z\in I,w(z,0) \geq b\}}- \bone_{\{z \geq a_2\}}\right)dz,
\end{align*}
and, using \eqref{j17.3},
\begin{align*}% \label{j17.3} 
&\alpha(v_1,b+\eps,x-t)\\
&= \inf\Bigg\{
y\geq x-t: \int_{x-t}^y \Big(\bone_{\{z\in I,v_1(z,0) \leq b+\eps\}} 
+ \bone_{\{z \leq a_1\}} \\
&\qquad\qquad\qquad
- \bone_{\{z\in I,w_1(z,0) \geq b+\eps\}}- \bone_{\{z \geq a_2\}}\Big)dz <0
\Bigg\}\\
 &\geq \inf\left\{
y\geq x-t: \int_{x-t}^y \left(\bone_{\{z\in I,v(z,0) \leq b\}} 
+ \bone_{\{z \leq a_1\}}
- \bone_{\{z\in I,w(z,0) \geq b\}}- \bone_{\{z \geq a_2\}}\right)dz <0
\right\}\\
&= \alpha(v,b,x-t).
\end{align*}
Hence, if $t\leq (\alpha(v,b,x-t)-(x-t))/2$ then $t\leq (\alpha(v_1,b+\eps,x-t)-(x-t))/2$. This, \eqref{j18.5} and \eqref{j18.2} imply that $x\in A(v_1,b+\eps,t)$. Hence, in view of \eqref{j17.7a}, $v_1(x,t) \leq v(x,t) +\eps$.
By symmetry, $|v_1(x,t) - v(x,t)| \leq\eps$. 
A similar bound for $w$ can be derived in an analogous way.
We have proved that
that if $|v_1(x,0) - v(x,0)|\leq \eps$ and $|w_1(x,0) - w(x,0)|\leq \eps$ for  all $x\in I$
then $|v_1(x,t) - v(x,t)|\leq \eps$ and $|w_1(x,t) - w(x,t)|\leq \eps$ for  all $x\in I$ and $t\geq 0$. In other words, the functional $\calT:\calX\to\calC_2$  is 1-Lipschitz.

Since polynomials are dense in the set of continuous functions,
the set
$\calX $ is dense in $\calC_1$ when $I=[a_1,a_2]$.
Thus we can extend the functional to $\calT:\calC_1\to\calC_2$ while preserving its 1-Lipschitz property.

 We will now prove part (iii) of the theorem in the case discussed so far, i.e. when $I=[a_1,a_2]$.
If $(v(x,t), w(x,t))\in \calT(\calC_1)$ then there is a sequence $(v_k(x,0), w_k(x,0))\in \calX$
converging to $(v(x,0), w(x,0))$ in $\|\,\cdot\|_I$ norm. Then
 $(v_k(x,t), w_k(x,t)) \to (v(x,t), w(x,t))$
in
$\|\,\cdot\,\|_{I\times[0,\infty)}$ norm.

Suppose that $v(x,t) > w(x,t)$ for some $(x,t)\in I\times [0,\infty)$. By continuity, the inequality holds in a neighborhood of $(x,t)$, so $v(y,s) > w(y,s)+\eps$ for some $\eps>0$ and a neighborhood $F$ of $(x,t)$. Then for some $k_1$,  $v_k(y,s) > w_k(y,s)+\eps/2$ for $k>k_1$ and $(y,s)\in F$. Since $v_k$'s and $w_k$'s are solutions to \eqref{n10.4}-\eqref{a5.4}, all characteristics of $v_k$ have slope 1 and  all characteristics of $w_k$ have slope $-1$ in $F$ for $k>k_1$. It follows that all characteristics of $v$ have slope 1 and  all characteristics of $w$ have slope $-1$ in $F$. In other words,
 $(v(x,t), w(x,t))$ is a
solution to \eqref{n10.4}-\eqref{a5.4} in $F$.

Next suppose that  $v(x,t) = w(x,t)$ for 
\begin{align*}
(x,t)\in
U:=\{(y,s):s>s_1, x_1+(s-s_1) < y < x_2-(s-s_1)\} .
\end{align*}
 Without loss of generality, suppose that $s_1=0$ and  $U=\{(y,s):s>0, x_1+s < y < x_2-s\} $.
We will consider two cases. 

First, suppose that $x\to v(x,0) = w(x,0)$ is non-decreasing on $[x_1,x_2]$. We can find strictly increasing $v_k(x,0)=w_k(x,0)$ satisfying Assumption \ref{m29.2}
and converging uniformly to $v(x,0)=w(x,0)$ on $[x_1,x_2]$ as $k\to\infty$. 
The characteristics of $v_k$ and $w_k$ have slope 0 inside $U$. Since $v(x,t)$ and $w(x,t)$ are uniform limits of $v_k(x,t)$ and $w_k(x,t)$ in $U$, the functions $t\to v(x,t)$ and $t\to w(x,t)$ must be constant for $x\in[x_1,x_2]$, as long as $(x,t)\in U$. 
We see that in this case, 
 $v(x,t)$ and $ w(x,t)$ are
solution to \eqref{n10.4}-\eqref{a5.4} in $U$.

If it is not true that $x\to v(x,0) = w(x,0)$ is non-decreasing on $[x_1,x_2]$
then
for some $x_3,x_4\in(x_1,x_2)$ and $\eps>0$, $x_3< x_4$, $v(x_3,0) >v(x_4,0) +\eps$
and $v(x_3,0)> v(x,0)> v(x_4,0)$ for all $x_3< x < x_4$. 
Let $\eps_1>0$ be such that $w(x,0) < v(x_3,0) -\eps_1$ for $x\in[x_3+(x_4-x_3)/4,x_4]$.
For $k\geq 1$, we can find $v_k(x,0)$ and $w_k(x,0)$ satisfying Assumption \ref{m29.2} and such that
$|v_k(x,0)-v(x,0)|<1/k$ and $|w_k(x,0)-w(x,0)|<1/k$ for $x_3\leq x \leq x_4$,
$v_k(x_3,0) > v(x_3,0)+1/(2k)$, 
$w_k(x,0) < w(x_3,0) = v(x_3,0)$ for $x_3\leq x \leq x_4$
and $w_k(x,0) < v(x_3,0) -\eps_1$ for $x\in[x_3+(x_4-x_3)/4,x_4]$.
The characteristic of $v_k$ emanating from $(x_3,0)$ will have slope 1,
i.e. $v_k(x_3+t,t)= v(x_3,0) $ for $t\in[0, (x_4-x_3)/2]$. 
We have $w_k(x_3+t,t)\leq v(x_3,0) -\eps_1$ for $t\in[(x_4-x_3)/4, (x_4-x_3)/2]$ because the value of  $w_k(x_3+t,t)$ with $t$ in the specified range is the value carried by a characteristic of $w_k$ emanating from a point $(x,0)$ with  $x\in[x_3+(x_4-x_3)/4,x_4]$. Since $\eps_1$ does not depend on $k$, passing to the limit yields $w(x_3+t,t)\leq v(x_3,0) -\eps_1 < v(x_3,0) = v_k(x_3+t,t)$ for $t\in[(x_4-x_3)/4, (x_4-x_3)/2]$. This contradicts the assumption that 
$v(x,t) = w(x,t)$ for $(x,t)\in U$ and thus completes the proof that 
$v(x,t)$ and $ w(x,t)$ are
solutions to \eqref{n10.4}-\eqref{a5.4} in $U$.
\medskip

(ii) Suppose that $I=\R$ and $a>0$. Consider $(v(\,\cdot\,,0),w(\,\cdot\,,0)),(v_1(\,\cdot\,,0),w_1(\,\cdot\,,0))\in\calC_1$ and $\eps>0$. Assume that  $|v_1(x,0) - v(x,0)|\leq \eps$ and $|w_1(x,0) - w(x,0)|\leq \eps$ for  all $x\in [-2a,2a]$. Consider any $x\in [-2a,2a]$, $t\in[0,a]$ and $b\in \R$.
Suppose that $x\in A(v,b,t)$. Then, using \eqref{j17.5a},
\begin{align}\label{j18.2a}
v_1(x-t,0) \leq v(x-t,0)+\eps\leq b+\eps.
\end{align}

We have
\begin{align*}
\bone_{\{z\in I,v_1(z,0) \leq b+\eps\}} \geq
\bone_{\{z\in I,v(z,0) \leq b\}},
\qquad
\bone_{\{z\in I,w_1(z,0) \geq b+\eps\}}\leq
\bone_{\{z\in I,w(z,0) \geq b\}},
\end{align*}
so
\begin{align*}
&\int_{x-t}^y \left(\bone_{\{z\in I,v_1(z,0) \leq b+\eps\}} 
- \bone_{\{z\in I,w_1(z,0) \geq b+\eps\}}\right)dz\\
&\geq
\int_{x-t}^y \left(\bone_{\{z\in I,v(z,0) \leq b\}} 
- \bone_{\{z\in I,w(z,0) \geq b\}}\right)dz,
\end{align*}
and, using \eqref{j17.1a},
\begin{align*}% \label{j17.3} 
&\alpha(v_1,b+\eps,x-t)\\
&= \inf\left\{
y\geq x-t: \int_{x-t}^y \left(\bone_{\{z\in I,v_1(z,0) \leq b+\eps\}} 
- \bone_{\{z\in I,w_1(z,0) \geq b+\eps\}}\right)dz <0
\right\}\\
 &\geq \inf\left\{
y\geq x-t: \int_{x-t}^y \left(\bone_{\{z\in I,v(z,0) \leq b\}} 
- \bone_{\{z\in I,w(z,0) \geq b\}}\right)dz <0
\right\}\\
&= \alpha(v,b,x-t).
\end{align*}
Hence, if $t\leq (\alpha(v,b,x-t)-(x-t))/2$ then $t\leq (\alpha(v_1,b+\eps,x-t)-(x-t))/2$. This, \eqref{j17.5a} and \eqref{j18.2a} imply that $x\in A(v_1,b+\eps,t)$. Hence, in view of \eqref{j17.7a}, $v_1(x,t) \leq v(x,t) +\eps$.
By symmetry, $|v_1(x,t) - v(x,t)| \leq\eps$. 
A similar bound for $w$ can be derived in an analogous way.
We have proved that
that if $|v_1(x,0) - v(x,0)|\leq \eps$ and $|w_1(x,0) - w(x,0)|\leq \eps$ for  all $x\in [-2a,2a]$
then $|v_1(x,t) - v(x,t)|\leq \eps$ and $|w_1(x,t) - w(x,t)|\leq \eps$ for  all $x\in [-a,a]$ and $t\in[ 0,a]$. In other words, the functional $\calT:\calX\to\calC_2$  is 1-Lipschitz relative to the norms  $\|\,\cdot\,\|_{[-2a,2a]}$ and $\|\,\cdot\,\|_{[-a,a]\times[0,a)}$.

Since polynomials are dense in the set of continuous functions on any finite interval,
the set
$\calX $ is dense in $\calC_1$ in the norm $\|\,\cdot\,\|_{[-2a,2a]}$.
Thus we can extend the functional $\calT$ to $\calT:\calC_1\to\calC_2$ in such a way that  is 1-Lipschitz relative to the norms  $\|\,\cdot\,\|_{[-2a,2a]}$ and $\|\,\cdot\,\|_{[-a,a]\times[0,a)}$ for every $a>0$.

The proof of (iii) when $I=\R$ is completely analogous to  that in case when $I=[a_1,a_2]$ so it is omitted.
\end{proof}

\section{Uniqueness of solutions}\label{sec:uniq}

\begin{proof}[Proof of Theorem \ref{m28.5}]

It will suffice to show that if $v$ and $w$ are solutions to \eqref{n10.4}-\eqref{a5.4}
in $U:=\{(x,t): t\geq 0, x\in( z-\eps + t , z+\eps -t)\cap I\}$ 
and characteristics are subsonic then the solutions are unique.
It will be enough to prove this in every case considered in the proof of Lemma \ref{j6.5}.   We will limit the present proof to the most complicated cases (iv), (vii) and (viii). 
Recall  Definition \ref{m26.3} and the definitions and properties of a freezing curve $\calF$ and a thawing curve $\calT$ given in Lemma \ref{m22.1}.
\medskip

(a) Recall case (iv) in the proof of  Lemma \ref{j6.5}. It is assumed that 
 $v(z,0) = w(z,0)$, the functions $x\to v(x,0)$ and $x\to w(x,0)$ are strictly increasing on $[z-\eps, z]$ and $x\to w(x,0)$ is strictly decreasing on $[z, z+\eps]$.
 
Let 
\begin{align*}%\label{?}
\wh U&=\{(x,t): t\geq 0, x\in( z-\eps + t , z -t)\}.
\end{align*}
Since $x\to v(x,0)$ and $x\to w(x,0)$ are strictly increasing on $[z-\eps, z]$, the set
 $\calF$ in $\wh U$ is non-empty.  The point $(z,0)$ is in $\calF$ because $v(z,0) = w(z,0)$.
Consider any $(x_1,t_1)$ below $\calF$ in the sense that there exists $(x_1,t_2)\in\calF$ with $0\leq t_1 < t_2$. Let $b = v(x_1 - t_2,0) = w(x_1+t_2,0)$. Then $v(y,0) > b$ for all $y\in[ x_1-t_2,z]$. A subsonic characteristic from $(x_1,t_1)$ must have an endpoint at $(y,0)$ for some $y\in[x_1-t_1,x_1]$. Hence, $v(x_1,t_1) >b$. A similar argument shows that $w(x_1,t_1)< b$. It follows that $v(x_1,t_1) >w(x_1,t_1)$ for all $(x_1,t_1)$ below $\calF$. Therefore, all characteristics of $v$ in this region have slope 1 and all characteristics of $w$ have slope $-1$.

For any $(x_4,t_4), (x_5,t_5)\in \calF$ with $x_4 < x_5$, the above discussion shows that we have $v(x_4,t_4)=w(x_4,t_4)$, $v(x_5,t_5)=w(x_5,t_5)$, $v(x_4,t_4)<v(x_5,t_5)$, and  $w(x_4,t_4)<w(x_5,t_5)$.

Next consider $(x_1,t_1) \in \wh U^+$, i.e. in $\wh U$ but above the freezing line.
Let $t_0$ be such that $(x_1,t_0)\in \calF$.
A subsonic characteristic of $v$ from $(x_1,t_1)$ crosses $\calF$ at a point $(x_6,t_6)$ with $x_6 \leq x_1$ so $v(x_1,t_1) = v(x_6,t_6) \leq v(x_1,t_0)$.
A subsonic characteristic of $w$ from $(x_1,t_1)$ crosses $\calF$ at a point $(x_7,t_7)$ with $x_7 \geq x_1$ so $w(x_1,t_1) = w(x_7,t_7) \geq w(x_1,t_0)=v(x_1,t_0)$. Hence, $v(x_1,t_1) \leq w(x_1,t_1)$. This is possible only if $v(x_1,t_1) = w(x_1,t_1)$. We conclude that all characteristics of $v$ and $w$ in $\wh U^+$ have slope 0.

We will now consider points in $U_r$. Recall that $(z,0)\in \calT$.
First suppose that $(x_1,t_1) \in  U_r^-$. In other words, $(x_1,t_1)$ lies below $\calT$, i.e., there exists $s_1\geq t_1$  such that $(x_1, s_1)\in \calT$. Then $x_1\leq z$ and $v(x_1, z-x_1) = w (x_1+s_1,0)$.
The function $x\to v(x, z-x)$ is increasing for $x\in[z-\eps,z]$ and the function $x\to w(x,0)$ is decreasing for $x\in[z,z+\eps]$. 
A subsonic characteristic of $v$ from $(x_1,t_1)$ passes through a point $(y_1,z-y_1)$ with $y_1 \leq x_1$ so $v(x_1,t_1) = v(y_1,z-y_1) \leq v(x_1,z-x_1)$.
A subsonic characteristic of $w$ from $(x_1,t_1)$ can pass though a point
$(y_2,z-y_2)$ with $x_1 \leq y_2 \leq z$
or through  point $(y_3,0)$ with $z\leq y_3 \leq x_1+t_1$.
In the first case $w(x_1,t_1) = w(y_2,z-y_2) \geq w(x_1,z-x_1)=v(x_1,z-x_1)$.
In the second case  $w(x_1,t_1) = w(y_3,0) \geq w(x_1+t_1,0)\geq w(x_1+s_1,0)=v(x_1,z-x_1)$. Hence, $v(x_1,t_1) \leq w(x_1,t_1)$. This is possible only if $v(x_1,t_1) = w(x_1,t_1)$. We conclude that all characteristics of $v$ and $w$ must have slope 0 in $U_r^-$.

We will argue that the whole interior of the region $  U_r^+$ is a liquid zone.

All characteristics of $v$ in the triangle $\wt U:=\{(x,t):t>0, z+t < x < z+\eps -t\}$ have slope 1 and all characteristics of $w$ have slope $-1$ by \eqref{a5.4}.

Suppose that $(x_1,t_1) \in  U_r^+ \setminus \wt U$. 
In view of \eqref{m22.6}, $z< x_1+t_1 < z+\eps$.
Let $z_1 = \inf\{x\in[z-\eps,x_1+t_1]: v(x, x_1+t_1 -x)>w(x, x_1+t_1 -x)\}$. Note that the set in the definition of $z_1$ is non-empty because the condition is satisfied for all  $(x, x_1+t_1 -x) \in \wt U$; for example, it is satisfied for $x$ very close to $x_1+t_1$. Hence, $z_1 < x_1+t_1$.

If $(z_1, x_1+t_1 -z_1)\in \calT$ for every $(x_1,t_1) \in  U_r^+ \setminus \wt U$
then  the whole interior of the region $  U_r^+$ is a liquid zone.

Suppose that for some $(x_1,t_1) \in  U_r^+ \setminus \wt U$,
$(z_1, x_1+t_1 -z_1)\notin \calT$.
Then, by continuity, 
\begin{align}\label{a4.1}
v(z_1, x_1+t_1 -z_1)=w(z_1, x_1+t_1 -z_1)=w(x_1+t_1,0).
\end{align}
The second equality holds because $v>w$ in a neighborhood of the open line segment connecting $(z_1, x_1+t_1 -z_1) $ and $(x_1+t_1,0)$ according to the definition of $z_1$, and, therefore, the characteristic of $w$ must have slope $-1$ in this neighborhood.

The halfline starting at $w(x_1+t_1),0)$ and passing through $(z_1, x_1+t_1 -z_1)$ may cross $\calT$ or not, resulting in (a) or (b) below.

(a) $(z_2, x_1+t_1 -z_2)\in \calT$ for some $z_2 < z_1$.

(b) $v(z_3, z-z_3)> w(x_1+t_1),0)=w(z_3+(x_1+t_1-z_3),0)$ for all $z_3$ such that $(z_3, s_2)\in \calT$ for some $s_2 \geq 0$. 

In case (a), a subsonic characteristic of $v$ from $(z_1, x_1+t_1 -z_1)$ may
intersect $\calT$ at $(y_4,s_3)$ with $y_4 >z_2$. Then, in view of \eqref{a4.1}, 
\begin{align*}
v(z_1, x_1+t_1 -z_1)&=v(y_4,s_3) > v(z_2, x_1+t_1 -z_2)\\
& = w(z_2+(x_1+t_1 -z_2),0)
= w(x_1+t_1,0)=v(z_1, x_1+t_1 -z_1),
\end{align*}
a contradiction.

Another possibility is that a subsonic characteristic of $v$ from $(z_1, x_1+t_1 -z_1)$
has an endpoint $(z_3,0)$ with $z_3\in[z,z_1]$. Recall that $z_1 < x_1+t_1$ and use  \eqref{a4.1} to see that $v(z_1, x_1+t_1 -z_1) = v(z_3,0) \geq w(z_3,0) > w(x_1+t_1,0) =v(z_1, x_1+t_1 -z_1)$, once again a contradiction.

In  case (b),
a subsonic characteristic of $v$ from $(z_1, x_1+t_1 -z_1)$ may
intersect $\calT$ at $(y_5,s_4)$. We use the condition in (b) to
choose $z_4 < y_5$ such that $v(z_4, z-z_4)> w(z_4+(x_1+t_1-z_4),0)$ and $(z_4, s_5)\in \calT$ for some $s_5 \geq 0$. Then
\begin{align*}
v(z_1, x_1+t_1 -z_1)&=v(y_5,s_4) > v(z_4,s_5)=v(z_4, z-z_4)\\
& > w(z_4+(x_1+t_1 -z_4),0)
= w(x_1+t_1,0)=v(z_1, x_1+t_1 -z_1),
\end{align*}
a contradiction.

It may also happen that  a subsonic characteristic of $v$ from $(z_1, x_1+t_1 -z_1)$
has an endpoint $(z_5,0)$ with $z_5\in[z,z_1]$. This leads to a contradiction, just like in case (a).

We have completed the proof that the whole interior of the region $  U_r^+$ is a liquid zone.
 Therefore, all characteristics of $v$ in $ U_r^+$ have slope 1 and all characteristics of $w$ have slope $-1$.
 
\medskip

(b) Recall case (vii) in the proof of  Lemma \ref{j6.5}.
Let $U= \{(x,t): t\geq 0, a_1  < x < a_1+\eps - t\}$.
We assume that $I=[a_1,a_2]$, $z= a_1 $,  $v( a_1 , t) = w( a_1 ,t)$ for $t\in[0,\eps]$, and the function  $x\to w(x,0)$ is strictly decreasing on $[ a_1 , a_1 + \eps]$.
By \eqref{a5.4} (ii), the triangle $U_1 := \{(x,t): t\geq 0, a_1 + t < x < a_1+\eps - t\}$ is a liquid zone.

We will argue that the interior of $U$ is a liquid zone. Suppose otherwise. Let 
 $(x_1,t_1) \in U \setminus U_1$ be such that $v(x_1,t_1)=w(x_1,t_1)$ and $x_1>a_1$.
Let $x_2$ be the infimum of $y\in[x_1,a_1+\eps]$ such that $v(y,x_1+t_1-y) > w(y,x_1+t_1-y)$.
Then $v(x_2,x_1+t_1-x_2) = w(x_2,x_1+t_1-x_2)$.
 
A subsonic characteristic of $v$ from $(x_2,x_1+t_1-x_2)$ may cross the boundary of $U_1$ at a point $(a_1+t_2, t_2)$ with $a_1+t_2\leq x_2$. Then $v(x_2,x_1+t_1-x_2) = v(a_1+t_2, t_2)
= v(a_1,0) = w(a_1,0) $. 
A subsonic characteristic of $w$ from $(x_2,x_1+t_1-x_2)$ must cross the boundary of $U_1$ at a point $(a_1+t_3, t_3)$ with $a_1+t_3\geq x_2\geq x_1 >a_1$. Then $w(x_2,x_1+t_1-x_2) = w(a_1+t_3, t_3)
= w(a_1+2t_3,0) < w(a_1,0) $. Hence, $v(x_2,x_1+t_1-x_2) >w(x_2,x_1+t_1-x_2)$, a contradiction.

If a subsonic characteristic of $v$ from $(x_2,x_1+t_1-x_2)$ does not cross the boundary of $U_1$ then it must end at $(a_1,t_4)$ for some $t_4\in[0 , x_1+t_1-x_2)$. Then $w(x_1+t_1,0) = w(x_2,x_1+t_1-x_2) =v(x_2,x_1+t_1-x_2) =v(a_1,t_4)=w(a_1,t_4)$.
A subsonic characteristic of $w$ from $(a_1,t_4)$ must have an end at $(x_3,0)$ with $x_3\in[a_1, a_1+t_4]$. Thus
$w(a_1,t_4) = w(x_3,0)\geq w(a_1+t_4,0) >w(a_1 +x_1+t_1-x_2,0)$
and, therefore, $w(x_1+t_1,0) >w(a_1 +x_1+t_1-x_2,0)$. This is impossible 
because $x\to w(x,0) $ is decreasing and $a_1-x_2 < 0$.

We proved that the whole interior of  $  U$ is a liquid zone.
All characteristics of $v$ in $ U$ have slope 1 and all characteristics of $w$ have slope $-1$.

\medskip

(c) We will analyze case (viii) in the proof of  Lemma \ref{j6.5}.
Let $U= \{(x,t): t\geq 0, a_1  < x < a_1+\eps - t\}$.
We assume that $I=[a_1,a_2]$, $z= a_1 $,  $v( a_1 , t) = w( a_1 ,t)$ for $t\in[0,\eps]$,
and the functions $x\to v(x,0)$ and $x\to w(x,0)$ are strictly increasing on $[ a_1 , a_1 + \eps]$. 

If the definition of $\wh U$ in \eqref{m25.1} is replaced with 
\begin{align*}%\label{m25.1}
\wh U&=\{(x,t): t\geq 0, x\in( a_1 , a_1+\eps -t)\},
\end{align*}
then the rest of the statement of Lemma \ref{m22.1} (i) and its proof remain valid. Hence, we have a well defined freezing curve $\calF$ and it follows from our current assumptions that $\calF$ is non-empty, it cuts $U$ into two connected sets $\wh U^-$ and $\wh U^+$, and the left endpoint of $\calF$ is $(a_1,0)$.

Let $\wh U^*=\{(x,t)\in \wh U^+: t\geq 0,a_1 + t < x< a_1+\eps-t\}$. 
The set $\wh U^*$ is a frozen zone and $\wh U^-$ is a liquid zone by the same argument that was applied  in part (a) of this proof to show that  $\wh U^+$ was frozen.

It remains to analyze $\wh U^+\setminus \wh U^*$. We will show that it is frozen. Suppose otherwise. 
Then there is $(x_1,t_1) \in \wh U^+ \setminus \wh U^*$ such that $v(x_1,t_1)> w(x_1,t_1)$ and $x_1>a_1$.
Let $t_2$ be the infimum of $t$ such that for some $x\geq x_1$, $(x,t)\in \wh U^+ \setminus \wh U^*$ and $v(x,t)> w(x,t)$. 
Let $(x_2,t_3)\in \wh U^+ \setminus \wh U^*$ be such that $v(x_2,t_3)> w(x_2,t_3)$, $x_2\geq x_1$
and $t_3 -t_2 < x_1-a_1$.
Let $t_4$ be such that $(x_2,t_4) \in \calF$ and note that the line segment $J$ connecting $(x_2, t_4)$ and $(x_2, t_2)$ is in the frozen zone so $v(x,t)=w(x,t) = v(x_2,t_4)=w(x_2,t_4)$ for all $(x,t)\in J$.
 
A subsonic characteristic of $w$ from $(x_2,t_3)$ crosses $\calF$ at a point $(x_3,t_5)$ with $x_3\geq x_2$. Hence 
\begin{align}\label{a6.1}
v(x_2,t_3) > w(x_2,t_3)=w(x_3,t_5)\geq w(x_2, t_4).
\end{align}

If a subsonic characteristic of $v$ from $(x_2,t_3)$ crosses $\calF$ then it crosses it at a point $(x_4, t_6)$ with $x_4 \leq x_2$. Then $v(x_2,t_3) = v(x_4,t_6) \leq v(x_2, t_4) = w(x_2,t_4)$,  contradicting \eqref{a6.1}.

If a subsonic characteristic of $v$ from $(x_2,t_3)$ does not cross $\calF$ then it has an endpoint $(a_1, t_7)$ with $t_7 \leq t_3 -(x_2-a_1)\leq t_2$. Then $v(x_2, t_3)=v(a_1,t_7)=w(a_1,t_7)$.

A characteristic of $w$ from $(a_1,t_7)$ either crosses $\calF$ at a point $(x_5,t_8)$ with $x_5 \leq x_2$ or it crosses the line segment $J$ at a point $(x_6,t_9)$. In the first case,
$w(a_1,t_7) = w (x_5, t_8) \leq w(x_2,t_4)$ and, therefore, $v(x_2, t_3) \leq w(x_2,t_4)$, a contradiction with \eqref{a6.1}. In the second case, $v(x_2,t_3) =  w(a_1,t_7) = w(x_6,t_9) = w(x_2,t_4)$, once again a contradiction with \eqref{a6.1}.
This proves that $\wh U^+$ is frozen.
\end{proof}

\section{Properties of solutions}\label{sec:prop}

\subsection{Eventual freezing}

\begin{proposition}\label{a22.10}
If $I = [a_1,a_2]$ then for all $s,t\geq 2(a_2-a_1)$
we have $v(x,t)=v(x,s)=w(x,t)=w(x,s)$ and  $x\to v(x,t)$ is non-decreasing. 
\end{proposition}

\begin{proof}
Let $a_*=a_2-a_1$.
It follows from \eqref{j17.3} that for every $b$
and $x,y \leq a_1 -a_*$, $\alpha( v,b,x) -x=\alpha( v,b,y) -y$.
If $z\in I$ and $t \geq 2 a_*$ then $z-t \leq a_1 -a_*$.
Then \eqref{j18.5} implies that for any $b\in\R$, $z\in I$ and $t_1,t_2 \geq 2 a_*$, $A(v,b,t_1) = A(v,b, t_2)$. This and \eqref{j17.7a} show that 
$v(x,t_1) = v(x, t_2)$ for any $b$ and $t_1,t_2 \geq a_2-a_1$.
The proof for $w$ is analogous.

Suppose that $t\geq 2a_*$.
If $v(x,t)> w(x,t)$ for some $x$ then there is $b$ such that $A(v,b,t)\cup A(w,b,t)\neq I$.
If $v(x,t)= w(x,t)$ for all $x$ and
 $x\to v(x,t)$ is not non-decreasing then for some $b$, the set $A(v,b,t)$  contains an interval  to the left of an interval  in $A(w,b,t)$, and both intervals  have positive length. 
 In either case, parts of $A(v,b,t)$ will move to the right and parts of $A(w,b,t)$ will move to the left after time $t$ (see Step 1 of the proof of Lemma \ref{a22.6}) and that will result in some annihilation and, therefore, in some changes to $A(v,b,s)$ and $A(w,b,s)$ for $s>t$, contradicting the first part of the proof.
\end{proof}

\subsection{Monotonicity of variation}

Recall the definition of total variation of a function $f: \R\to \R$,
\begin{align*}
\bV(f) = \sup\sum_{k=1}^n |f(x_k)-f(x_{k-1})|,
\end{align*}
where the supremum is taken over all $n$ and all sequences
$x_k$ such that $ x_0\leq x_1 \leq \dots \leq x_n $.
We say that \(f\) has finite total variation if \(V(f)<\infty\).\\

\begin{proposition}
  Suppose that $I=\R$ and $v$ and $w$ are solutions to
  \eqref{n10.4}-\eqref{a5.4} in the sense of Theorem
  \ref{j22.1} with arbitrary continuous initial
  conditions. If the intitial conditions have finite total
  variation, then $t\to \bV(v(\,\cdot\,,t))$ and
  $t\to \bV(w(\,\cdot\,,t))$ are non-increasing.
\end{proposition}

\begin{proof}
Since  convergence in Theorem \ref{j22.1} is uniform, it will suffice to prove the claim for initial conditions satisfying Assumption \ref{m29.2}.
It will suffice to prove monotonicity of variation in any fixed triangle $U$ of the type considered in Lemma  \ref{j6.5}.
Recall that in every $U$, the function $v(\,\cdot\,,0)$ is increasing, or decreasing, or has only one local extremum.
In every case discussed in the proof of Lemma  \ref{j6.5} (b), for every fixed $t\geq 0$, the function $x\to v(x,t)$ for $x$ such that $(x,t)\in U$ can take only (some) values taken by $v(x,0)$, $(x,0)\in U$, and the values are taken in the same order. This easily implies that $t\to \bV(v(\,\cdot\,,t))$ is non-increasing. The argument for $w$ is analogous.
\end{proof}

\subsection{Conservation of energy and momentum}

We define frozen and liquid regions and the boundary between them as follows,
\begin{align*}
F(s,u)&=\{(x,t): s \leq t \leq u, x\in [a_1,a_2],
v(x,t)=w(x,t)\},\\
L(s,u)&=\{(x,t): s \leq t \leq u, x\in [a_1,a_2],
v(x,t)> w(x,t)\},\\
D(s,u)&= \prt F(s,u) \cap \prt L(s,u).
\end{align*}

\begin{lemma}\label{s5.2}

Suppose that the functions $ v(x,0)$ and $ w(x,0)$ satisfy Assumption \ref{m29.2} and, moreover, are continuous and piecewise linear with a finite number of intervals of linearity.Then the boundary  $D(0,t)$ consists of a finite number of line segments for every $t<\infty$. 

\end{lemma}

\begin{proof}

Our assumptions imply that there is finite sequence $-\infty=b_1 < b_2 < \dots < b_n=\infty$ such that for every $k$, the inverse image of $[b_k,b_{k+1}]$ by $v(x,0)$ 
consists of a finite number of intervals, and $v(x,0)$ is linear on each of these intervals. Moreover,  the inverse image of $[b_k,b_{k+1}]$ by $w(x,0)$ 
consists of a finite number of intervals, and $w(x,0)$ is linear on each of these intervals.
This implies that $\Gamma_k:=\{(x,b): x\in \R, b_k \leq b \leq b_{k+1}\}$ is partitioned into a finite number of cells by the graphs of $v(x,0)$ and $w(x,0)$ for every $k$. Formula \ref{j17.3} then shows that
$\Gamma_k$ can be partitioned into a finite number of cells such that
$(b,x) \to \alpha(v,b,x)$ is bilinear on each cell and the boundaries of cells consist of a finite number of line segments.
The conditions in definition \eqref{j18.5} are 
\begin{align*} 
z=x+t, t\leq (\alpha(v,b,x)-x)/2 \text{  and  }( v(x,0)\leq b  \text{  or  }
x \leq a_1). 
\end{align*}
For any fixed $s>0$, the set $\Gamma_k \times [0,s]$ is partitioned into a finite number of cells such that for every cell, the conditions are satisfied for all points in the interior of the cell or for no points in the interior of the cell. Moreover, the boundaries of the cells are subsets of a finite number of 2-dimensional planes. This, Lemma \ref{a22.6}, \eqref{j18.5} and \eqref{j17.7a} imply that $I\times[0,t]$ can be partitioned into a finite number of cells such that
$(x,t) \to v(x,t)$ is bilinear on each cell and the boundaries of cells consist of a finite number of line segments. 
\end{proof}

\begin{proposition}\label{a5.1}
If $I=[a_1,a_2]$ and $\mu$ and $\sigma$ are solutions to 
\eqref{n10.1}-\eqref{n10.3} with continuous initial conditions  then
for all $t\geq 0$,
\begin{align}\label{a5.2}
\int_{a_1}^{a_2} \mu(x,t) dx &= \int_{a_1}^{a_2} \mu(x,0) dx,\\
\int_{a_1}^{a_2} \left(\mu(x,t)^2 + \sigma(x,t)^2\right) dx &=
\int_{a_1}^{a_2} \left(\mu(x,0)^2 + \sigma(x,0)^2\right) dx.
\label{a5.3}
\end{align}
\end{proposition}

\begin{proof}
First we will prove the proposition in the case when the initial conditions satisfy the assumptions of Lemma \ref{s5.2}. It follows from the lemma that there exists a finite or infinite sequence $0=s_1< s_2< \dots$ without an accumulation point such that within each region $[a_1,a_2] \times [s_j, s_{j+1}]$ the solutions  $\mu(x,t)$ and $\sigma(x,t)$  to 
\eqref{n10.1}-\eqref{n10.3} are bilinear functions  in the interiors of liquid and frozen regions. Moreover, the boundaries of liquid and frozen regions consist of a finite number of  line segments. 
Fix some $[a_1,a_2] \times [s_j, s_{j+1}]$ and
let $\{\gamma_k(t), s_j\leq t \leq s_{j+1}\}$, $k=0,\dots,n$ (where $n$ depends on $j$) be such that
$\gamma_0(t) =a_1$, $\gamma_n(t) =a_2$, $\gamma_k(t) \leq \gamma_{k+1}(t)$ for all $k$,  $\mu(x,t)$ and $\sigma(x,t)$  are frozen on the intervals $[\gamma_{k}(t), \gamma_{k+1}(t)]$ for $k\in 2 \Z$
and liquid in the interiors of the complementary intervals. 

In the following calculations we use the following observations:  $\sigma(\gamma_k(t),t)=0$ for all $k$ and $t$; $\gamma_0'(t) =\gamma'_n(t) =0$ for all $t$.

We have
\begin{align*}%\label{a5.2}
\frac \prt {\prt t}
\int_{a_1}^{a_2} \mu(x,t) dx &
= \sum_{k=0}^{n-1} \frac \prt {\prt t}
\int_{\gamma_{k}(t)}^{\gamma_{k+1}(t)} \mu(x,t) dx \\
&=\sum_{k=0}^{n-1} \left(
\int_{\gamma_{k}(t)}^{\gamma_{k+1}(t)} \mu_t(x,t) dx 
+ \mu(\gamma_{k+1}(t),t)\gamma'_{k+1}(t)
- \mu(\gamma_{k}(t),t)\gamma'_{k}(t)
\right)\\
&=\sum_{k=0}^{n-1} 
\int_{\gamma_{k}(t)}^{\gamma_{k+1}(t)} \mu_t(x,t) dx \\
&=\sum_{k\in2 \Z} 
\int_{\gamma_{k}(t)}^{\gamma_{k+1}(t)} \mu_t(x,t) dx 
+ \sum_{k\notin2 \Z} 
\int_{\gamma_{k}(t)}^{\gamma_{k+1}(t)} \mu_t(x,t) dx \\
&=\sum_{k\in2 \Z} 
\int_{\gamma_{k}(t)}^{\gamma_{k+1}(t)} 0\, dx 
- \sum_{k\notin2 \Z} 
\int_{\gamma_{k}(t)}^{\gamma_{k+1}(t)} \sigma_x(x,t) dx \\
&= \sum_{k\notin2 \Z} (\sigma(\gamma_{k}(t),t) - 
\sigma(\gamma_{k+1}(t),t))=0.
\end{align*}
This shows that
\begin{align*}
\int_{a_1}^{a_2} \mu(x,t) dx &= \int_{a_1}^{a_2} \mu(x,s_j) dx,
\end{align*}
for every $j$ and $t\in[s_j,s_{j+1}]$, and, therefore, completes the proof of \eqref{a5.2} under the assumptions of Lemma \ref{s5.2}.

We have
\begin{align*}%\label{a5.2}
\frac \prt {\prt t}
\int_{a_1}^{a_2}& \left(\mu(x,t)^2 + \sigma(x,t)^2\right) dx 
= \sum_{k=0}^{n-1} \frac \prt {\prt t}
\int_{\gamma_{k}(t)}^{\gamma_{k+1}(t)} \left(\mu(x,t)^2 + \sigma(x,t)^2\right) dx \\
&=\sum_{k=0}^{n-1} 
\int_{\gamma_{k}(t)}^{\gamma_{k+1}(t)} \left(2\mu(x,t)\mu_t(x,t) + 2\sigma(x,t)\sigma_t(x,t)\right) dx \\
&\qquad +\sum_{k=0}^{n-1}
\left(\mu(\gamma_{k+1}(t),t)^2 + \sigma(\gamma_{k+1}(t),t)^2\right)
\gamma'_{k+1}(t)\\
&\qquad -\sum_{k=0}^{n-1} \left(\mu(\gamma_{k}(t),t)^2 + \sigma(\gamma_{k}(t),t)^2\right)
\gamma'_{k}(t)\\
&=\sum_{k=0}^{n-1} 
\int_{\gamma_{k}(t)}^{\gamma_{k+1}(t)} \left(2\mu(x,t)\mu_t(x,t) + 2\sigma(x,t)\sigma_t(x,t)\right) dx \\
&=\sum_{k\in2 \Z} 
\int_{\gamma_{k}(t)}^{\gamma_{k+1}(t)} \left(2\mu(x,t)\mu_t(x,t) + 2\sigma(x,t)\sigma_t(x,t)\right) dx \\
&\qquad+ \sum_{k\notin2 \Z} 
\int_{\gamma_{k}(t)}^{\gamma_{k+1}(t)} \left(2\mu(x,t)\mu_t(x,t) + 2\sigma(x,t)\sigma_t(x,t)\right) dx \\
&=\sum_{k\in2 \Z} 
\int_{\gamma_{k}(t)}^{\gamma_{k+1}(t)} 0\, dx 
- \sum_{k\notin2 \Z} 
\int_{\gamma_{k}(t)}^{\gamma_{k+1}(t)} \left(2\mu(x,t)\sigma_x(x,t) + 2\sigma(x,t)\mu_x(x,t)\right) dx \\
&= 2\sum_{k\notin2 \Z}\left(
\mu(\gamma_{k}(t),t) \sigma(\gamma_{k}(t),t) - 
\mu (\gamma_{k+1}(t),t) \sigma(\gamma_{k+1}(t),t)\right)=0.
\end{align*}
This shows that 
\begin{align*}
\int_{a_1}^{a_2} \left(\mu(x,t)^2 + \sigma(x,t)^2\right) dx &=
\int_{a_1}^{a_2} \left(\mu(x,s_j)^2 + \sigma(x,s_j)^2\right) dx,
\end{align*}
for every $j$ and $t\in[s_j,s_{j+1}]$, and, therefore, completes the proof of \eqref{a5.3} under the assumptions of Lemma \ref{s5.2}

The general case follows from an approximation argument using
Theorem \ref{j22.1}.
\end{proof}

\begin{remark}
The integrals in \eqref{a5.2} and \eqref{a5.3} represent ``momentum''  and ``energy'' in the pinned balls model discussed in Section \ref{sec:pinn}. So \eqref{a5.2}-\eqref{a5.3} can be interpreted as ``conservation'' of momentum  and energy; this agrees well with the physical interpretation of the evolution of pinned balls' velocities.
\end{remark}

\subsection{Conservation of occupation measure}

The following definitions are inspired by the concept of ``local time'' well known in stochastic analysis, see, e.g.,
\cite[Sec. 3.6]{KS}.

For a function $f: I\to \R$, we define  occupation measure $\calD$ by
\begin{align}\label{a22.5}
\calD(f,B) = \int_I \bone_{\{f(x)\in B\}} dx,
\qquad B\subset \R.
\end{align}

If the following derivative exists then it will be called the level density,
\begin{align*}
\calL(f,x) = \frac d {dx} \calD(f,(-\infty, x]),
\qquad x \in \R.
\end{align*}

\begin{proposition}\label{a22.1}

(i)
For all $B\subset \R$ and $t\geq 0$,
\begin{align}\label{a22.4}
\calD(v(\,\cdot\,,t), B)
-\calD(w(\,\cdot\,,t), B)
= \calD(v(\,\cdot\,,0), B) - \calD(w(\,\cdot\,,0), B).
\end{align}

(ii) If the following level densities exist then for all $x\in \R$  and $t\geq 0$,
\begin{align*}
\calL(v(\,\cdot\,,t), x)
-\calL(w(\,\cdot\,,t), x)
= \calL(v(\,\cdot\,,0), x) - \calL(w(\,\cdot\,,0), x).
\end{align*}
\end{proposition}

\begin{proof}
We will write $\Leb$ for Lebesgue measure.

Fist consider the case when $I=\R$.
Suppose that $\calD(v(\,\cdot\,,t),(b_1,b_2))< \infty$ and $\calD(w(\,\cdot\,,t),(b_1,b_2))< \infty$. Then almost all points  $y\in(b_1,b_2)$ are continuity points, i.e., $\calD(v(\,\cdot\,,t),\{y\})=\calD(w(\,\cdot\,,t),\{y\})=0$.
To show that \eqref{a22.4} holds for all   $B\subset(b_1,b_2)$,
it will suffice to prove \eqref{a22.4}
for $B=[c_1,c_2]$ where  $c_1,c_2\in(b_1,b_2)$ are continuity points.
It has been proved in Lemma \ref{a22.6} that $\calA(v,b,t)= A(v,b,t)$ and $\calA(w,b,t)= A(w,b,t)$.
By \eqref{j21.1}-\eqref{j21.2} and \eqref{a22.5},
\begin{align*}%\label{a22.5}
\calD(v(\,\cdot\,,t), [c_1,c_2])
= \int_I \bone_{\{v(x,t)\in [c_1,c_2]\}} dx
= \Leb(A(v, c_2,t)) - \Leb(A(v, c_1,t)).
\end{align*}
Similarly,
\begin{align*}%\label{a22.5}
\calD(w(\,\cdot\,,t), [c_1,c_2])
= \int_I \bone_{\{w(x,t)\in [c_1,c_2]\}} dx
= \Leb(A(w, c_1,t)) - \Leb(A(w, c_2,t)),
\end{align*}
so
\begin{align*}
&\calD(v(\,\cdot\,,t), [c_1,c_2])
-\calD(w(\,\cdot\,,t), [c_1,c_2])\\
&= \Leb(A(v, c_2,t)) - \Leb(A(v, c_1,t))
-\Leb(A(w, c_1,t)) + \Leb(A(w, c_2,t))\\
&= [\Leb(A(v, c_2,t)) + \Leb(A(w, c_2,t))]
-[\Leb(A(w, c_1,t)) - \Leb(A(v, c_1,t))].
\end{align*}
Recall from the proof of Lemma \ref{a22.6} that the sets 
$ A(v,b,t)$ and $ A(w,b,t)$ annihilate each other at the same rate
during the evolution in time---this
was stated informally in Step 1 and proved rigorously later in the proof.
Hence, $t\to \Leb(A(v, c_2,t)) + \Leb(A(w, c_2,t))$ is constant
and so is $t\to \Leb(A(w, c_1,t)) - \Leb(A(v, c_1,t))$.
It follows that $t\to \calD(v(\,\cdot\,,t), [c_1,c_2])
-\calD(w(\,\cdot\,,t), [c_1,c_2])$ is constant. 
This completes the proof of \eqref{a22.4}  when 
 $I=\R$, $\calD(v(\,\cdot\,,t),(b_1,b_2))< \infty$ and $\calD(w(\,\cdot\,,t),(b_1,b_2))< \infty$.
A similar argument applies when one of the sides of  \eqref{a22.4} is infinite.

Now assume that $I=[a_1,a_2]$. 
Fix $T>0$ and let
\begin{align*}
 A_T(v,b,t) &=  A(v,b,t) \cup [a_1 -T+ t, a_1 ] ,\\
 A_T(w,b,t) &=  A(w,b,t) \cup [a_2, a_2 +T- t ] .
\end{align*}
For $t\in[0,T]$,
\begin{align*}%\label{a22.5}
\calD(v(\,\cdot\,,t), [c_1,c_2])
&= \int_I \bone_{\{v(x,t)\in [c_1,c_2]\}} dx
= \Leb(A(v, c_2,t)) - \Leb(A(v, c_1,t))\\
&= \Leb(A_T(v, c_2,t)) - \Leb(A_T(v, c_1,t))
\end{align*}
and 
\begin{align*}%\label{a22.5}
\calD(w(\,\cdot\,,t), [c_1,c_2])
&= \Leb(A_T(w, c_2,t)) - \Leb(A_T(w, c_1,t)),
\end{align*}
so
\begin{align*}
&\calD(v(\,\cdot\,,t), [c_1,c_2])
-\calD(w(\,\cdot\,,t), [c_1,c_2])\\
&= [\Leb(A_T(v, c_2,t)) + \Leb(A_T(w, c_2,t))]
-[\Leb(A_T(w, c_1,t)) - \Leb(A_T(v, c_1,t))].
\end{align*}
It follows from cases (VI)-(VIII) in the proof of Lemma \ref{a22.6} that the sets 
$ A_T(v,b,t)$ and $ A_T(w,b,t)$ annihilate each other at the same rate
when $t\in[0,T]$.
Hence, $t\to \Leb(A_T(v, c_2,t)) + \Leb(A_T(w, c_2,t))$ is constant
and so is $t\to \Leb(A_T(w, c_1,t)) - \Leb(A_T(v, c_1,t))$.
It follows that $t\to \calD(v(\,\cdot\,,t), [c_1,c_2])
-\calD(w(\,\cdot\,,t), [c_1,c_2])$ is constant. 
This completes the proof of \eqref{a22.4}  when 
 $I=[a_1,a_2]$, $\calD(v(\,\cdot\,,t),(b_1,b_2))< \infty$ and $\calD(w(\,\cdot\,,t),(b_1,b_2))< \infty$.
A similar argument applies when one of the sides of  \eqref{a22.4} is infinite.

Part (ii) of the proposition follows from part (i).
\end{proof}

\begin{corollary}
Suppose that $a>0$, $I=[-a,a]$, $\sigma(x,0)=0$ for all $x$, and $\mu(x,0)$ is strictly decreasing. Recall from Proposition  \ref{a22.10} that $\sigma(x,t) = 0$ for all $x$ and $t\geq 4a$. We have $\mu(x,t) = \mu(-x,0)$ for
all $t\geq 4a$ and $x\in I$. 

\end{corollary}

\begin{proof}
Suppose that $t\geq 4a$ and consider any $b\in\R$.
We apply \eqref{a22.4} to $B=(-\infty,b]$ to see that
\begin{align}\label{a22.11}
\calD(v(\,\cdot\,,t), (-\infty,b])
-\calD(w(\,\cdot\,,t), (-\infty,b])
= \calD(v(\,\cdot\,,0), (-\infty,b]) - \calD(w(\,\cdot\,,0), (-\infty,b]).
\end{align}
Since $\mu(x,0)=2v(x,0)=2w(x,0)$ is strictly decreasing and, by Lemma \ref{a22.10},  $\mu(x,t)=2v(x,t)=2w(x,t)$ is strictly increasing, 
\begin{align*}%\label{a22.11}
2a&=\calD(v(\,\cdot\,,t), (-\infty,b])
+\calD(w(\,\cdot\,,t), (-\infty,b])\\
&= \calD(v(\,\cdot\,,0), (-\infty,b]) + \calD(w(\,\cdot\,,0), (-\infty,b]).
\end{align*}
This and \eqref{a22.11} imply that 
\begin{align}\label{a22.12}
\calD(v(\,\cdot\,,t), (-\infty,b])
&= \calD(v(\,\cdot\,,0), (-\infty,b]),\\
\calD(w(\,\cdot\,,t), (-\infty,b])
&=  \calD(w(\,\cdot\,,0), (-\infty,b]).\label{a22.13}
\end{align}
By Lemma \ref{a22.10}, $v(x,t)$ and $w(x,t)$ are strictly increasing
while $v(x,0)$ and $w(x,0)$ are strictly decreasing by assumption. This and \eqref{a22.12}-\eqref{a22.13} imply that $v(x,t) = v(-x,0)$ and $w(x,t) = w(-x,0)$ for $x\in I$. According to Lemma \ref{a22.10}, $\sigma(x,t)=0$ for $x\in I$. Hence $\mu(x,t) =2v(x,t) = 2v(-x,0)=\mu(-x,0)$.
\end{proof}

\subsection{Corners for freezing and thawing curves}

Recall the definitions of freezing and thawing boundaries stated in \eqref{m25.2} and \eqref{a28.1}.

\begin{proposition}\label{f20.1}

Assume  that the initial conditions satisfy  Assumption \ref{m29.2}.

(i)
Suppose that the right endpoints of a piece of the freezing boundary and a piece of the thawing boundary
meet  at a point $D$,
and characteristics  emanating from $v(x_1,0)$ and $w(y_1,0)$ also meet at $D$. Suppose that $v(x,0)$ is locally $C^1$ at $x_1$ and has a strictly positive derivative at $x_1$, while $w(y,0)$ is locally $C^2$ at $y_1$, has zero derivative at $y_1$, and its second derivative is strictly negative.
Then at $D$, the (one-sided) slope of the freezing boundary is 1 and the (one-sided) slope of the thawing boundary is $-3$ (see Fig. \ref{fig1}).

(ii)
Suppose that the left endpoints of a piece of the thawing boundary and and a piece of the freezing boundary
meet  at a point $C$,
and characteristics  emanating from $v(x_1,0)$ and $w(y_1,0)$ also meet at $C$. Suppose that $v(x,0)$ is locally $C^1$ at $x_1$ and has a strictly positive derivative at $x_1$, while $w(y,0)$ is locally $C^2$ at $y_1$, has zero derivative at $y_1$, and its second derivative is strictly negative.
Then at $C$, the (one-sided) slope of the thawing boundary is $-\infty$ and the (one-sided) slope of the thawing boundary is $0$ (see Fig. \ref{fig2}).

\end{proposition}

\begin{center}
\begin{figure}%[h]
\includegraphics[width=0.8\linewidth]{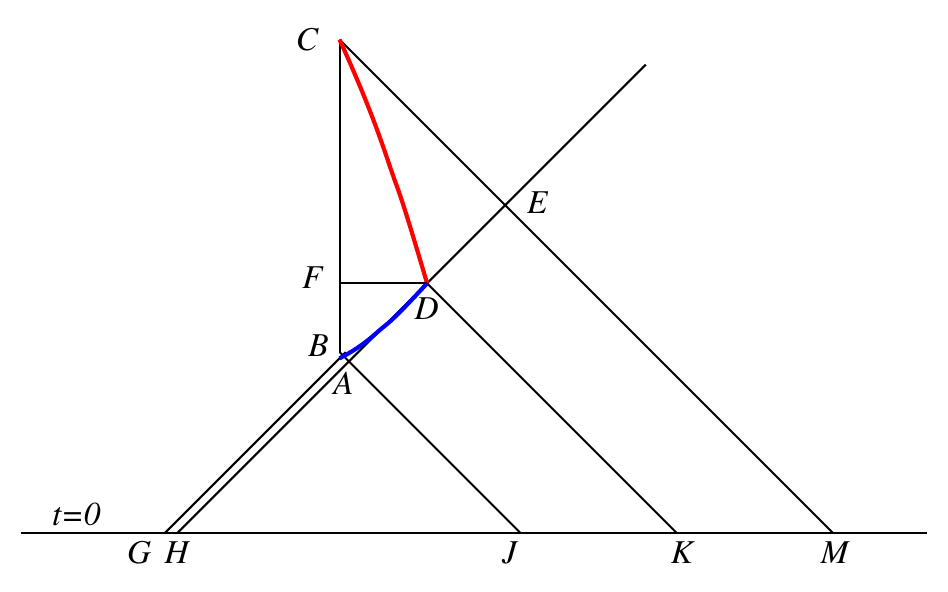}
\caption{A piece of the freezing boundary $BD$ (blue) and a piece of the thawing boundary $DC$ (red) meet at the point $D$.}
\label{fig1}
\end{figure}
\end{center}

\begin{center}
\begin{figure}%[h]
\includegraphics[width=0.8\linewidth]{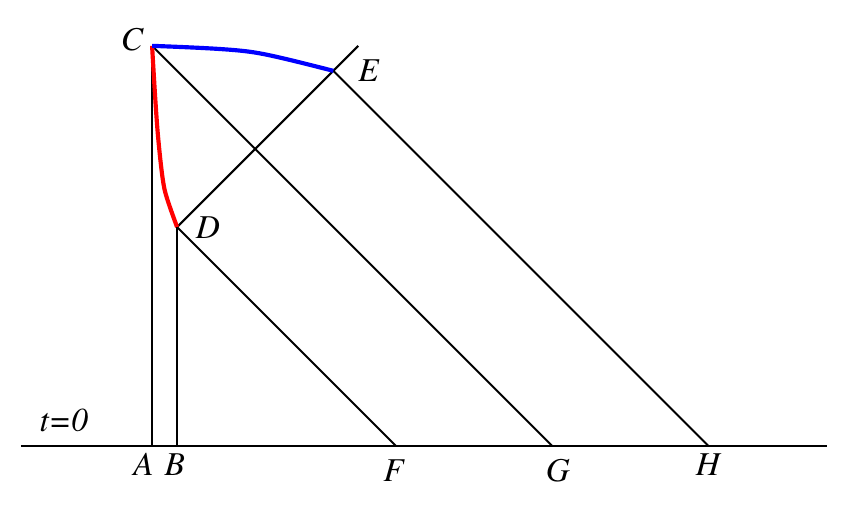}
\caption{A piece of the freezing boundary $CE$ (blue) and a piece of the thawing boundary $DC$ (red) meet at the point $C$.}
\label{fig2}
\end{figure}
\end{center}

\begin{proof}
We will sketch the proof in an informal way.

(i)
Please refer to Fig. \ref{fig1} for notation. Note that $H=(x_1,0)$ and $K=(y_1,0)$.
Suppose that 
$v(H)=v(A)=v(D)=v(E) = w(K)$, and $v(G) = v(B)=v(C) = w(J)=w(M)$.
If $|GH|$ is very small then $|GH|/|JK|$ is very small 
because the derivative of $v(x,0)$ is strictly positive at $x_1$ and 
the derivative of $w(y,0)$ is zero at $y_1$.
Therefore, the characteristic from $H$ to $D$ is tangent to the freezing boundary $BD$ at the point $D$. Hence, the slope of the freezing boundary at $D$ is 1.

Suppose that $|JK|$ is very small. Since $w(x,0)$ is $C^2$, we have $|KM|\approx |JK|$. This implies that $|AD|\approx|DE|$. The triangle $ACE$ is almost identical to the triangle $BCE$, and the triangles are almost isosceles  triangles with the right angle at $E$. Let $a=|AE|$. We have the following. $|AD| \approx a/2$. $|FD|\approx |BF| \approx \sqrt{2}a/4$. $|CB| \approx \sqrt{2} a$. $|CF| = |CB|-|BF| \approx 3\sqrt{2} a/4$.  Since $|CF|/|FD| \approx  
(3\sqrt{2} a/4)/(\sqrt{2} a/4)=3$, it follows that the slope of the thawing boundary at $D$ is $-3$.

The proof of (ii) is based on similar ideas. 
\end{proof}

\subsection{Monotone dependence on initial conditions}

\begin{proposition}\label{sec:monotonedependance}
  Suppose that $ v_{1}(x,0) \le v_{2}(x,0) $ and $ w_{1}(x,0) \le w_{2}(x,0)$
  for all $x\in I$. Then $ v_{1}(x,t) \le v_{2}(x,t) $ and $ w_{1}(x,t) \le w_{2}(x,t)$
  for all $x\in I$ and $t\geq 0$.
\end{proposition}

\begin{proof}
It follows from \eqref{j17.1a} and \eqref{j17.3} that
$\alpha(v_1,b,x) \geq \alpha(v_2,b,x)$ for all $b$ and $x$. This, \eqref{j17.5a} and \eqref{j18.5} imply that $A(v_2, b,t)\subset   A(v_1, b,t)$ for all $b$ and $t$. Hence, \eqref{j17.7a} and Lemma \ref{a22.6} show that $ v_{1}(x,t) \le v_{2}(x,t) $ for all $x\in I$ and $t\geq 0$.
The proof for $w$ is analogous.
\end{proof}

\subsection{Terminal conditions}

The following theorem is closely related to the results  in \cite{KBAO}. It presents a construction of $v$ and $w$ with a given freezing line, one liquid zone and one  frozen zone.

\begin{proposition}
Suppose that  $T: \R\to  [0,\infty)$ satisfies $|T(x)-T(y)|< |x-y|$ for all $x$ and $y$ and $f:\R \to \R$ is continuous and strictly increasing. Then there exist unique initial conditions such that the solution to \eqref{n10.1}-\eqref{n10.3} satisfies $\mu(x,T(x)+t)=\mu(x,T(x)) = f(x)$ and $\sigma(x,T(x)+t) =\sigma(x,T(x)) = 0$ for $x\in \R $ and $t\geq 0$.
\end{proposition}

\begin{proof}
In terms of $v$ and $w$, we are looking for solutions such that for all $x$ and $t\geq 0$,
\begin{align}\label{j23.1}
 v(x,T(x)+t)+w(x,T(x)+t)=v(x,T(x))+w(x,T(x)) &= f(x),\\
 v(x,T(x)+t)-w(x,T(x)+t)=v(x,T(x))-w(x,T(x)) &= 0.\label{j23.2}
\end{align}
Tracing characteristics back in time we see that \eqref{j23.1}-\eqref{j23.2} are satisfied if and only if for all $x\in\R$,
\begin{align*}
v(x,0) &= f(y) \text{  where  }
y = \inf\{z\geq x: T(z) = z-x\},\\
w(x,0) &= f(r) \text{  where  }
r = \sup\{z\leq x: T(z) = x-z\}.
\end{align*}
Then we let $\mu(x,0) = v(x,0)+w(x,0) $ and $\sigma(x,0) = v(x,0)-w(x,0) $
for all $x\in \R$.
\end{proof}

\subsection{Time reversibility}

\begin{remark}\label{a28.2}
  
  Consider $I=\R$ or $I=[-a,a]$.  Suppose that $(v,w)$
  satisfy \eqref{n10.4}-\eqref{a5.4}.  Fix some $T>0$.
  Let $\wt v(x,t) = v(-x,T-t)$ and
  $\wt w(x,t) = w(-x,T-t)$ for all $x\in I $ and
  $t\in[0,T]$.  Then $(\wt v, \wt w)$ also satisfy
  \eqref{n10.4}-\eqref{n10.8} in the region
  $\{(x,t): x\in I, t\in[0,T]\}$.
  However, frozen regions
  for $(\wt v, \wt w)$, if there are any, ``point in the
  wrong direction.'' For example, the triangles in Fig.
  \ref{fig14} and the Christmas tree in Fig. \ref{fig17}
  are upside down.  The reason for the apparent
  contradiction is that $(\wt v, \wt w)$ do not satisfy
  the assumptions of Theorem \ref{m28.5}, specifically,
  $v$ and $w$ must have backward characteristics starting
  at every point in space-time. However, some
  characteristics may not continue in the forward
  direction of time; see Fig. \ref{fig15}. This property
  is not preserved under time reversal, so even if it
  holds for $(v,w)$, it does
  not need to hold for $(\wt v, \wt w)$.
  In these examples left and right are reversed, so,
  we have frozen regions with \(v\) and \(w\)
  decreasing. For such regions, \(v\) and \(w\) must
  immediately thaw, but they don't, so these examples
  violate \ref{a5.4}

\end{remark}

\subsection{Differentiability of solutions}

\begin{remark}\label{f23.3}
   It is not a surprise that for generic smooth
    initial conditions, $v$ and $w$ are non-differentiable
    at the boundary between liquid and frozen regions
    (although one-sided derivatives may exist). The system
    \eqref{n10.4}-\eqref{n10.6} is hyperbolic in the
    liquid and frozen regions. For smooth initial data
    with isolated extrema, the slopes of the thawing and
    freezing curves are non-characteristic (slope of the
    curve is not a characteristic direction---characteristic directions have slope one or zero). So
    we expect singularities (points of
    non-differentiability) in the values of \(v\) and
    \(w\) on the freezing and thawing curves to propagate
    forward in time. However, if the initial data is
    smooth, although the derivatives of \(v\) and \(w\)
    jump across these curves, they remain tangentially
    smooth on both sides, so no new singularities propagate
    from the interior of these curves. New singularities
    do, however, propagate along forward characteristics
    that originate at the corners where two
    thawing curves meet or where a freezing curve meets a
    thawing curve, as illustrated in the figures (from a
    numerically computed example) below.

    \FloatBarrier
  
  \begin{center}
\begin{figure}%[h]
\includegraphics[width=0.8\linewidth]{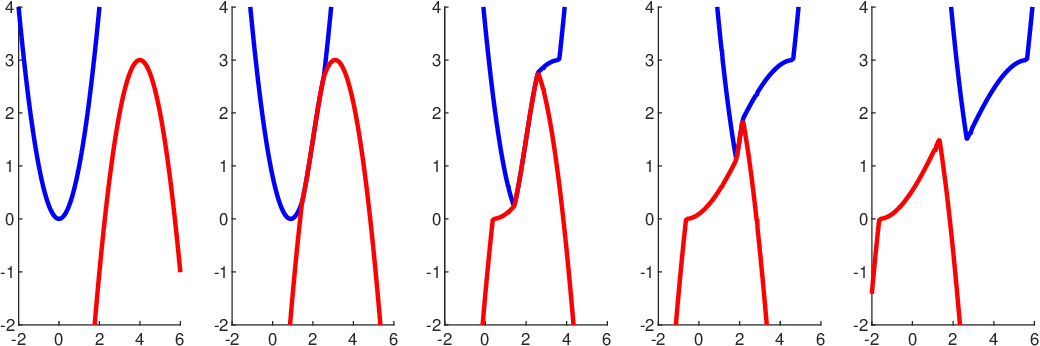}
\caption{Graphs
    of \(v\) and \(w\) at a sequence of times. Initial conditions:
    \(v_{0}(x)=x^{2}\) and \( w_{0}(x)= -(x-4)^{2}+3\).}
\label{fig:23}
\end{figure}
\end{center}

\begin{center}
\begin{figure}%[h]
\includegraphics[width=0.5\linewidth]{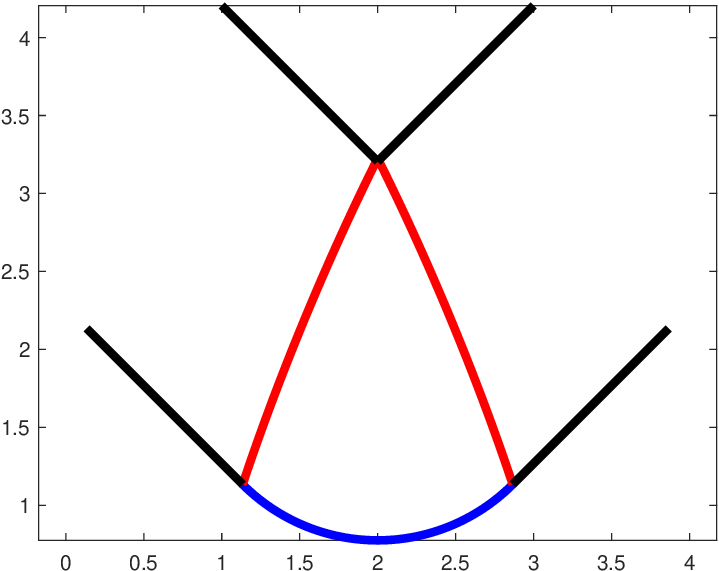}
\caption{Frozen triangle and some
    characteristics emanating from the corners.
     Initial conditions:
    \(v_{0}(x)=x^{2}\) and \( w_{0}(x)= -(x-4)^{2}+3\).}
\label{fig:24}
\end{figure}
\end{center}

  \FloatBarrier  
  
   Suppose that
\begin{align*}
v(x,0) &= x^2,\\
w(x,0) & = -(x-4)^{2}+3.
\end{align*}
Meeting location $z=z(b)$ of characteristics corresponding to value $b=v(x_1,0)=w(y_1,0)\in[0,3]$ is 
\begin{align}\label{j24.1}
z = \frac 1 2 (\sqrt{b}-\sqrt{3-b}+4).
\end{align}
The times of freezing and thawing at value $b$ are given by
\begin{align}\label{j24.2}
S_f(b) &= \frac 1 2 (4-\sqrt{3-b}-\sqrt{b}),\\
S_t(b) &= 
\begin{cases}
\frac 1 2 (3\sqrt{b}-\sqrt{3-b}+4) & b\leq 3/2,\\
\frac 1 2 (-\sqrt{b}+3\sqrt{3-b}+4) & b\geq 3/2.
\end{cases}
\label{j24.3}
\end{align}
We will write the freezing times and thawing times as functions of location. To do that we solve \eqref{j24.1} for $b$ and substitute it into 
\eqref{j24.2}-\eqref{j24.3},
\begin{align*}
b(z)&= \frac{1}{2} \left(3\pm 2 \sqrt{2} \sqrt{-2 z^4+16 z^3-45 z^2+52 z-20}\right),
\qquad z > 2, z < 2,
\\
T_f(z)=S_f(b(z))&= \frac{1}{4} \Big(-\sqrt{4 \sqrt{2} \sqrt{-(z-2)^2 \left(2 z^2-8 z+5\right)}+6}\\
&\qquad -\sqrt{6-4 \sqrt{2} \sqrt{-(z-2)^2 \left(2 z^2-8 z+5\right)}}+8\Big),\\
T_t(z)=S_t(b(z))&=
\frac{1}{4} \Big(-\sqrt{4 \sqrt{2} \sqrt{-(z-2)^2 \left(2 z^2-8 z+5\right)}+6}\\
&\qquad +3 \sqrt{6-4 \sqrt{2} \sqrt{-(z-2)^2 \left(2 z^2-8 z+5\right)}}+8\Big).
\end{align*}
The derivatives of $T_f$ and $T_t$ at the corners are as follows.
The right derivatives of $T_f$ and $T_t$ at $z=(4 - \sqrt{3})/2$ are $-1$ and $3$.
The left derivatives of $T_f$ and $T_t$ at $z=(4 + \sqrt{3})/2$ are 1 and $-3$.
Sea Fig. \ref{fig:24} for graphs of $T_f$ and $T_t$.

  Fig. \ref{fig:23} shows the graphs
    of \(v\) and \(w\) at a sequence of times. Fig.
    \ref{fig:24} shows the frozen triangle and some
    characteristics emanating from the corners. The
    derivatives of both \(v\) and \(w\) jump across the
    boundary of the (black) frozen triangle. In addition, the
    derivative of \(v\) jumps across the blue 
    characteristics emanating from the two corners on the
    right side of the triangle and the derivative of \(w\)
    jumps across the red characteristics emanating from
    the two corners on the left side. Note that there is
    no forward \(w\) characteristic emanating from the
    lower right corner, or from any point on the thawing
    boundary that forms the right side of the
    triangle. To within the computational
    accuracy, the one-sided slopes at the freeze-thaw
    corners of the triangle are as described in Proposition \ref{f20.1}.

   Recall from Fig. \ref{fig14} the shape of a
    frozen triangle.  Suppose that the initial conditions
    $v(x,0)$ and $w(x,0)$ are $C^2$.  In the generic case,
    the derivatives of $v$ and $w$ are discontinuous along
    the characteristics emanating from the tip (top
    vertex) of the triangle, although $v$ and $w$ are
    continuous everywhere.  Suppose that the
    characteristic of $v$ emanating from $x_1$ and the
    characteristic of $w$ emanating from $y_1$ meet at the
    tip of the frozen triangle at time $t_1$. The values
    of $v$ to the left of the characteristic for times
    greater than $t_1$ are those to the left of $x_1$ at
    time 0.  The values of $v$ to the right of the
    characteristic for times greater than $t_1$ are those
    of $w$ to the left of $y_1$ at time 0. While we have
    $v(x_1,0)=w(y_1,0)$ by assumption, the derivatives of
    $v$ at $x_1$ and $w$ at $y_1$ do not have to be
    related in any way so, generically, the one-sided
    derivatives of $v$ along the characteristic passing
    through the tip, past $t_1$, will be different. The
    same remark applies to $w$.

    See the explicit formulas in the slightly different  Example \ref{a28.3}. 
    Example \ref{a2.1} (ii) and the accompanying Fig. \ref{fig:31} present some of the above ideas in a different way.
    
\end{remark}

\section{Examples}\label{sec:exam}

\begin{example}

Fig. \ref{fig:24} illustrates a transition from freezing to thawing at two corners of a frozen triangle. This example will illustrate transition from thawing to freezing. Let
\begin{align*}
v(x,0) &= x+2,\\
w(x,0) & = 
\begin{cases}
x+2 & x< -1,\\
x^2 & -1\leq x \leq 1,\\
x& x> 1.
\end{cases}
\end{align*}
Meeting locations $z=z(b)$ of characteristics corresponding to value $b=v(x_1,0)=w(y_1,0)\in[0,1]$ are
\begin{align}\label{j24.5}
z_1& = b-2,\\
z_2&= \frac{1}{2} \left(\sqrt{b}+b-\left(-\sqrt{b}-b+2\right)-2\right).
\label{j24.6}
\end{align}
The times of thawing and freezing  at value $b$ are given by
\begin{align}\label{j24.7}
S_t(b) &= -\sqrt{b} -b+2,\\
S_f(b) &= s_f(b)+ \sqrt{b} = 2-b.\label{j24.8}
\end{align}
We will write the freezing times and thawing times as functions of location. To do that we solve \eqref{j24.5}-\eqref{j24.6} for $b$ and substitute it into 
\eqref{j24.7}-\eqref{j24.8},
\begin{align*}
b_1(z_1)&=2+z_1,\\
b_2(z_2) &= \frac{1}{2} \left(2 z_2-\sqrt{4 z_2+9}+5\right),\\
T_t(z)&=S_f(b_1(z))= -z-\sqrt{z+2},\\
T_f(z)&=S_t(b_2(z))=\frac{1}{2} \left(-2 z+\sqrt{4 z+9}-1\right).
\end{align*}
We have
\begin{align*}
T_t'(z)&= -\frac{1}{2 \sqrt{x+2}}-1,\\
T_f'(z)&=
\frac{1}{2} \left(\frac{2}{\sqrt{4 x+9}}-2\right).
\end{align*}
The right derivatives of $T_t$ and $T_f$  at $z=-2$ are $-\infty$ and $0$.
Sea Fig. \ref{fig27} for graphs of $T_f$ and $T_t$.

\begin{center}
\begin{figure}%[h]
\includegraphics[width=0.3\linewidth]{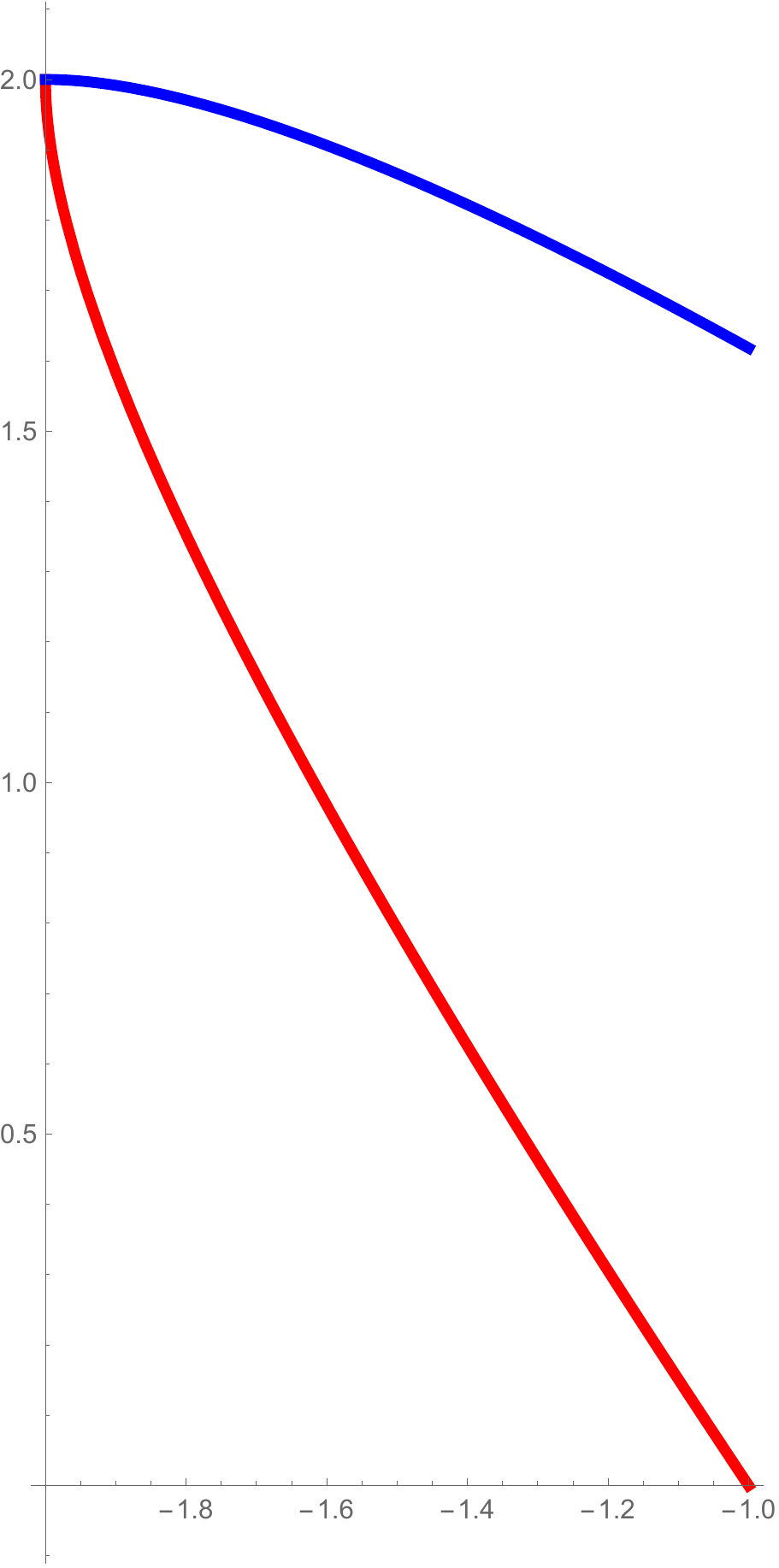}
\caption{Freezing (blue) and thawing (red) boundaries.}
\label{fig27}
\end{figure}
\end{center}

\end{example}

\begin{example}\label{a28.3}
We include another  example illustrating jumps of derivatives of
$v$ and $w$ that admits very simple formulas. Let
 \begin{align*}
      v(x,0) = |x|,
      \qquad
      w(x,0) = \frac{x}{2}.
    \end{align*}
    Then
    \begin{align*}
      v(x,t) &=
      \begin{cases}
        -(x-t) & x<\frac{3t}{5},
        \\
        \frac{2x}{3} &  \frac{3t}{5} < x < 3t,
        \\
        x-t & 3t < x,
      \end{cases}
      \\
      w(x,t) &=
               \begin{cases}
                 \frac{x+t}{2} & x < -t,
                 \\
                 \frac{x+t}{4} & -t < x < \frac{3t}{5},
                 \\
                 \frac{2x}{3} &  \frac{3t}{5} < x < 3t,
                 \\
                 \frac{x+t}{2} & 3t < x.
               \end{cases}
    \end{align*}

    Derivatives of \(v\) jump across
    \begin{enumerate}
       \item Freezing boundary \(x=3t\),
      \item Thawing boundary \(x = \frac{3t}{5}\).
      \end{enumerate}
     Derivatives of \(w\) jump across
     \begin{enumerate}
      \item Freezing boundary \(x=3t\),
      \item Thawing boundary \(x = \frac{3t}{5}\),
      \item The $w$-characteristic propagating from
        \((0,0)\), the point where the derivative of
        \(v(x,0)\) jumps. Note that the jump in the
        derivative of \(v\) at \((0,0)\) does not
        propagate along a $v$-characterisitic, as the point
        \((0,0)\) is part of a thawing boundary which has
        no forward \(v\) characteristics.
     \end{enumerate}   
 \FloatBarrier  
\begin{center}
\begin{figure}[h]
\includegraphics[width=0.9\linewidth]{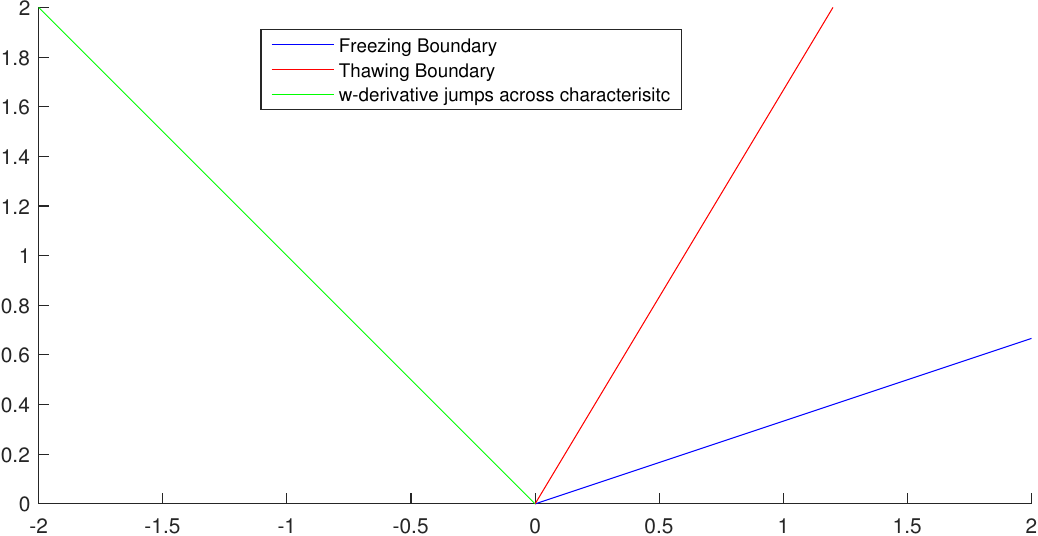}
\caption{Freezing (blue) and thawing (red) boundaries for
   Example \ref{a28.3}.}
\label{fige12}
\end{figure}
\end{center}
%     \FloatBarrier
\end{example}

\begin{example}\label{a2.1}
What shapes can the frozen region  take? While we are far from being able to give a full answer to this question, we will point out some restrictions on possible shapes under Assumption \ref{m29.2}. Some restrictions on the slopes of freezing and thawing curves were stated in \eqref{m22.5} and \eqref{m22.6}. Below we present two extra restrictions.

(i) First, the liquid region cannot contain a
characteristic rectangle $R$ which
 has edges with slopes 1 and $-1$,
is contained in the liquid region except for its top three vertices,
 its top vertex belongs to the relative interior of the freezing boundary, the left and right vertices belong to relative interiors of the thawing curves, and the bottom vertex belongs to the interior of the liquid region (see Fig. \ref{fig29}). 
\begin{center}
\begin{figure}%[h]
\includegraphics[width=0.7\linewidth]{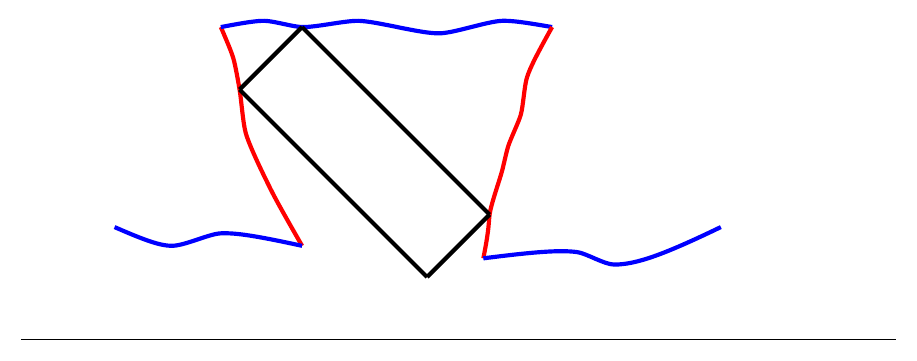}
\caption{The liquid region cannot contain a rectangle with three top vertices on its boundary (see Example \ref{a2.1} (i) for details).}
\label{fig29}
\end{figure}
\end{center}
The 
reason why this is impossible is that
the sides of $R$ with slopes 1 are characteristics of $v$, the sides with slopes $-1$ are characteristics of $w$, and the values of $v$ and $w$ on these characteristics are equal because they have to be equal at the three top vertices. Hence
 $v$ must be equal to $w$ at the bottom vertex, which contradicts the assumption that this vertex belongs to the interior of the liquid region.

(ii) The second restriction is related to characteristics emanating from the top vertex $(y,s)$ of shapes such as frozen triangles or Christmas tree in Figs. \ref{fig14}-\ref{fig17}. For $t>s$, in a neighborhood of $s$, $x\to v(x,s)$ is decreasing to the left of $y+(t-s)$ and increasing to the right of $y+(t-s)$. Hence, the characteristic $\{(y+(t-s), t): t\geq s\}$ of $v$ must end at a meeting point of a freezing curve and thawing curve. A similar remark applies to the characteristic of $w$ moving in the opposite direction. 
The following remarks pertain to Fig. \ref{fig:31}.
  \begin{enumerate}
  \item All points on thawing curves with positive slopes are
    minima of \(v\) and all points on thawing curves with negative
    slopes are maxima of \(w\).
    
  \item There can be no \(v\) or \(w\) extrema in the
    frozen region or the interior of the freezing
    boundary, and extrema in the liquid region
    propagate along characteristics.
    
  \item Thus minima of \(v\) propagate along \(v\) characteristics
    until they reach a freeze/thaw corner, whence they continue
    as thawing boundaries.
    We say
      ``freeze/thaw'' to indicate that the freezing
      boundary is below the (positive-sloped) thawing
      boundary, as pictured on the bottom right corner of
      Fig. \ref{fig29}. At a ``freeze/thaw'' corner, a
      characteristic freezes, then thaws. A thaw/freeze
      corner is one where the thawing curve is below the
      freezing curve, so that a frozen characteristic
      thaws, then refreezes.

  \item Maxima of \(v\) propagate only along
    characteristics, and only in the liquid region. Each
    such characteristic continues until it meets a
    \(v\)-minimum at a thaw/freeze corner, where they
    annihilate each other.
    
  \item A similar description applies to
    \(w\)-extrema. 
  \item An important corollary is that the number of
    extrema of each of \(v\) and \(w\) can only decrease
    with time. In the figure, each of \(v\) and \(w\)
    start with one mimimum and one maximum point. After
    some time, no extremal points remain.
  \end{enumerate}

  \begin{center}
\begin{figure}[h]
\includegraphics[width=0.7\linewidth]{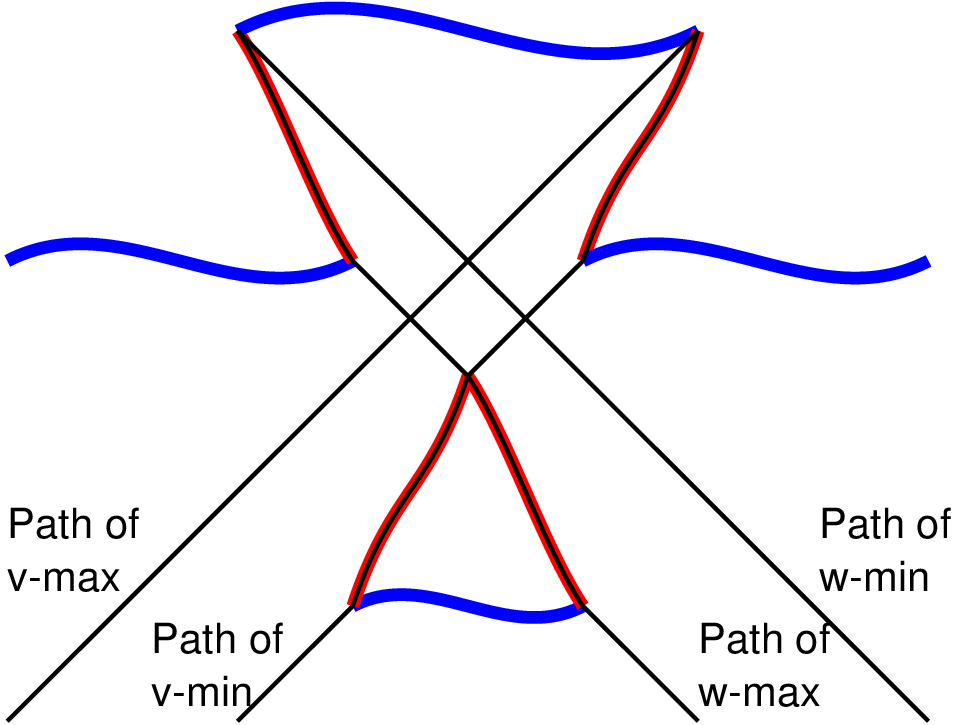}
\caption{}
\label{fig:31}
\end{figure}
\end{center}

\end{example}

\section{Acknowledgments}

We are grateful to Pablo Ferrari, Jean-Francois Le Gall, Jeremy Hoskins, Adam Ostaszewski and Stefan Steinerberger for the most useful advice.

\newcommand{\etalchar}[1]{$^{#1}$}


\begin{thebibliography}{BDSG{\etalchar{+}}15}

\bibitem[AAV11]{AAV}
Gideon Amir, Omer Angel, and Benedek Valk\'{o}.
\newblock The {TASEP} speed process.
\newblock {\em Ann. Probab.}, 39(4):1205--1242, 2011.

\bibitem[ABD21]{fold}
Jayadev~S. Athreya, Krzysztof Burdzy, and Mauricio Duarte.
\newblock On pinned billiard balls and foldings.
\newblock {\em Indiana U. Math. J.}, 2021.
\newblock To appear, Arxiv:1807.08320.

\bibitem[AHR09]{AHR}
Omer Angel, Alexander Holroyd, and Dan Romik.
\newblock The oriented swap process.
\newblock {\em Ann. Probab.}, 37(5):1970--1998, 2009.

\bibitem[AHRV07]{AHRV}
Omer Angel, Alexander~E. Holroyd, Dan Romik, and B\'{a}lint Vir\'{a}g.
\newblock Random sorting networks.
\newblock {\em Adv. Math.}, 215(2):839--868, 2007.

\bibitem[BDS83]{BDS}
C.~Boldrighini, R.~L. Dobrushin, and Yu.~M. Sukhov.
\newblock One-dimensional hard rod caricature of hydrodynamics.
\newblock {\em J. Statist. Phys.}, 31(3):577--616, 1983.

\bibitem[BDSG{\etalchar{+}}15]{Bert}
Lorenzo Bertini, Alberto De~Sole, Davide Gabrielli, Giovanni Jona-Lasinio, and Claudio Landim.
\newblock Macroscopic fluctuation theory.
\newblock {\em Rev. Modern Phys.}, 87(2):593--636, 2015.

\bibitem[BHS22]{KBJHSS}
Krzysztof Burdzy, Jeremy~G. Hoskins, and Stefan Steinerberger.
\newblock From pinned billiard balls to partial differential equations, 2022.
\newblock Preprint. ArXiv 2209.01503.

\bibitem[BO23]{KBAO}
Krzysztof Burdzy and Adam~J. Ostaszewski.
\newblock Freezing in space-time: a functional equation linked with a {PDE} system.
\newblock {\em J. Math. Anal. Appl.}, 524(2):Paper No. 127018, 11, 2023.

\bibitem[DF77]{DF}
R.~L. Dobrushin and J.~Fritz.
\newblock Non-equilibrium dynamics of one-dimensional infinite particle systems with a hard-core interaction.
\newblock {\em Comm. Math. Phys.}, 55(3):275--292, 1977.

\bibitem[FFGS23]{FerEt}
Pablo~A. Ferrari, Chiara Franceschini, Dante G.~E. Grevino, and Herbert Spohn.
\newblock Hard rod hydrodynamics and the L\'evy Chentsov field.
\newblock {\em Ensaios Matem\'aticos}, 38:185--222, 2023.

\bibitem[FO23]{FerOl}
Pablo Ferrari and Stefano Olla.
\newblock Macroscopic diffusive fluctuations for generalized hard rods dynamics, 2023.
\newblock Preprint. ArXiv 2305.13037.

\bibitem[GG08a]{Gaspard_2}
P.~Gaspard and T.~Gilbert.
\newblock Heat conduction and {F}ourier's law by consecutive local mixing and thermalization.
\newblock {\em Phys. Rev. Lett.}, 101:020601, Jul 2008.

\bibitem[GG08b]{Gaspard_2008}
Pierre Gaspard and Thomas Gilbert.
\newblock Heat conduction and {F}ourier's law in a class of many particle dispersing billiards.
\newblock {\em New Journal of Physics}, 10(10):103004, oct 2008.

\bibitem[GG08c]{Gaspard_2008_1}
Pierre Gaspard and Thomas Gilbert.
\newblock On the derivation of {F}ourier's law in stochastic energy exchange systems.
\newblock {\em Journal of Statistical Mechanics: Theory and Experiment}, 2008(11):P11021, nov 2008.

\bibitem[KS91]{KS}
Ioannis Karatzas and Steven~E. Shreve.
\newblock {\em Brownian motion and stochastic calculus}, volume 113 of {\em Graduate Texts in Mathematics}.
\newblock Springer-Verlag, New York, second edition, 1991.

\bibitem[Lax84]{Lax1984}
Peter~D. Lax.
\newblock Shock waves, increase of entropy and loss of information.
\newblock In {\em Seminar on nonlinear partial differential equations ({B}erkeley, {C}alif., 1983)}, volume~2 of {\em Math. Sci. Res. Inst. Publ.}, pages 129--171. Springer, New York, 1984.

\bibitem[Ser99]{Serre}
Denis Serre.
\newblock {\em Systems of conservation laws. 1}.
\newblock Cambridge University Press, Cambridge, 1999.
\newblock Hyperbolicity, entropies, shock waves, Translated from the 1996 French original by I. N. Sneddon.

\bibitem[Smo94]{Smoller}
Joel Smoller.
\newblock {\em Shock waves and reaction-diffusion equations}, volume 258 of {\em Grundlehren der mathematischen Wissenschaften [Fundamental Principles of Mathematical Sciences]}.
\newblock Springer-Verlag, New York, second edition, 1994.

\end{thebibliography}
\end{document}